\documentclass{article}

\usepackage{geometry}
\usepackage[utf8]{inputenc} % allow utf-8 input
\usepackage[T1]{fontenc}    % use 8-bit T1 fonts
\usepackage{hyperref}       % hyperlinks
\usepackage{url}            % simple URL typesetting
\usepackage{booktabs}       % professional-quality tables
\usepackage{amsfonts}       % blackboard math symbols
\usepackage{nicefrac}       % compact symbols for 1/2, etc.
\usepackage{microtype}      % microtypography
\usepackage{lipsum}
\usepackage{graphicx}
\graphicspath{ {./images/} }
\usepackage{subcaption}
\usepackage{amsmath}
\usepackage{amsthm}
\newtheorem{theorem}{Theorem}[section]
\newtheorem{lemma}[theorem]{Lemma}
\newtheorem{proposition}[theorem]{Proposition}

\title{Dynamical analysis of r-Chialvo neuron map with cosine memristive}

\author{
 Ajay Kumar \thanks{p22ma004@iitj.ac.in} \quad  VVMS Chandramouli \thanks{Corresponding author: chsarma@iitj.ac.in} \\
  Department of Mathematics\\
  Indian Institute of Technology Jodhpur\\
  Rajasthan, India, 342030. 
}

\begin{document}
\maketitle
\begin{abstract}
In this work, we construct a novel two-dimensional discrete neuron map by incorporating a cosine-based memristor into the reduced Chialvo neuron map to examine the dynamical analysis of electromagnetic modulation. The nonlinear current-voltage characteristics of the memristor enrich the neuron map's behavior, leading to diverse firing regimes, stability behaviors, and chaotic attractors. This study begins to establish the equilibrium points using both analytical and numerical methods. Additionally, we determine the conditions on parameters under which the proposed map exhibits a Neimark-Sacker bifurcation. Further, the numerical study reveals the antimonotonicity structure through the forward and backward bifurcation diagrams. The model exhibits a wide range of codimension-one and codimension-two bifurcation patterns, including Neimark-Sacker, period-doubling, saddle-node, generalized period-doubling, cusp-point, fold-flip, and various resonance structures (1:1, 1:2, 1:3, and 1:4). We also observe that the coexistence of multistable attractors including a stable limit cycle, a period-five attractor, and a chaotic attractor, along with their respective basins of attraction. Furthermore, we extend this analysis to the network of neurons under the ring-star configuration and discuss several spatiotemporal patterns. This network investigation reveals complex collective patterns, including imperfect synchronization, clustered patterns, and multi-chimera state phenomena, which have not been previously observed in existing Chialvo-based studies. These results highlight the potential of the discrete memristor-based neuron map for advancing theoretical neurodynamics and offer a robust framework for investigating low-dimensional yet dynamically rich neuron systems.
\end{abstract}

\noindent\textbf{Keywords:} Neuron map, Memristor, Neimark-Sacker bifurcation, Multistability, Ring-star network, Multi-chimera state, Imperfect synchronization.
% keywords can be removed
% \keywords{Neuron map, Memristor, Neimark-Sacker bifurcation, Multistability, Ring-star network, Multi-chimera state, Imperfect synchronization.}

\section{Introduction}
The dynamical study of neuron models under the influence of electromagnetic radiation plays a crucial role in understanding the nonlinear mechanisms that govern the neuronal dynamics. In this field, researchers typically analyze the dynamics of neurons using mathematical equations, such as ordinary differential equations, partial differential equations, and difference equations. Particularly, nonlinear equations are widely employed to model the neuron's firing patterns \cite{izhikevich2000neural}. The neurodynamics research field began with the introduction of the HH neuron model in 1952 \cite{hodgkin1952quantitative}, and subsequently, biological modeling has undergone significant evolution. This leads to the development of simplified neuron models, including the Fitzhugh-Nagumo model \cite{fitzhugh1955mathematical}, the Morris-Lecar \cite{morris1981voltage}, the Hindmarsh-Rose model \cite{hindmarsh1984model}, the Izhikevich model \cite{izhikevich2007dynamical}, and the Chialvo neuron model \cite{chialvo1995generic}. These advancements have greatly enriched theoretical research in the field of neurodynamics.

In a neuron model, incorporating flexible electronic elements such as memristors enhances the neuron's functionality. The memristor is recognized as the fourth fundamental electronic component, capturing the relationship between electric charge and magnetic flux \cite{chua2003memristor}. Its resistance varies based on the applied voltage (V) and current (I), i.e., it can exhibit nonlinear current-voltage (I–V) characteristics, enabling it to implement threshold-like behavior, which is essential for firing spikes in a neuron model. So, designing the memristor and constructing the corresponding mathematical model plays a crucial role in theoretical analysis, numerical simulation, and experimental investigation of these devices \cite{cao2024discrete,wang2025spatiotemporal}. Among the various types of memristors, the discrete ones are replicable, stable, and controllable. Therefore, they offer a significant advantage for implementation in discrete systems or digital circuits. In \cite{bao2021discrete}, it is proposed that the four types of discrete memristors (DMs) derived from the general discrete memristor: Absolute value DM, Quadratic DM, Exponential DM, and Sinusoidal DM. The designing and application of various forms of DMs have emerged as a growing trend, as studied in detail \cite{peng2020discrete,yu2022dynamic,chen2023new,xu2024dynamical,wang2025complex}. 

Neurons often exhibit complex kinetic behaviors and diverse firing patterns during the transmission of information and signal processing. These patterns can be observed by studying the dynamical behavior of the neuron models. As shown in \cite{xiu2020new,yu2021dynamics}, the discrete neuron models are capable of doing faster computations compared to the continuous neuron models. They exhibit more intricate and rich dynamical behaviors, such as hyperchaos, which can emerge at lower dimensions in discrete systems compared to continuous systems \cite{liu2022memcapacitor}. In particular, discrete models not only capture richer dynamics but also exhibit enhanced sensitivity to external influences. In this context, recent advancements in the discrete Chialvo neuron model have revealed a significant enhancement in its dynamical properties under the influence of electromagnetic stimuli \cite{muni2022dynamical,kumar2026study}. Particularly, these models showcased enhanced spiking patterns and chaotic behavior in response to external inputs. In \cite{muni2022dynamical}, a three-dimensional Chialvo neuron map with an external stimulus was initially analyzed, which demonstrated different dynamical properties such as multistability, various routes to chaos, and the emergence of chimera states and clustered states in a network. Recently in \cite{kumar2026study}, a two-dimensional reduced version of the Chialvo neuron map with external stimuli was proposed. This version preserved external stimulation and revealed further advancements, such as an increased number of stable fixed points, symmetric behavior in bifurcation, a higher number of clustered states, and proper continuous traveling wave patterns. However, both studies have some limitations, such as the higher-dimensional model is computationally intensive and structurally complex. In contrast, the two-dimensional (flux-coupled reduced Chialvo) map is more trackable and has made some advancements, but on the other hand, it loses some properties, such as exhibiting only bistability behavior rather than the richer multistability behavior. Additionally, both studies primarily focus on basic spatiotemporal patterns, such as identifying chimera states in network topologies, and not obtained more complex dynamical states, such as multi-chimera states. Multi-chimera states are particularly important as they provide a framework to describe the coexistence of multiple coherent and incoherent activity patterns observed in a neuron network, such as the localized synchronization. More detailed investigations can be found in 
\cite{wang2024multi}. Therefore, the existing extensions either increase dimensionality at the cost of analytical traceability or compromise important dynamical features. This creates a gap that motivates the development of lower-dimensional yet dynamically richer neuron maps.

The objective of the present work is to design a discrete memristor model based on a cosine function and couple it within the reduced Chialvo neuron model to obtain a lower-dimensional model, which exhibits rich dynamics compare to the previous studies on the Chialvo model with flux \cite{muni2022dynamical,kumar2026study}. The proposed formulation bridges the gap between the model simplicity and dynamical complexity, enabling a deeper investigation into the nonlinear phenomena within neuron-inspired systems. It exhibits several dynamical properties, including various stability behaviors of the fixed points, Neimark-Sacker bifurcation, multistability in the system, spiking and bursting behavior of the map, and a chaotic attractor. This study shows several improved phenomena in ring-star networks, including multi-chimera states, which are not commonly seen in recent literature, more clustered states, and an imperfectly synchronized state within the topology. These contributions enhance our understanding in the field of neurodynamics and establish the proposed model as an effective framework for studying such systems.

This paper is organized into several sections as follows. Section 2 presents the construction of the discrete memristor model and its incorporation into a reduced Chialvo map, along with an examination of its basic properties. In Section 3, we discuss analytically the unique positive fixed points and their stability. Further, it numerically explores the behavior of fixed points and their stability. Section 4 explores the theoretical aspects of Neimark-Sacker bifurcation, including numerical forward and backward bifurcation, as well as bifurcation in the parameter space. Section 5 presents the map's multistability behavior, and the corresponding basin of attraction region. Section 6 highlights the model's spiking and bursting behavior. Section 7 demonstrates the chaotic attractors and the calculation of the corresponding fractional dimension. Section 8 explores the dynamics of a neuron's network and illustrates various rich dynamical states in network topologies. Finally, Section 9 summarizes the key findings and results. All the computational results presented in this paper were obtained using MATLAB (R2024a) version.

\section{A discrete memristor model} 

\subsection{Memristor formation} 

Following the definition of a generic magnetically controlled memristor, the mathematical formulation for the discrete memristor (DM) is expressed as:

\begin{equation}\label{1}
\begin{aligned}
i_n &= W(\phi_n)\, v_n = \cos(\pi \phi_n)\, v_n, \\
\phi_{n+1} &= h k_1 v_n - k_2 \phi_n .
\end{aligned}
\end{equation}

\noindent where the symbols $v_n$ and $i_n$ represent the input voltage and output current across the terminal, respectively. The function $\phi_n$ represents the change in flux, and the function $W(.)$ is defined as a \textit{cosine} function. The parameter $h$ represents a material constant associated with the property of the memristor medium. For the analysis of the memristor model's characteristics, the initial flux is set to zero (i.e., $\phi_0=0$) and the parameter $h=1$ is considered fixed throughout this paper unless specified. Additionally, the parameters $k_1$ and $k_2$ represent the memristor parameters. The memristor model closely approximates the mathematical representation of electromagnetic waves, enabling it to effectively simulate neuron response to externally induced electromagnetic fields. A detailed structural diagram of the model, as described in Eq. \eqref{1}, is illustrated in Fig. \ref{memristoronly}. 

\begin{figure}[ht!] 
    \centering
    \subfloat[Memristor structure]{%
        \includegraphics[width=0.8\textwidth]{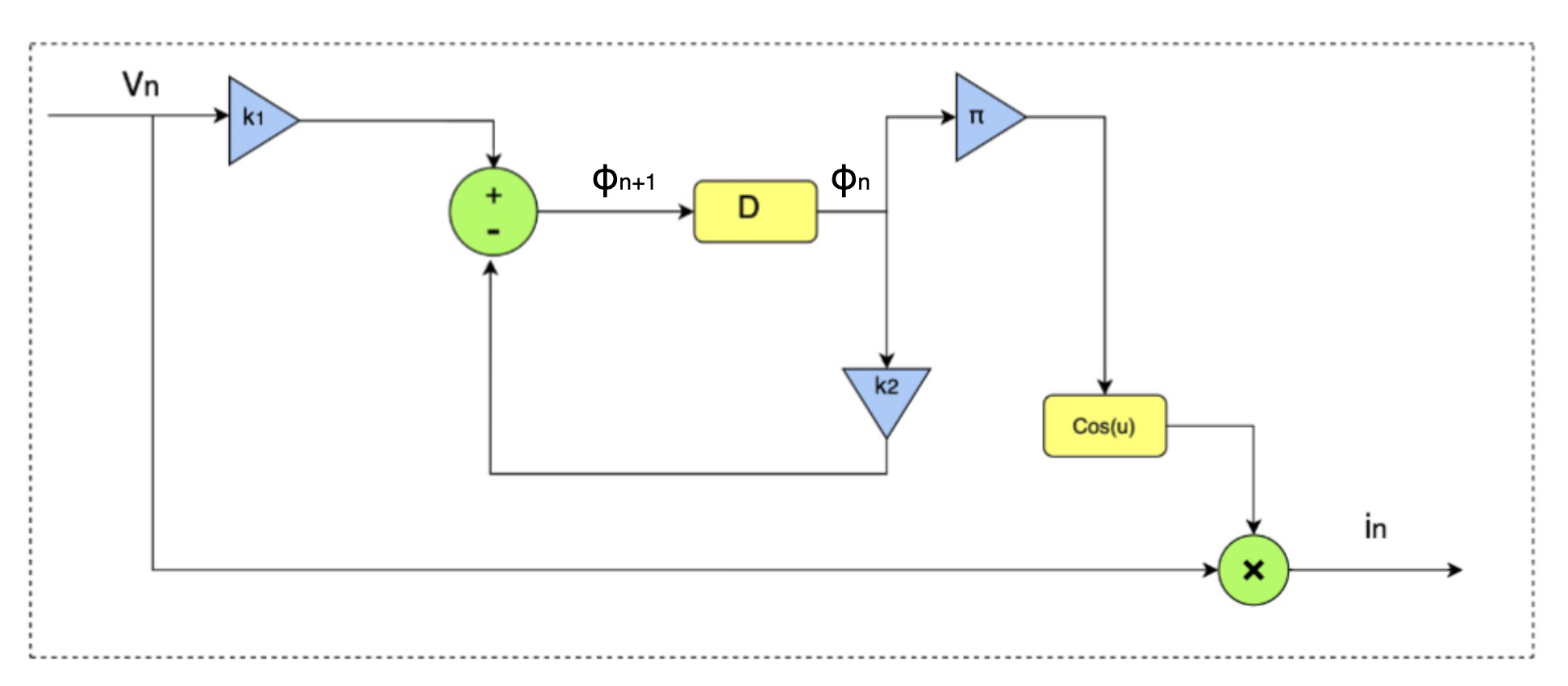}%
        }%
        \vfill
        \centering
    \subfloat[Symbolic form of memristor]{%
        \includegraphics[width=0.6\textwidth]{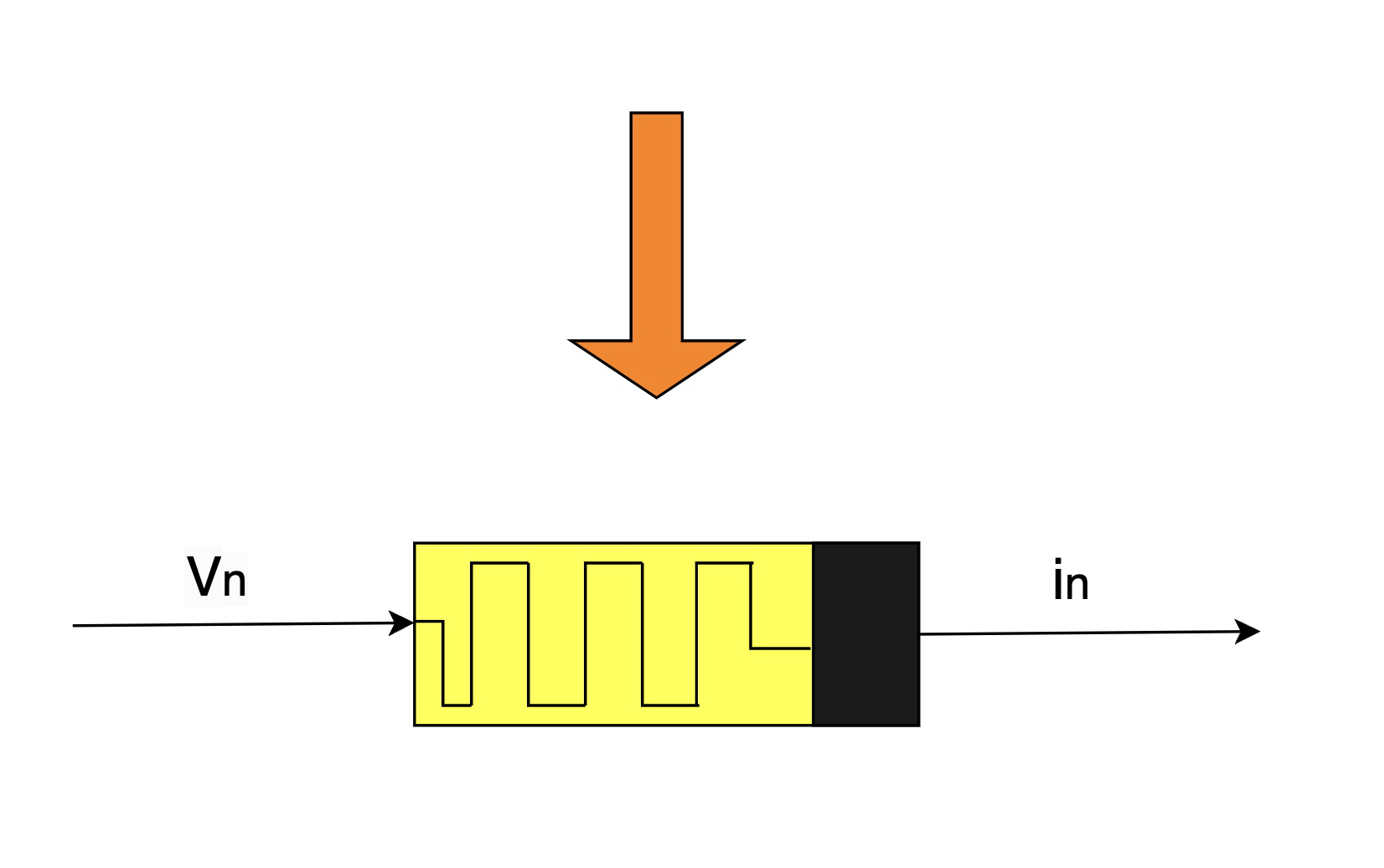}%
        }% 
        \caption{Structure diagram for the memristor model.}\label{memristoronly}
\end{figure}

After introducing the input sinusoidal voltage signal $v_n=v_m \sin(\omega {n})$ into the memristor, the resulting sequence diagram is presented in Fig. \ref{seqence}. In this figure, the red (thin) curve corresponds to the applied input voltage $v_n$ to the memristor, while the blue (thick) curve represents the output current $i_n$ generated by the memristor. This voltage-current relation is also symbolized in Fig. \ref{seqence}. Furthermore, the corresponding Pinched Hysteresis Loops (PHLs) produced by the memristor model, as obtained from the numerical simulation, are shown in Fig. \ref{PHL}, highlighting the nonlinear memory-dependent behavior of the system.

\begin{figure}[ht!] 
    \centering
    \subfloat[$v_m=1, \;\omega=3.7$]{%
        \includegraphics[width=0.55\textwidth]{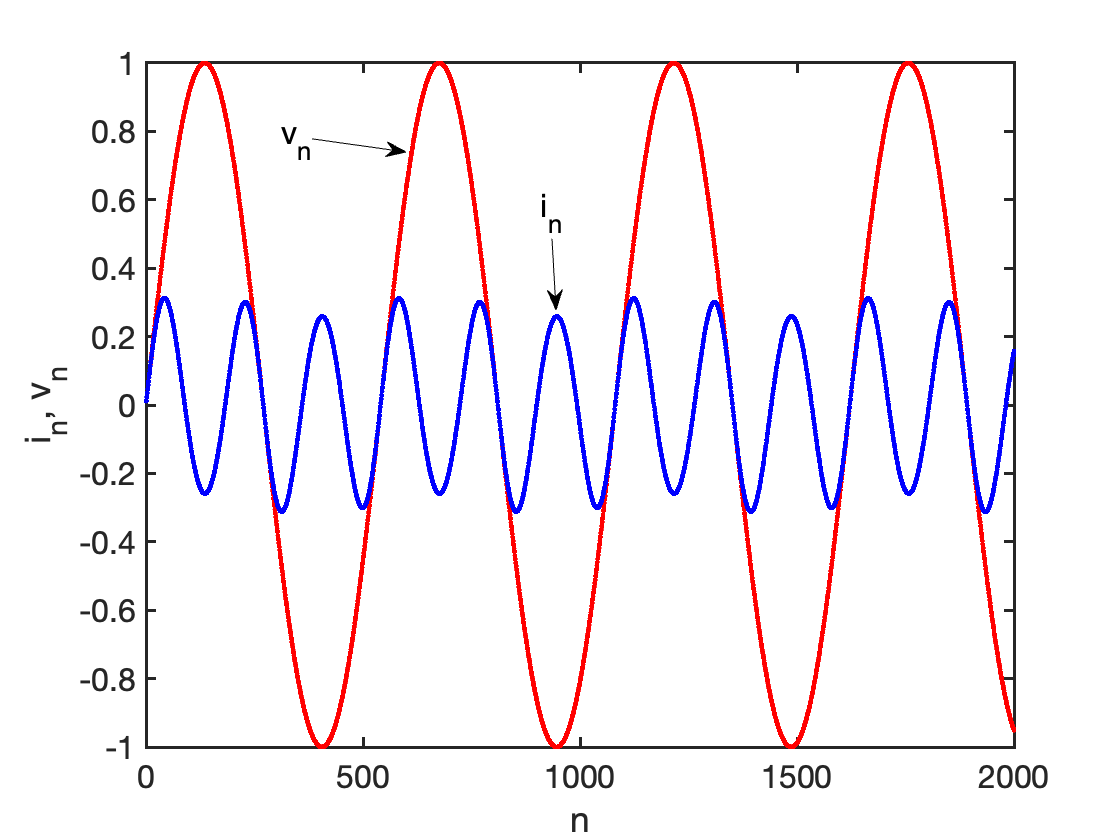}%
        }%
    %\hfill%
    \subfloat[$v_m=1.3, \;\omega=3.4$]{%
        \includegraphics[width=0.55\textwidth]{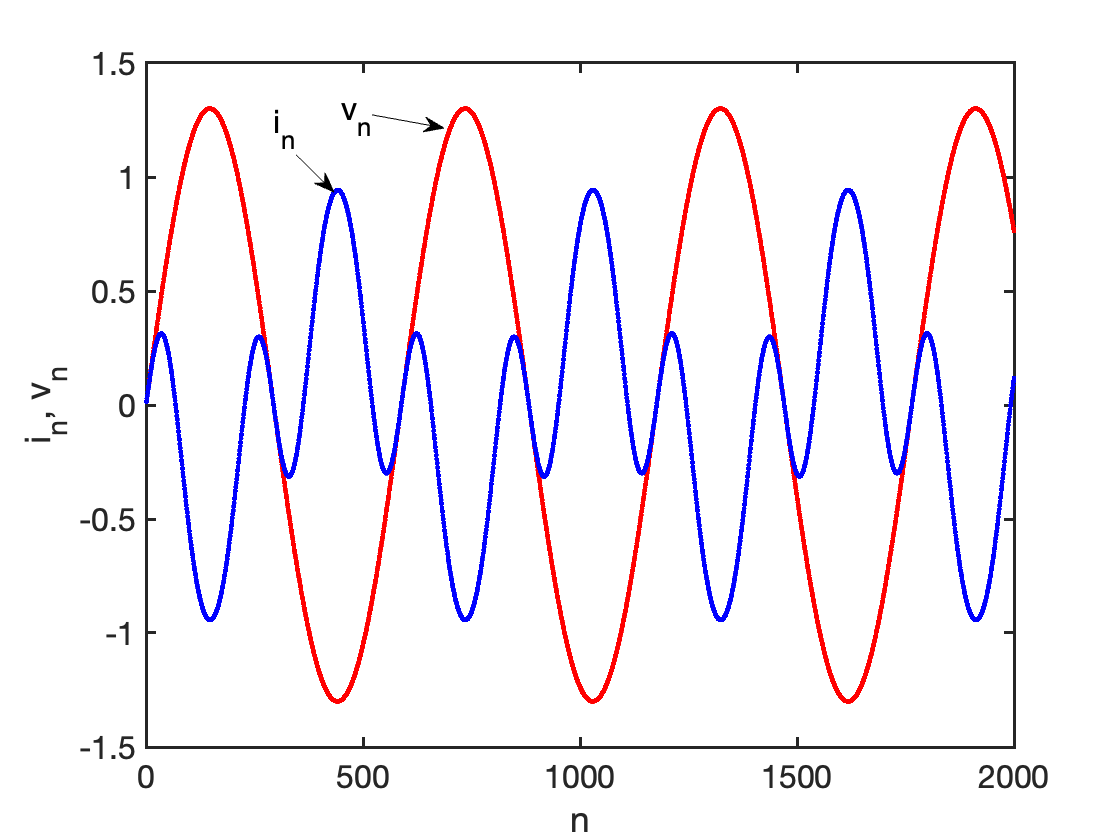}%
        }%
     \caption{Sequential diagram of the DM model. Parameters are considered as: $k_1=0.7, k_2=0.2, $ and $ \phi_0=0$.}\label{seqence}
\end{figure}

\begin{figure}[ht!] 
    \centering
    \subfloat[$v_m=1,\; \omega=3.7$]{%
        \includegraphics[width=0.55\textwidth]{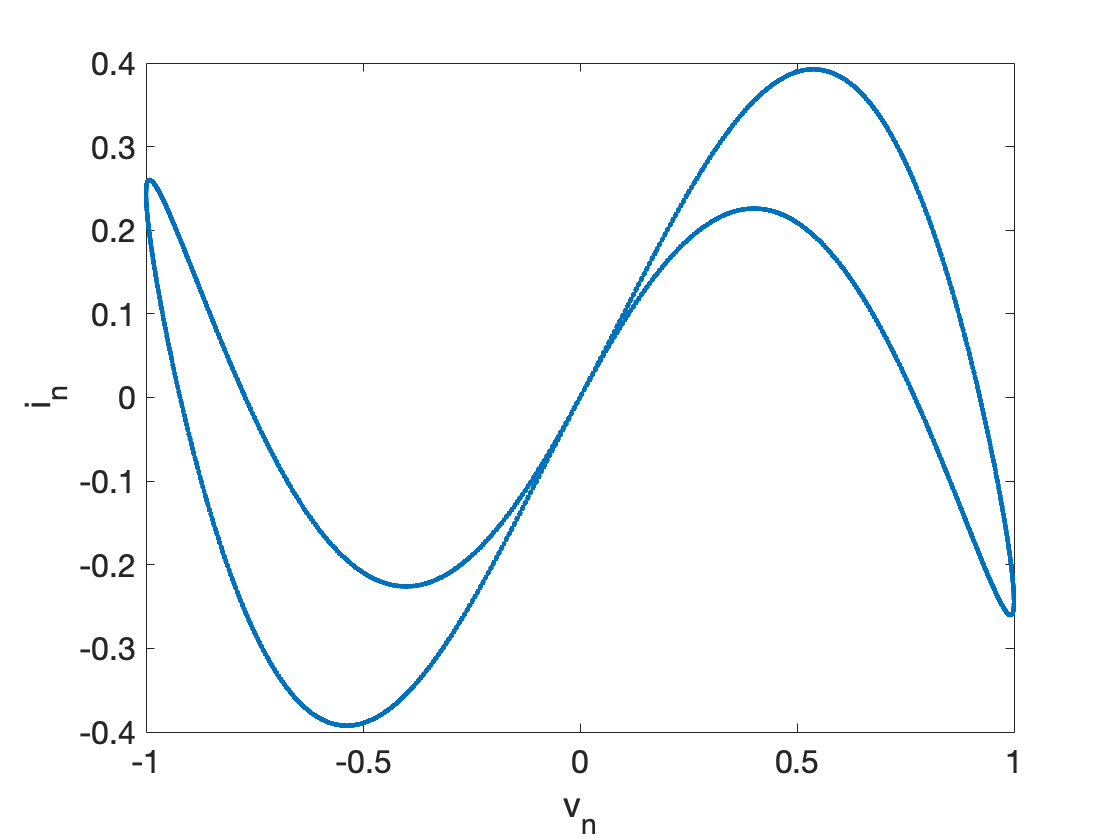}%
        }%
    %\hfill%
    \subfloat[ $v_m=1.3, \; \omega=3.4$]{%
        \includegraphics[width=0.55\textwidth]{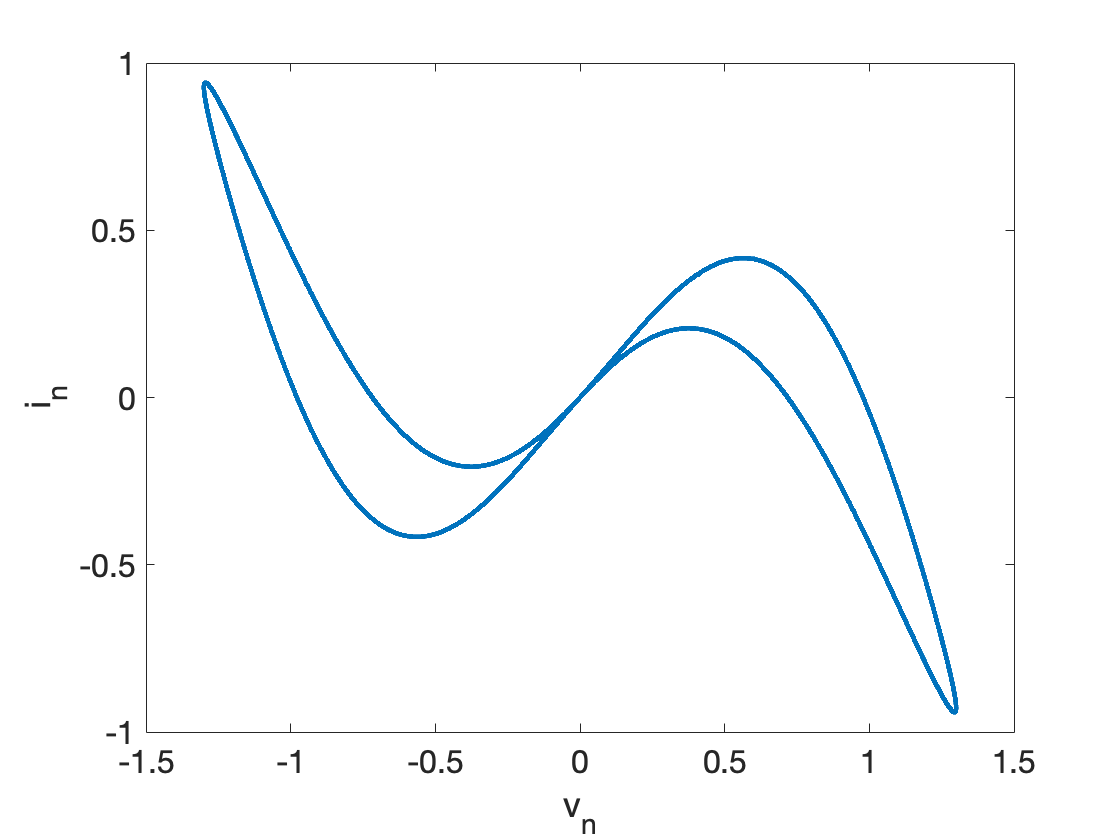}%
        }%
    \caption{Pinched Hysteresis Loops (PHLs) diagram of the DM model for the parameters: $k_1=0.7, k_2=0.2, $ and $ \phi_0=0$.}\label{PHL}
\end{figure}

% the area enclosed by the hysteresis loops increases as the amplitude $V_m$ increases."

As discussed in \cite{adhikari2019three}, the PHLs are the $v-i$ characteristic curves of the memristor, which depends on the different amplitudes and frequencies of the applied signal. When a sinusoidal voltage input $ v_n=v_m\;\sin ({\omega}n)$ is applied to the memristor model Eq. \eqref{1}, the resulting PHLs pass through the origin under different conditions, as illustrated in Fig. \ref{Pinch}. The presence of PHLs confirms that the system behaves like a memristor, i.e., it has both nonlinearity and memory. As observed in Fig. \ref{Pinch}(a), for a fixed frequency $\omega=3.7$, the area enclosed by the hysteresis loop curves increases as the amplitude $v_m$ increases. This means that a large $v_m$ causes greater deviation from the equilibrium, indicating that nonlinearity and memory effects become more pronounced with a stronger force. Moreover, when the amplitude is fixed at $v_m=1$, the area enclosed by the hysteresis loop curves decreases with the increase of frequency $\omega$, as shown in Fig. \ref{Pinch}(b). This indicates that at higher frequencies, the systems have less time to respond fully to input changes, leading to reduced dynamics lag and thus smaller hysteresis loops.

\begin{figure}[ht!] 
    \centering
    \subfloat[$\omega=3.7$]{%
        \includegraphics[width=0.55\textwidth]{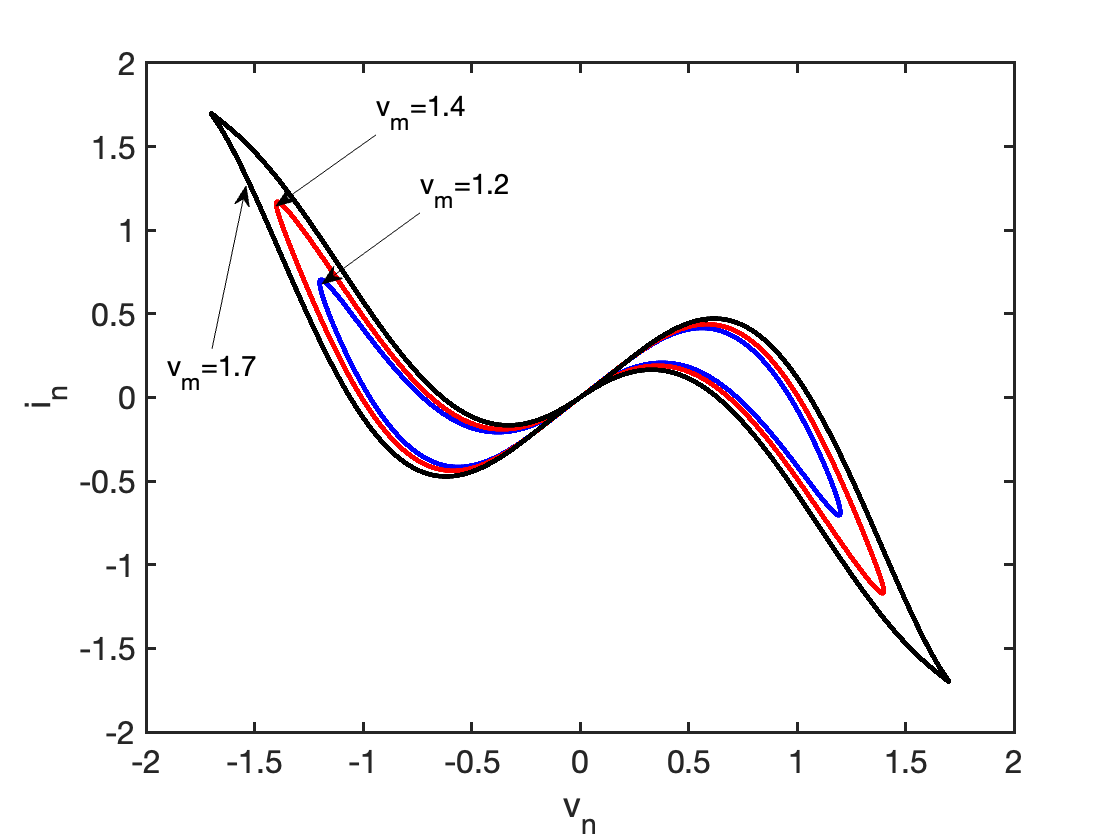}%
        }%
    %\hfill%
    \subfloat[$v_m=1$]{%
        \includegraphics[width=0.55\textwidth]{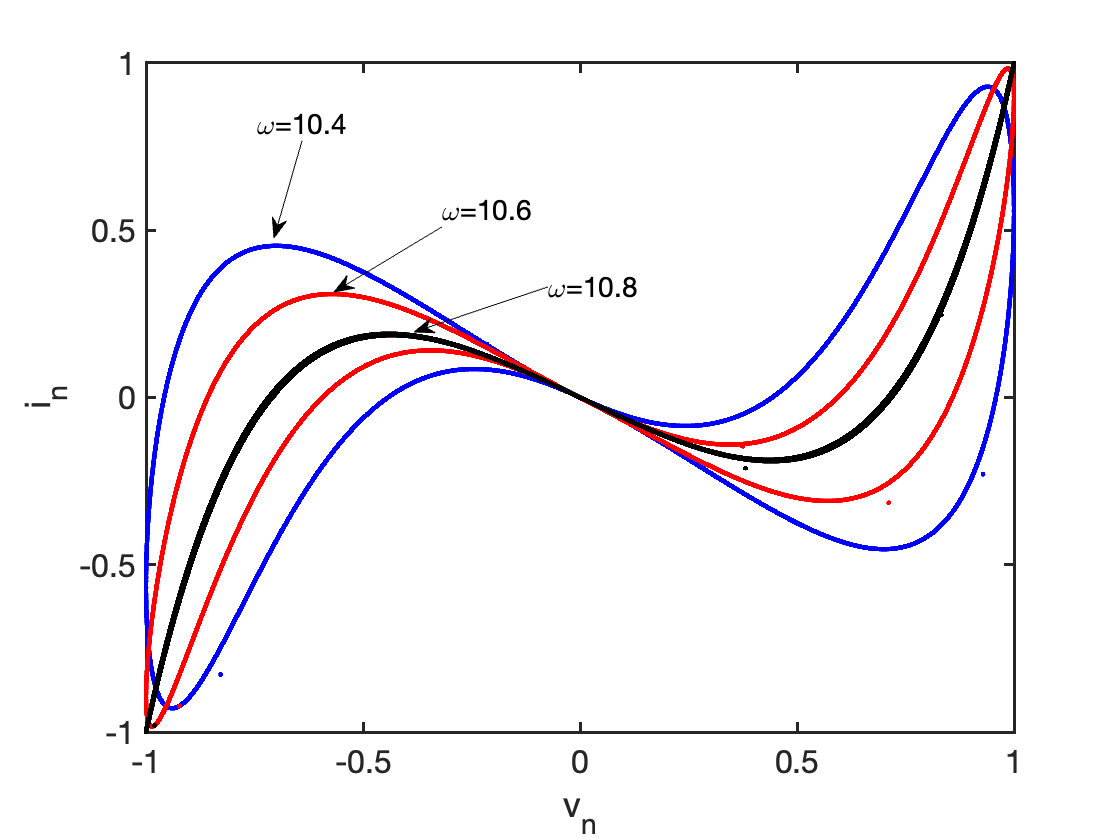}%
        }%
    \caption{The pinched hysteresis loops (PHLs) of the DM model. Parameters are considered as: $k_1=0.7, k_2=0.2, $ and $ \phi_0=0.$ }\label{Pinch}
\end{figure}

% \begin{figure}[ht!] 
%     \centering
%         \includegraphics[width=0.55\textwidth]{POP.png}%
%         \caption{POP of the discrete model.}
% \end{figure}

%%%%%%%%%%%%%%%%%%%%%%%%%%%%%%%%%%%%%%%%%%%%%%%%%%%%%%%%%%%%%%%%%%%%%%%%%%%%
\subsection{M-rChialvo Model } In \cite{chialvo1995generic}, Dante R. Chialvo introduces a 2-D discrete neuron model as follows:

\begin{equation}\label{2}
\left(
\begin{array}{c} 
x_{n+1}  \\ 
y_{n+1} 
\end{array} 
\right)
=
\left(
\begin{array}{c} 
x_n^2e^{(y_n-x_n)}+k_0 \\ 
ay_n-bx_n+c 
\end{array} 
\right) 
\end{equation}

% \begin{align}
%     \begin{split}
%     x_{n+1}&=x_n^2e^{(y_n-x_n)}+k_0,\\
%     y_{n+1}&=ay_n-bx_n+c.
%     \end{split}
% \end{align}

\noindent Here, the variable $x$ represents the membrane potential, while $y$ is the symbol of the recovery (slow) variable. The parameter $k_0$ serves either as a constant bias or a time-dependent additive perturbation. The time constant $a\in(0,1)$, activation dependence $b\in(0,1)$, and $c$ is the offset parameter. 

This model is commonly referred to as the original (or 2-D) Chialvo neuron model. Further, Chialvo treated the map's recovery variable $y$ as a parameter $r$ and therefore the reduced form of the original model as a 1-D model is as follows:

\begin{equation}\label{3}
    x_{n+1}=x_n^2e^{(r-x_n)}+k_0.
\end{equation}
\noindent where $r$ is the real-valued parameter.

% A non-ideal flux-control memristor possessing cosine memconductence \cite{xu2023dynamical} is use to depict the electromagnetic induction in Eq. \eqref{2} (i.e. reduces Chialvo model). In this case, a memristive electromagnetic induction current acts as the role of current stimulus, and the \textit{M-rChialvo} map is  then built as 

% , as discussed in \cite{},

 In this work, we employ a flux-control memristor in the form of cosine memconductance to represent electromagnetic induction into the reduced Chialvo model Eq.\eqref{3}. In this scenario, the memristive electromagnetic induction current serves as the external stimulus, leading to the formation of a novel neuron map as follows:

\begin{equation}\label{4} 
\left(
\begin{array}{c} 
x_{n+1}  \\ 
\phi_{n+1} 
\end{array} 
\right)
=
\left(
\begin{array}{c} 
{x_n}^2 e^{(r-x_n)}+k_0+kx_n \cos(\pi\phi_n) \\ 
k_1x_n-k_2\phi_n 
\end{array} 
\right) 
\equiv\mathcal{M}_{r,k}(x,\phi)
\end{equation}

\noindent This, we refer as the M-rChialvo neuron map (or symbolized with $\mathcal{M}_{r,k}(x,\phi)$), where variable $x$ denotes the membrane potential, while $\phi$ represents the memristor's internal flux state. The term $\cos(\pi\phi)$ symbolizes the periodic memconductance function, while $k$ denotes the coupling strength of the induction current, resulting in $k\;\cos(\pi\phi)$ characterizing the memristive electromagnetic induction current. Here, $k_0$ denotes a time-independent parameter. The parameters $k_1$ and $k_2$ characterize the influence of the memristor's internal state variable in relation to its input and internal dynamics. The term $k_1x$ signifies the variation in magnetic flux induced by the membrane potential, while $k_2\phi$ signifies the leakage flux.  A detailed structural diagram of this model, as described in Eq. \eqref{4}, is illustrated in Fig. \ref{map+mem}. 

\begin{figure}[ht!] 
    \centering
\subfloat{\includegraphics[width=0.8\textwidth,height=0.4\textwidth]{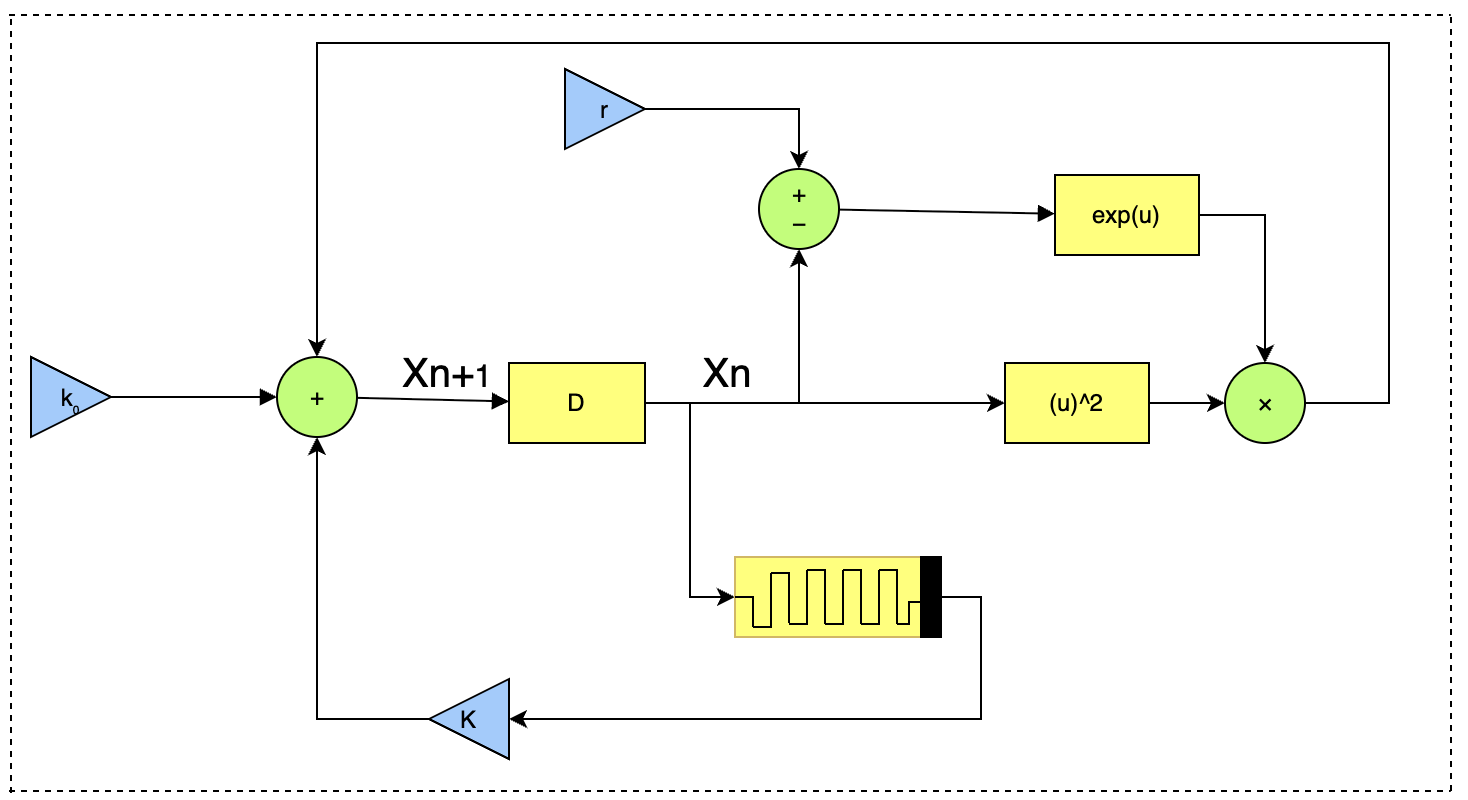}}%
        \caption{The structure diagram of the map $\mathcal{M}_{r,k}(x,\phi)$.}\label{map+mem}
\end{figure}

\section{Existence of fixed point and stability} We investigate the fixed points of the map $\mathcal{M}_{r,k}(x,\phi)$ and analyze their stability. Due to the presence of trigonometric and exponential terms, obtaining the explicit fixed points is analytically challenging. To address this, we derive conditions on the parameters that ensure the existence and uniqueness of a positive fixed point. Furthermore, we conduct a numerical study to explore further details of the fixed points. 

\subsection{Unique fixed point and stability analysis}\label{3.1}

This study aims to demonstrate the existence and uniqueness of the unique positive fixed point. For the convenience of our analysis, we introduce the following terminology, which will be used throughout this paper: let $N>1$ be a natural number, $M_1=N+2$, and $M_2=\frac{k_1}{1+k_2}M_1$.

\begin{theorem}
Let $\mathcal{M}_{r,k}(x,\phi)$ be the map, satisfying the condition $0<k_0<N$ for $N\in \mathbb{N}$, then the map exhibit a unique positive fixed point $(x^*,\phi^*)\in [0,M_1]*[0,M_2]$, whenever the parameter satisfies the following conditions: $k\in (-\frac{1}{M_1},\frac{1}{M_1})$, $0<k_1<1$, $k_2\geq4\pi$ and $r<0$.
\end{theorem}

% \vspace{0.2cm}
\begin{proof}
     The fixed points of the map $\mathcal{M}_{r,k}(x,\phi)$ are obtained by solving the following equations simultaneously.
\begin{equation}\label{5a}
        x^2e^{(r-x)}+k_0+kx \cos(\pi\phi)=x,
\end{equation}
\noindent and
  \begin{equation}\label{5b}
         k_1x-k_2\phi=\phi.
  \end{equation}

\noindent From Eq.\,\eqref{5b}, we write
\begin{equation}\label{6}
    \phi=\frac{k_1x}{1+k_2}.
\end{equation}

\noindent Use the value of $\phi$ in Eq.\eqref{5a}, we obtain
\begin{equation}\label{7}
    x^2e^{(r-x)}+k_
0+kx \cos \Big\{\pi\left(\frac{k_1x}{1+k_2}\right) \Big\}=x.
\end{equation}

\noindent Let us rewrite the above expression in the form of a function
\begin{equation}\label{8}
   F(x)= f(x)-x, \hspace{0.4cm}  where \hspace{0.4cm} f(x)=x^2e^{(r-x)}+k_
0+kx \cos \Big\{\pi\left(\frac{k_1x}{1+k_2}\right) \Big\}.
\end{equation}

\noindent Note that the $x$-component of fixed points of the map $\mathcal{M}_{r,k}(x,\phi)$ is same as the roots of above function $F(x)$. Also, the $\phi$-component can be determined by using Eq. \eqref{5b}.

\noindent Assume that  $0<k_0<N$, then it follows that \; $F(0)=k_0 \implies F(0)>0$.

\noindent Further assume that the condition  $r<0$, $M_1=N+2$, $k\in (-\frac{1}{M_1},\frac{1}{M_1})$, $0<k_1<1$, $k_2\geq2\pi$ holds, then we obtain that \; $F(M_1)<0$.

\noindent Therefore, $F(x)$ has at least one positive real root in the interval $[0,M_1]$. Further, to establish the uniqueness of the positive fixed point, we calculate

$F'(x)=(2-x)xe^{r-x}+k \cos\left(\frac{\pi{k_1}x}{1+k_2}\right)-\frac{\pi{k}xk_1}{1+k_2} \sin\left(\frac{k_1{\pi}x}{1+k_2}\right)-1$

\noindent Since $(2-x)xe^{r-x}<\frac{1}{2}$, $-\frac{1}{M_1}<k \cos\left(\frac{\pi{k_1}x}{1+k_2}\right)<\frac{1}{M_1}$,\\ and \;
$\frac{1}{4}<-\frac{\pi{k}xk_1}{1+k_2} \sin\left(\frac{k_1{\pi}x}{1+k_2}\right)<\frac{1}{4}$. 

\vspace{0.2cm}
\noindent So, we get $F'(x)<0$,  for all $x\in[0,M_1]$.

\noindent Hence the map $\mathcal{M}_{r,k}(x,\phi)$ has a unique positive fixed point in $[0,M_1]*[0,M_2]$.
\end{proof}

\vspace{0.2cm}

\noindent We now proceed to analyze the stability behavior of this positive fixed point $(x^*,\phi^*)$. Let J denote the Jacobian matrix of the map $\mathcal{M}_{r,k}(x,\phi)$
calculated at the fixed point $(x^*,\phi^*)$ and it is given by:

\[
J(x^*,\phi^*) = 
\begin{bmatrix}
(2-x^{*})x^{*}e^{r-x^{*}}+k \cos(\pi\phi^*) \hspace{1cm}&      -k{\pi}x  \sin(\pi\phi^*) \\
k_1 & -k_2
\end{bmatrix}
\]

\noindent Moreover, the characteristic polynomial of the Jacobian matrix J calculated at the fixed point $(x^*,\phi^*)$ is expressed as:  

  \begin{equation}\label{char_eqn}
   P(\lambda)= \lambda^2-p(x^*,\phi^*)\lambda+q(x^*,\phi^*),
\end{equation}

\noindent where

 $p(x^*,\phi^*)= (2-x^*)\gamma_1+k \cos(\pi\phi^*)-k_2$,
 
 $q(x^*,\phi^*)= -k_2(2-x^*)\gamma_1+k \cos(\pi\phi^*)+k{k_1}\gamma_2$,
 
 $\gamma_1 := x^*e^{r-x^*}$, and $\gamma_2 := {\pi}{x^*} \sin({\pi}{\phi^*}).$

% \noindent  the fixed point $(x^*,\phi^*)$ is classified as a \textit{sink} if $\lvert \lambda_{1,2} \rvert <1$, in which it is locally asymptotically stable. It is called a \textit{source} or repeller if $\lvert \lambda_{1,2} \rvert >1$, implying that the source is always unstable. If one eigenvalue lies inside the unit circle and the other lies outside. i.e., $\lvert \lambda_1 \rvert <1$ and $\lvert \lambda_2 \rvert >1$, then the fixed point is called a \textit{saddle}, which is also unstable. Finally, the fixed point is known as non-hyperbolic if either of the eigenvalues has a modulus equal to one, i.e., $\lvert \lambda_1 \rvert = 1$ or $\lvert \lambda_2 \rvert = 1$. 

\noindent

\noindent Let $\lambda_1$ and $\lambda_2$ be the eigenvalues of matrix J calculated at the unique positive fixed point $(x^*,\phi^*)$ of the map $\mathcal{M}_{r,k}(x,\phi)$, i.e., the roots of the characteristic Eq. \eqref{char_eqn}. Then, to analyze the stability of the unique positive fixed point of the map $\mathcal{M}_{r,k}(x,\phi)$, we use the following known Lemma. The result can be easily obtained from the relationship between the roots and the coefficients of a characteristic equation.

\vspace{0.5cm}
\begin{lemma}
        Let $P(\lambda)=\lambda^2-S\lambda+T$, and suppose that $P(1)>0$ with $\lambda_1, \lambda_2$ are the roots of the equation $P(\lambda)=0$. Then the following results follow:    
 \begin{enumerate}\label{lemma_stability}
     \item $ \lvert \lambda_{1,2} \rvert <1$ if and only if $P(-1)>0$ and $T<1$.
     \item $\lvert \lambda_1 \rvert <1$ and $\lvert \lambda_2 \rvert >1$ if and only if $P(-1)<0$
     \item $\lvert \lambda_{1,2} \rvert >1$ if and only if $P(-1)>0$ and $T>1$.
     \item $ \lambda_1  =-1$ and $\lvert \lambda_2 \rvert \neq1$ if and only if $P(-1)=0$ and $T \neq0,2$.
     \item $\lambda_1$ and $\lambda_2$ are complex and $\lvert\lambda_{1,2}\rvert$=1 if and only if $S{^2}-4T<0$ and $T=1$.
 \end{enumerate}   
\end{lemma}

\noindent The following proposition shows the local dynamics of the unique positive fixed point $(x^*,\phi^*)$ from Lemma-\ref{lemma_stability}. 

\vspace{0.2cm}
\begin{proposition}
Consider the map $\mathcal{M}_{r,k}(x,\phi)$ possessing the fixed point $(x^*,\phi^*)$, then we characterized the following natures:
    \begin{enumerate}
        \item The unique fixed point is a repeller if and only if
        
        $\lvert -k_2[(2-x^{*})x^{*}e^{r-x^{*}}+k \cos(\pi\phi^*)]+{\pi}k{k_1}x \sin({\pi}\phi^*) \rvert >1$,
 
  % \vspace{0.2cm}       
\noindent and 
% \vspace{0.2cm} 
  
        $\lvert (2-x^{*})x^{*}e^{r-x^{*}}+k \cos(\pi\phi^*)-k_2 \rvert <\lvert 1-k_2[(2-x^{*})x^{*}e^{r-x^{*}}+k \cos(\pi\phi^*)]+{\pi}k{k_1}x \sin({\pi}\phi^*) \rvert $.

        % \vspace{0.2cm}
        \item The unique positive fixed point is a saddle if and only if 

        $[(2-x^{*})x^{*}e^{r-x^{*}}+k \cos(\pi\phi^*)-k_2]^2>-4k_2[(2-x^{*})x^{*}e^{r-x^{*}}+k \cos(\pi\phi^*)]+4{\pi}k{k_1}x \sin({\pi}\phi^*)$,
        
% \vspace{0.2cm} 
      \noindent  and
% \vspace{0.2cm}

       $\lvert (2-x^{*})x^{*}e^{r-x^{*}}+k \cos(\pi\phi^*)-k_2 \rvert > \lvert 1-k_2[(2-x^{*})x^{*}e^{r-x^{*}}+k \cos(\pi\phi^*)]+{\pi}k{k_1}x \sin({\pi}\phi^*) \rvert $.
% \vspace{0.2cm}        
        \item The  unique fixed point is a non-hyperbolic if and only if 
\begin{equation}\label{saddle_c1}
\begin{split}
    \lvert (2-x^{*})x^{*}e^{r-x^{*}}+k \cos(\pi\phi^*)-k_2 \rvert =\lvert  1-k_2[(2-x^{*})x^{*}e^{r-x^{*}}\\
    +k\;\cos(\pi\phi^*)]+{\pi}k{k_1}x \sin({\pi}\phi^*)\rvert. 
    \end{split}
\end{equation}

% \vspace{0.2cm} 
\noindent        or
        \begin{equation}\label{saddle_c2}
        \begin{split}
              k_2(x^*-2)x^*e^{r-x^*}-k_2k \cos(\pi\phi^*)+k_1k\pi\;\sin(\pi\phi)=1,\vspace{0.3cm} and \\
         \lvert (2-x^*)x^*e^{r-x^*}+k \cos(\pi\phi^*)-k_2 \rvert \leq 2.
         \end{split}
\end{equation} 
        
    \end{enumerate}
\end{proposition}

\noindent The following result shows the necessary and sufficient conditions under which the unique positive fixed point of the map $\mathcal{M}_{r,k}(x,\phi)$ is locally asymptotically stable.

% \vspace{0.2cm}
\begin{proposition}
Let $\mathcal{M}_{r,k}(x,\phi)$ be the map, and both the Eqs. \eqref{saddle_c1}, \eqref{saddle_c2} are not satisfied together. Then the unique positive fixed point $(x^*,\phi^*)$  of the map is locally asymptotically stable if and only if 
   \begin{equation}
   \begin{split}   
       \lvert (2-x^{*})x^{*}e^{r-x^{*}}+k \cos(\pi\phi^*)-k_2 \rvert <1-k_2[(2-x^{*})x^{*}e^{r-x^{*}} +k \cos(\pi\phi^*)]\\+{\pi}k{k_1}x\;\sin({\pi}\phi^*)<2.
    \end{split}
   \end{equation}    
\end{proposition}

%%%%%----------------------------------

\subsection{Numerical stability analysis}

After the examination of a unique positive fixed point of the map $\mathcal{M}_{r,k}(x,\phi)$, a natural next step is to use the numerical simulation to determine whether we can further explore more about fixed points and their stability behavior. To address this, we use numerical simulations to compute the fixed points by utilizing the earlier established result that the $x$-component of the fixed points coincides with the roots of Eq. \eqref{8}. Further, use the earlier subsection (\ref{3.1}) condition on eigenvalues to determine the corresponding stability. This information on fixed points and their stability behavior with respect to $r$ is plotted in Fig. \ref{Fig3}. The blue (thick) curve represents the stable fixed point (i.e., eigenvalues lie within the unit disc), the red (thin) curve represents the saddle fixed points (i.e., one eigenvalue lies within the unit disc and one lies outside the unit disc), and the black (dash) curve represents the unstable fixed points (i.e., both eigenvalues lie outside the unit disc). In this analysis of the fixed points and their stability, the map $\mathcal{M}_{r,k}(x,\phi)$ exhibits dynamically rich behavior for the chosen set of parameter values as shown in Fig. \ref{Fig3}. By ``dynamically rich" we refer to the presence of the maximum number of fixed points accompanied by the variety of stability behaviors observed.

\begin{figure}[h]
   \subfloat[$x$ component of the fixed point]
      {\includegraphics[width=.55\textwidth]{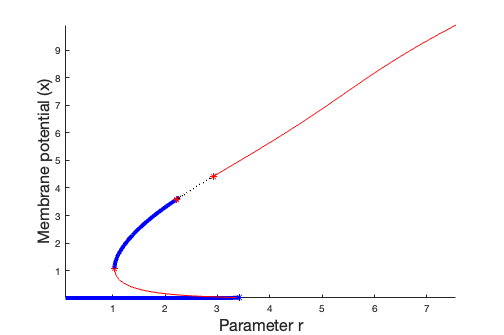}}
~
   \subfloat[$\phi$ component of the fixed point]
      {\includegraphics[width=.55\textwidth]{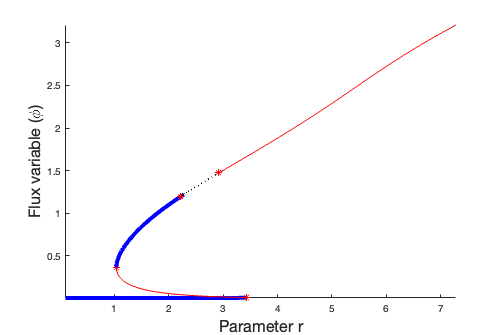}}

   \caption{ Illustrates the number and stability characteristics of fixed points for the map $\mathcal{M}_{r,k}(x,\phi)$ w.r.to the parameter $r$. Fixed points are color-coded: blue (thick) denotes asymptotically stable points, red (thin) indicates saddle points, and black (dash) represents unstable points. Parameters are considered as: $k_0=0.01, k_1=0.5, k_2=0.5, $ and $ k=-0.1$.}\label{Fig3}
\end{figure}

\noindent To validate the findings in Fig. \ref{Fig3}, we use the graphical approach to verify the number of fixed points for some discrete values of parameter $r$, as shown in Appendix Fig. \ref{Fig4}. Furthermore, we also calculated the fixed points and their stability behavior numerically as shown in Table \ref{Table:2}, which validates the information depicted in Fig. \ref{Fig3}.   

\begin{table}[ht!]
\centering
\caption{Fixed point and Stability analysis of the map $\mathcal{M}_{r,k}(x,\phi)$.}
\vline\begin{tabular}{||c||c||c||c||}
\hline
$r$ &     $(x,\phi)$  &      eigenvalues &              Behavior \\
\hline
$2$ & $(0.00972,0.00324)$ & $(0.04159,-0.5)$& stable \\ 
& $(0.1659,0.0553)$ & $(1.8,-0.5)$&  saddle\\ 
&$ (3.296,1.098)$ &$ (-0.78 + 0.275i, -0.78-0.275i)$& stable  \\
 \hline
$2.8$ & $(0.01082,0.0036)$ & $(0.25,-0.5)$&  stable\\ 
& $(0.06033,0.02)$ & $(1.712,-0.5)$& saddle \\
& $(4.279,1.426)$ &$ (-1.609,-1.0896)$& unstable \\
\hline
${3}$ & $(0.01146,0.0038)$ &$ (0.352,-0.5)$&  stable\\ 
& $(0.04601,0.0153)$ & $(1.6248,-0.5)$&  saddle\\  
& $(4.507,1.5)$ & $(-2.046,-0.9)$&  saddle\\ 
\hline
{$5$} & $(6.869,2.289)$ & $(-5.395,-0.325)$& saddle \\

\hline 
\end{tabular}\vline\label{Table:2}
\end{table}

%------------------------------------------------------------

\section{Bifurcation analysis}
A dynamical phenomenon that helps us to understand the qualitative change in stability of a system occurs with a change in a parameter and has many practical applications, such as predicting the spiking activities of neurons. In \cite{izhikevich2007dynamical}, it is noted that there can be various reasons triggering a neuron to fire a spike, but there are only four primary types of bifurcation that mathematically can explain the cause of these behaviors. This underscores the importance of bifurcation analysis, as studying only four types of bifurcation has the potential to understand millions of ways of spiking generation. In this section, we investigate the system bifurcation behavior using both analytically and numerical approaches. 

\subsection{Neimark-Sacker (NS) bifurcation}

We analyze the NS-bifurcation of the unique positive fixed point $(x^*,\phi^*)$ of the map $\mathcal{M}_{r,k}(x,\phi)$, by employing bifurcation theory with $k$ treated as a bifurcation parameter. An NS-bifurcation gives rise to dynamically invariant circles, which can contain both isolated orbits with periodic behavior and dense trajectories covering the invariant circle \cite{din2017global}. We aim to determine the condition under which the map $\mathcal{M}_{r,k}(x,\phi)$ exhibits a non-hyperbolic fixed point characterized by a pair of complex conjugate eigenvalues with unit modulus. Another type of methodology is used to study the NS-bifurcation for a $2D$ discrete delayed Homographic Ricker map, as shown in \cite{chandramouli2025dynamics}.

 Now, the characteristic polynomial Eq. \eqref{char_eqn} of the Jacobian matrix of the map $\mathcal{M}_{r,k}(x,\phi)$ at the unique positive fixed point $(x^*,\phi^*)$ possess a pair of complex conjugate eigenvalues with unit modulus if the following conditions are fulfilled:
\begin{equation}
    k=\frac{1+{k_2}(2-x^*)\gamma_1}{k_{1}\gamma_2-k_2\;\cos(\pi\phi^*)}; \hspace{0.4cm} k_{1}\gamma_2\neq k_2\; \cos(\pi\phi^*), \label{K}
\end{equation}

\noindent and
\begin{equation}
    \big\lvert  (2-x^*)\gamma_1+k\;\cos(\pi\phi^*) -k_2 \big\rvert<2. \label{2nd}
\end{equation}

\noindent where $\gamma_1 = x^*e^{r-x^*}$, and $\gamma_2 = {\pi}{x^*}\;\sin({\pi}{\phi^*})$.

\vspace{0.2cm}

\noindent Assume that :  $\Psi_{NS}=\{(k_0,k_1,k_2,k,r): $ Eqs. $ \eqref{K} $ and $ \eqref{2nd}$ are satisfied.$\}$    

\vspace{0.2cm}
\noindent The unique positive fixed point $(x^*,\phi^*)$ of the map $\mathcal{M}_{r,k}(x,\phi)$ undergoes NS-bifurcation as the parameters vary within the small neighborhood of $\Psi_{NS}$. Let $k^{'}=\frac{1+{k_2}(2-x^*)\gamma_1}{k_{1}\gamma_2-k_2\;\cos(\pi\phi^*)}$ and $(k_0,k_1,k_2,k^{'},r) \in \Psi_{NS}$, we consider the map $\mathcal{M}_{r,k}(x,\phi)$ with parameter $(k_0,k_1,k_2,k^{'},r)$ which is expressed by the following map:
\begin{equation}\label{map_ns_def1}
 \begin{bmatrix} x \\ \phi \end{bmatrix}  \mapsto \begin{bmatrix} x^{2}e^{r-x}+k_0+k^{'}x\;\cos(\pi\phi) \\ k_{1}x-k_{2}\phi \end{bmatrix}
\end{equation}

\noindent We treat $\Tilde{k}$ as the bifurcation parameter and apply the perturbation in Eq. \eqref{map_ns_def1} in the following form:
\begin{equation}\label{map_ns_def2}
 \begin{bmatrix} x \\ \phi \end{bmatrix}  \mapsto \begin{bmatrix} x^{2}e^{r-x}+k_0+(k^{'}+\Tilde{k})x\;\cos(\pi\phi) \\ k_{1}x-k_{2}\phi \end{bmatrix}
\end{equation}

\noindent where $\lvert\Tilde{k} \rvert\ \ll 1$ represents a small perturbation.

\noindent Next, we apply the transformation $X=x-x^*, \Phi=\phi-\phi^*$, then the map \eqref{map_ns_def2} is transferred as follows:
\begin{equation}\label{map_ns_def3}
 \begin{bmatrix} X \\ \Phi \end{bmatrix}  \mapsto 
  \begin{bmatrix} a_{11} & a_{12} \\ a_{21} & a_{22} \end{bmatrix}  \begin{bmatrix} X \\ \Phi \end{bmatrix}+
 \begin{bmatrix} f_{1}(X,\Phi) \\ f_{2}(X,\Phi)
 \end{bmatrix}
\end{equation}

\noindent where

$f_{1}(X,\Phi)=a_{13}X^2+a_{14}X{\Phi}+a_{15}\Phi^2+b_{1}X^3+b_{2}X^{2}\Phi+b_{3}X\Phi^2+b_{4}\Phi^3+O\left((\lvert X \rvert+ \lvert \Phi \rvert)^4\right)$,

\vspace{0.2cm}
$f_{2}(X,\Phi)=a_{23}X^2+a_{24}X{\Phi}+a_{25}\Phi^2+c_{1}X^3+c_{2}X^{2}\Phi+c_{3}X\Phi^2+c_{4}\Phi^3+O\left((\lvert X \rvert+ \lvert \Phi \rvert)^4\right)$,

\vspace{0.2cm}
$a_{11}=(2-x^*)x^{*}e^{r-x^*}+(k^{'}+\Tilde{k})\;\cos(\pi\phi^*)$, \hspace{0.3cm} \hspace{0.3cm} $a_{12}=-\pi(k^{'}+\Tilde{k})x^{*}\;\sin(\pi\phi^*)$, \hspace{0.3cm} \hspace{0.3cm}  $a_{22}=-k_2$, \hspace{0.3cm}  $a_{21}=k_1$, \hspace{0.3cm}  $a_{13}=\frac{e^{r-x^*}(x^*-2)^2}{2}$,  \hspace{0.3cm} $a_{14}=-\pi(k^{'}+\Tilde{k})\;\sin(\pi\phi^*)$, \hspace{0.3cm} $a_{15}=-\frac{\pi^{2}(k^{'}+\Tilde{k})x^{*}\;\cos(\pi\phi^*)}{2}$, \hspace{0.3cm}    $a_{23}=0$, \hspace{0.3cm} $a_{24}=0$, \hspace{0.3cm} $a_{25}=0$,\hspace{0.3cm}
$b_{1}=\frac{e^{r-x^*}(x^*-2)(4-x^*)}{12}$,\hspace{0.3cm}
$b_{2}=0$, \hspace{0.3cm}
$b_3=-\frac{\pi^2(k^{'}+\Tilde{k})\;\cos(\pi\phi^*)}{12},$  \hspace{0.3cm}          $b_4=\frac{\pi^3(k^{'}+\Tilde{k})x^{*}\;\sin(\pi\phi^*)}{12},$\hspace{0.3cm}
$c_1=0$, \hspace{0.2cm} $c_2=0$,\hspace{0.2cm} $c_3=0$,\hspace{0.2cm} and\hspace{0.2cm} $c_4=0.$

\noindent The characteristic equation of the Jacobian matrix for the system \eqref{map_ns_def3} calculated at the fixed point $(0,0)$ is given by:
\begin{equation}\label{char_for_k'}
    \lambda^2-p(\Tilde{k})\lambda+q(\Tilde{k})=0.
\end{equation}

\noindent where 
\begin{equation}
    p(\Tilde{k})=(2-x^*)\gamma_1+(k^{'}+\Tilde{k})\;\cos(\pi\phi^*)-k_2.\label{(p(x))}
\end{equation}

\noindent and 
\hspace{1.8cm} $q(\Tilde{k})=-k_2[(2-x^*)\gamma_1+(k^{'}+\Tilde{k})\;\cos(\pi\phi^*)]+(k^{'}+\Tilde{k})k_{1}\gamma_2$.

% \vspace{0.2cm}
\noindent Since $(k_0,k_1,k_2,k^{'},r) \in \Psi_{NS}$, the root of \eqref{char_for_k'} are complex conjugate number $\lambda_{1,2}$ with $\lvert \lambda_{1,2} \rvert = 1$. Therefore, it follows:

% \vspace{0.2cm}
$\lambda_{1,2} = \frac{p(\Tilde{k})}{2} \pm \frac{i}{2}\sqrt{4q(\Tilde{k})-p^2(\Tilde{k})}$.

% \vspace{0.2cm}
\noindent Then we obtain

% \vspace{0.2cm}
$\lvert \lambda_{1,2} \rvert = \sqrt{q(\Tilde{k)}}$,\; and \;$\left(\frac{d \lvert \lambda_{1,2} \rvert}{d\Tilde{k}}\right)_{\Tilde{k}=0}  = \frac{k_1\gamma_2-k_2\;\cos(\pi\phi^*)}{2\sqrt{k^{'}k_1\gamma_2-k_2[(2-x^*)\gamma_1+k^{'}\;\cos(\pi\phi^*)]}} \neq 0$.

\vspace{0.2cm}
\noindent Additionally, from the Eq. (\ref{(p(x))}), we have $p(0) \neq 0$, and $p(0) \neq -1$. Since $(k_0,k_1,k_2,k^{'},r) \in \Psi_{NS}$, it follows that $-2<p(0)<2$. Therefore the condition $p(0) \neq \pm 2,0,-1$ ensure that the $\lambda_1^n,  \lambda_2^n \neq1$, for all $ n=\{1,2,3,4\}$ at $\Tilde{k}=0$. Hence, the roots of Eq. \eqref{char_for_k'} do not lie in the intersection point of the unit circle with the coordinate axes under $\Tilde{k}=0$ (i.e., do not take any value from $\{\pm 1, \pm i\}$), provided the following conditions are satisfied:
\begin{equation}\label{resonance}
    (2-x^*)\gamma_1+k^{'}\;\cos(\pi\phi^*) \neq k_2 , \hspace{0.2cm}   (2-x^*)\gamma_1+k^{'}\;\cos(\pi\phi^*) \neq k_2 -1.
\end{equation}

\noindent To drive the normal form for the system \eqref{map_ns_def3} at $\Tilde{k}=0$, we set $\alpha=\frac{p(0)}{2}$, $\beta=\frac{1}{2}\sqrt{4p(0)-p^2(0)}$ and apply the following transformation:
\begin{equation}\label{map_ns_def4}
 \begin{bmatrix} X \\ \Phi \end{bmatrix}  \mapsto 
  \begin{bmatrix} a_{12} \hspace{0.3cm} &  0 \\ \alpha-a_{11} \hspace{0.3cm} &  -\beta \end{bmatrix}  \begin{bmatrix} u \\ v \end{bmatrix}
\end{equation}

\noindent Under the transformation from the system \eqref{map_ns_def4}, the normal form for the system \eqref{map_ns_def3} takes the following form:
\begin{equation}\label{map_ns_def5}
 \begin{bmatrix} u \\ v \end{bmatrix}  \mapsto 
  \begin{bmatrix} \alpha \hspace{0.3cm} &  -\beta \\ \beta \hspace{0.3cm} &  \alpha\end{bmatrix}  \begin{bmatrix} u \\ v \end{bmatrix}+ \begin{bmatrix}
      G_1(u,v)\\ G_2(u,v)
  \end{bmatrix}
\end{equation}

\noindent where
\noindent $G_1(u,v)= \frac{a_{13}}{a_{12}}X^2+\frac{a_{14}}{a_{12}}X{\Phi}+\frac{a_{15}}{a_{12}}\Phi^2+\frac{b_{1}}{a_{12}}X^3+\frac{b_{2}}{a_{12}}X^{2}\Phi+\frac{b_{3}}{a_{12}}X\Phi^2+\frac{b_{4}}{a_{12}}\Phi^3+O\left((\lvert X \rvert+ \lvert \Phi \rvert)^4\right),$

\vspace{0.2cm}
\noindent $G_2(u,v)= \left(\frac{a_{13}(\alpha-a_{11})}{{\beta}a_{12}}-\frac{a_{23}}{\beta}\right)X^2+\left(\frac{a_{14}(\alpha-a_{11})}{{\beta}a_{12}}-\frac{a_{24}}{\beta}\right)X\Phi+\left(\frac{a_{15}(\alpha-a_{11})}{{\beta}a_{12}}-\frac{a_{25}}{\beta}\right)\Phi^2+\left(\frac{b_{1}(\alpha-a_{11})}{{\beta}a_{12}}-\frac{c_{1}}{\beta}\right)X^3+\left(\frac{b_{2}(\alpha-a_{11})}{{\beta}a_{12}}-\frac{c_{2}}{\beta}\right)X^2\Phi+\left(\frac{b_{3}(\alpha-a_{11})}{{\beta}a_{12}}-\frac{c_{3}}{\beta}\right)X\Phi^2+\left(\frac{b_{4}(\alpha-a_{11})}{{\beta}a_{12}}-\frac{c_{4}}{\beta}\right)\Phi^3+O\left((\lvert X \rvert+ \lvert \Phi \rvert)^4\right).$

\vspace{0.2cm}
\noindent $X=a_{12}u$ and $\Phi=(\alpha-a_{11})u-{\beta}v$. We now define the following nonzero real quantity:

% \vspace{0.2cm}
$\theta=[-Re\left(\frac{(1-2\lambda_1)\lambda_2^2}{1-\lambda_1}L_{11}L_{12}\right)-\frac{1}{2}\lvert L_{11} \rvert^{2}-\lvert L_{21} \rvert^{2}+Re(\lambda_2L_{22})]_{\Tilde{k}=0}$,

% \vspace{0.2cm}
\noindent where

% \vspace{0.2cm}
$L_{11}=\frac{1}{4}[({G_1}_{uu}+{G_1}_{vv})+i({G_2}_{uu}+{G_2}_{vv})]$

% \vspace{0.2cm}
$L_{12}=\frac{1}{8}[({G_1}_{uu}-{G_1}_{vv}+2{G_2}_{uv})+i({G_2}_{uu}-{G_2}_{vv}-2{G_1}_{uv})]$

% \vspace{0.2cm}
$L_{21}=\frac{1}{8}[({G_1}_{uu}-{G_1}_{vv}-2{G_2}_{uv})+i({G_2}_{uu}-{G_2}_{vv}+2{G_1}_{uv})]$

% \vspace{0.2cm}
$L_{22}=\frac{1}{16}[({G_1}_{xxx}+{G_1}_{xyy}+{G_2}_{xxy}+{G_2}_{yyy})+i({G_2}_{xxx}+{G_2}_{xyy}-{G_1}_{xxy}-{G_1}_{yyy})]$

% \vspace{0.2cm}
\noindent Based on the above analysis, we establish the following theorem.

% \vspace{0.2cm}
\begin{theorem}
  Let $\mathcal{M}_{r,k}(x,\phi)$ be the map, and it satisfies the condition \eqref{resonance} and $\theta\neq0$. Then the map $\mathcal{M}_{r,k}(x,\phi)$ undergoes a NS-bifurcation at the unique positive fixed point $(x^*,\phi^*)$ as the parameter $k$ varies in a small neighborhood of $k^{'}=\frac{1+{k_2}(2-x^*)\gamma_1}{k_{1}\gamma_2-k_2\;\cos(\pi\phi^*)}$. Moreover, if $\theta<0$, then an attracting invariant closed curve bifurcation occurs from the fixed point for $k>k^{'}$, whereas if $\theta>0$, then a repelling invariant closed curve bifurcation occurs from the fixed point for $k<k^{'}$.
\end{theorem}

%%-------------------------------------------------------

\subsection{Bifurcation simulation}

\begin{figure}[ht!] 
    \centering
    \subfloat[Bif. of $x$-component of map $\mathcal{M}_{r,k}(x,\phi)$ ]{%
        \includegraphics[width=0.55\textwidth,height=0.35\textwidth]{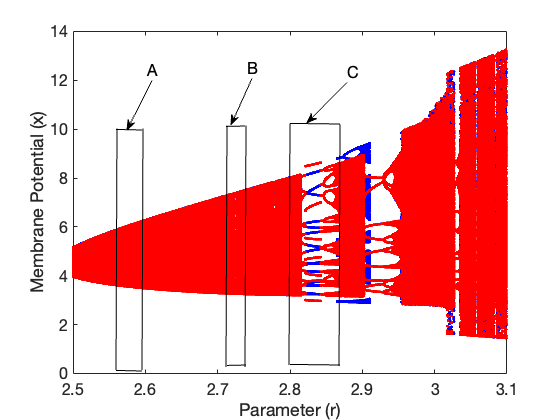}%
        }%
    %\hfill%
    \subfloat[Zoomed part A]{%
        \includegraphics[width=0.55\textwidth,height=0.35\textwidth]{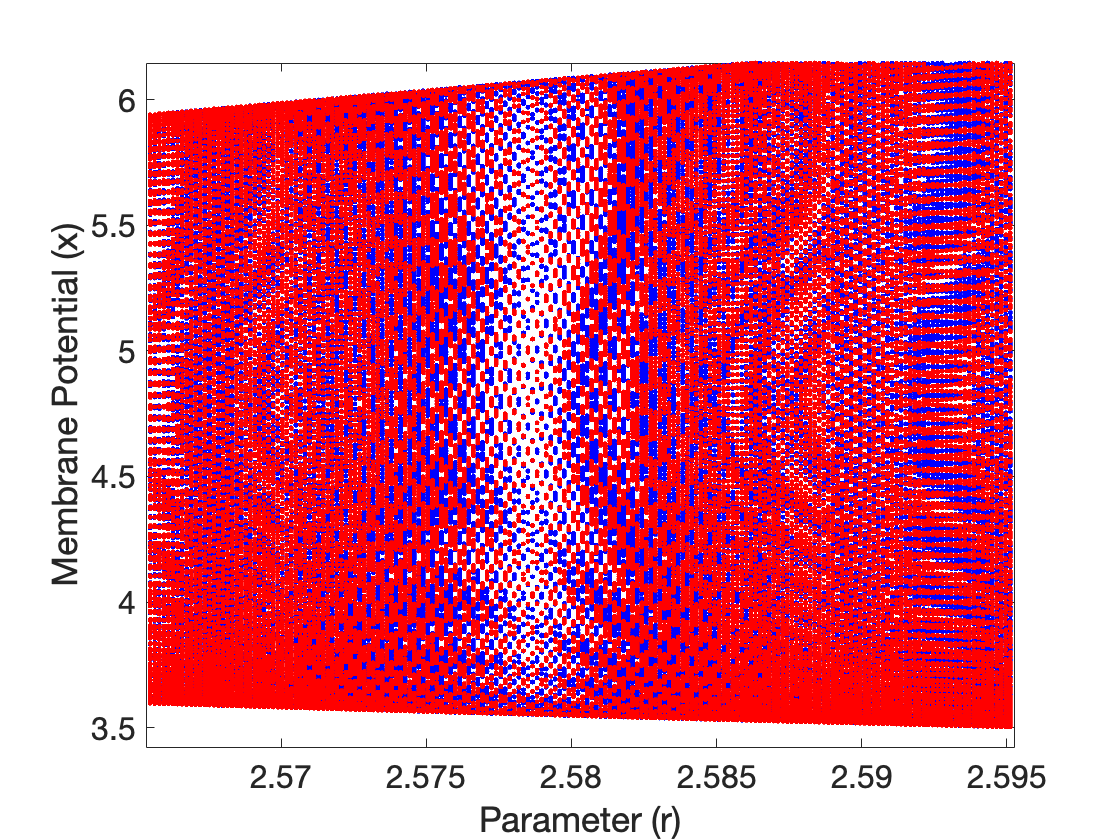}%
        }%
        \vfill
         \subfloat[Zoomed part B]{%
        \includegraphics[width=0.55\textwidth,height=0.35\textwidth]{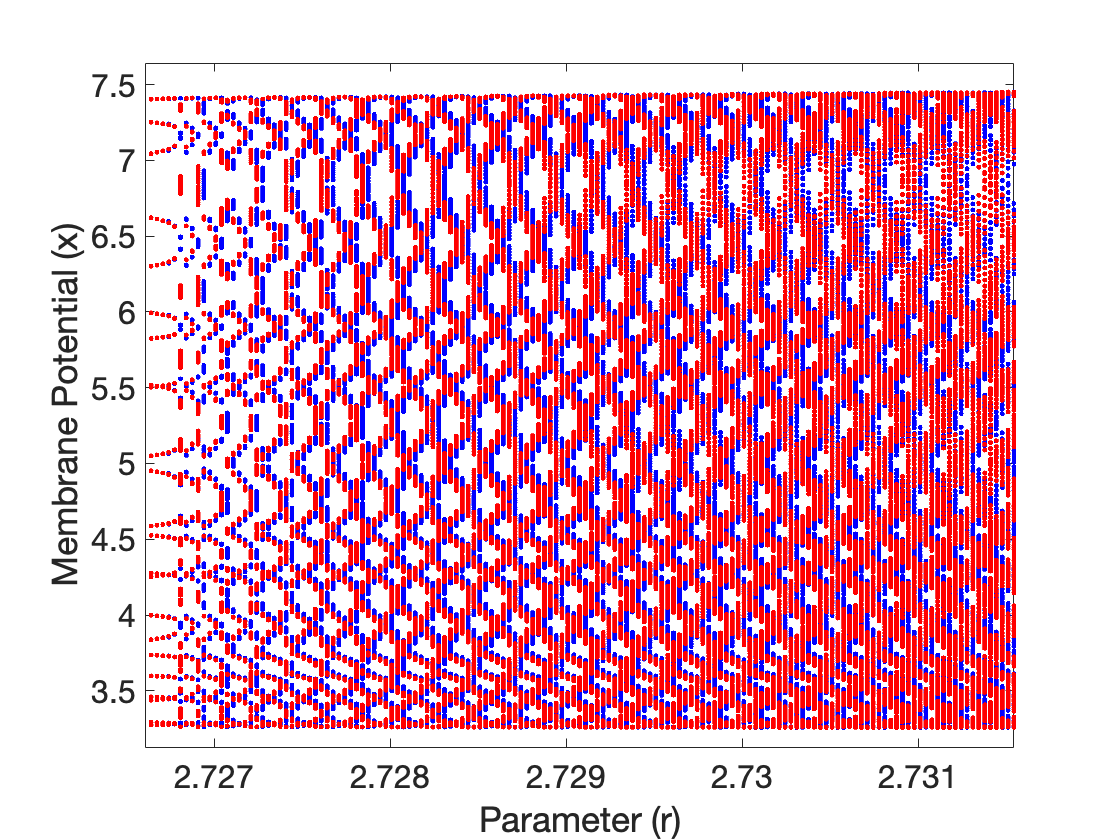}%
        }%
    %\hfill%
    \subfloat[Zoomed part C]{%
        \includegraphics[width=0.55\textwidth,height=0.35\textwidth]{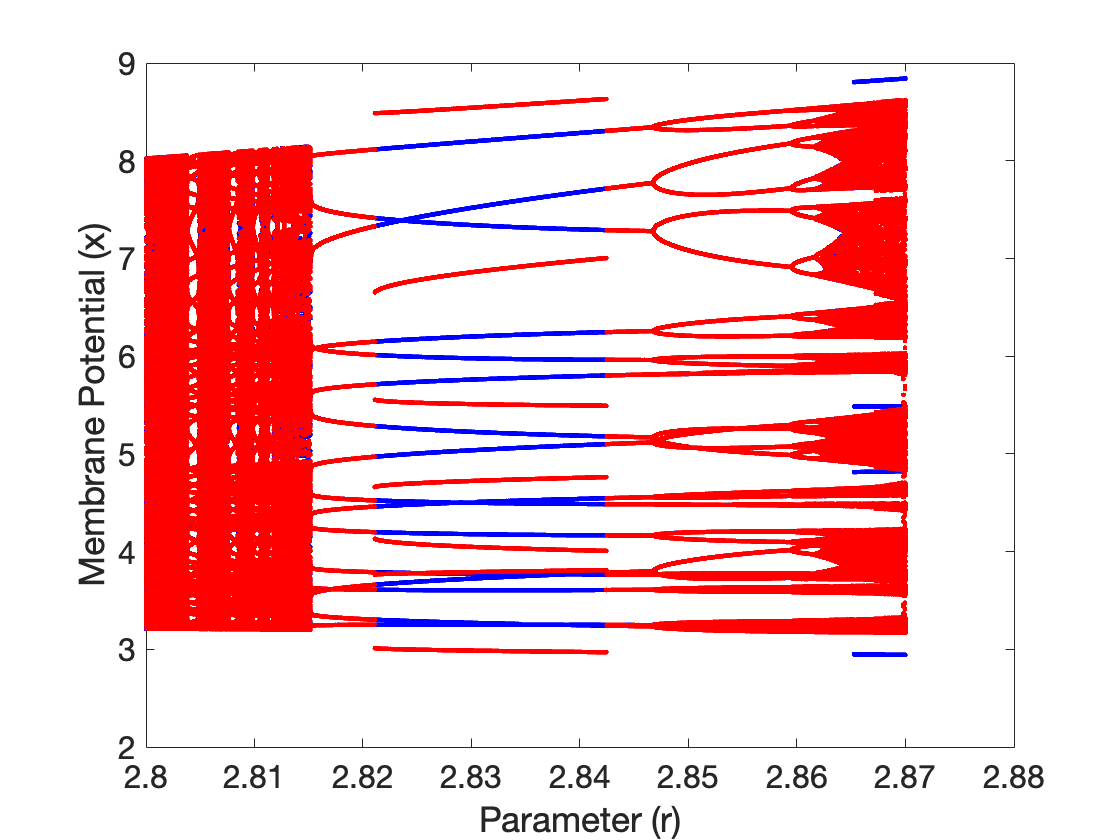}%
        }%
    \caption{(a) shows the forward (marked in red) and backward (marked in blue) bifurcation plot of membrane potential $x$ of the map $\mathcal{M}_{r,k}(x,\phi)$ w.r.to parameter $r$. In (b), (c), and (d) are the zoomed versions for different parameter $r$ ranges. Parameters are considered as: $k_0=1, k_1=0.1, k_2=0.5, $ and $ k=0.3$.}\label{for_back_bif}
\end{figure}

In this subsection, we discuss a numerical bifurcation analysis of the map $\mathcal{M}_{r,k}(x,\phi)$ with respect to the parameter $r$. For a fine range parameter $r\in [2.5,3.1]$, a bifurcation diagram of membrane potential is plotted in the $(r,x)$ plane through both forward and backward continuation, as illustrated in Fig.\ref{for_back_bif}. Forward continuation (which is marked in red color) is performed by varying the recovery parameter $r$ from $2.5$ to $3.1$, and the backward continuation (which is marked with blue color) is performed by varying the recovery parameter $r$ from $3.1$ to $2.5$. The points of both forward and backward continuation are plotted together in the same figure. This approach provides us with the benefit of detecting the multistability in the system. If the points obtained through forward and backward continuation do not fully coincide, this indicates the presence of multistability (a phenomenon in which multiple attractors coexist for a fixed value of parameters) in the system. One might naturally wonder what those stable states imply from this analysis. In this subsection, we focus on the numerical bifurcation analysis. A  detailed investigation of the coexisting multiple attractors and their nature is presented in the next section.

In Fig.\ref{for_back_bif}(a), it is shown that the forward and backward bifurcation plots do not overlap entirely with each other. Although the overlap visualization is unclear from the figure Fig. \ref{for_back_bif}(a). But as we zoom the different ranges with respect to the parameter $r$, then we obtain a more transparent view as demonstrated by those different versions in subfigures (i.e.,(b),(c),(d)) of Fig.\ref{for_back_bif}. Subfigure (b) presents a clearer view of non-overlapping forward and backward bifurcation. In subfigure (c), we observe the formation of chaotic bubbles in the bifurcation diagram. The chaotic bubbles in the bifurcated diagram indicate the simultaneous emergence and disappearance of periodic orbits. This phenomenon is known as anti-monotonicity, and it is driven by the bifurcation diagram of non-monotonic behavior \cite{dawson1992antimonotonicity}. Antimonotonicity is a fundamental characteristic of nonlinear dynamical systems and is found in many systems, including experimental observations in Chua's circuit \cite{kocarev1993experimental}. In subfigure (d), period-doubling and reverse period-doubling phenomena along with a stable state are observed for the map $\mathcal{M}_{r,k}(x,\phi)$.

%-----------------------------

% \newpage

\subsection{Codimensional bifurcation study} In this subsection, we analyze the bifurcation study using MATContM Software as detailed in \cite{kuznet2019numerical}. This study explores the codimension-1 and codimension-2 bifurcation analysis for the map $\mathcal{M}_{r,k}(x,\phi)$. In \cite{depannemaecker2026next}, it has been discussed that the bifurcation patterns that are state invariant under both codimensional-1 and codimensional-2 bifurcations can serve as the indicators of common topological structures across different models. This study used the bifurcation structure, which is preserved from codimensional-1 to codimensional-2 studies, to identify shared topological features, even when the models under consideration differ. In the following subsection, we explore bifurcation analysis considering codimensional-1 and -2 spaces. The abbreviations of different codimensional-1 and -2 bifurcations are illustrated in Table \ref{Table_correlation}.

\begin{table}[ht!]
\centering
\caption{Abbreviations for the codimension-1 and -2 bifurcation.}
\vline\begin{tabular}{||c||c||c||c||}
\hline
\textbf{Codimension-1}\\
\hline
LP & Saddle-node (fold) & PD & Period-doubling (flip) \\
 \hline
NS & Neimark-Sacker \\ 

\hline
\textbf{Codimension-2}
 
 \\
\hline
GPD & Generalized flip  & LPPD& Fold-Flip\\
LPNS& Fold-Neimark-Sacker & PDNS& Flip-Neimark-Sacker\\
CH & Chenciner  & CP& Cusp  \\ 
R1 & 1:1 resonance & R2 & 1:2 resonance \\
R3 & 1:3 resonance & R4 & 1:4 resonance \\ 
\hline 
\end{tabular}\vline\label{Table:3}
\end{table}

We began by considering $k_0$ and $k_1$ as the primary parameters for the codimension-1 and -2 study of the map $\mathcal{M}_{r,k}(x,\phi)$, which is demonstrated in Fig. \ref{N_one}. Subfigure \ref{N_one}(a) presents the bifurcation diagram showing the change in membrane potential $x$ with the change in parameter $k_0$, and treating another parameter constant as $r=2, k=-0.1, k_1=0.5, $ and $ k_2=0.5$. Here, the thick blue line denotes the stable fixed point, while the dotted red line indicates the unstable fixed point of the map $\mathcal{M}_{r,k}(x,\phi)$. The system shows two different stages for a sufficiently low value of the parameter $k_0$. The first one is an unstable fixed point corresponding to the negative value of the membrane potential $x$. The other one is that the system exhibits a saddle-node $(LP_2)$ bifurcation (a state where a stable fixed point from the upper side and an unstable fixed point from the lower side collapse and disappear), with the normal form of coefficient $-1.057105e^{+00}$. As the parameter $k_0$ increases, the system exhibits three fixed points: one stable and two unstable. Further, with an increase in $k_0$, we witness supercritical period-doubling $(PD_2)$ bifurcation (a dynamical state where the system transforms from a stable fixed point to an unstable one and a stable period-2 orbits emerges) with the normal form of coefficient $4.120733e^{-01}$, and along with that two unstable fixed points are also observed in the system. With further increase in the parameter $k_0$, we observe the subcritical period-doubling $(PD_3)$ bifurcation (when a system transforms from an unstable state to a stable one) with the normal form of coefficient $-4.251796e^{-01}$, and along with this, two unstable fixed points are also seen in the system. Furthermore, as the parameter $k_0$ increases, the system shows the Neimark-Sacker $(NS_1)$ bifurcation (Neutral saddle which indicates a neutral stability state due to a pair of complex conjugate eigenvalues crossing the unit circle), and one unstable and one stable fixed point are also demonstrated. With the increase in parameter $k_0$, we observe a subcritical period-doubling $(PD_1)$ bifurcation with the normal form of coefficient $3.395820e^{+01}$. As $k_0$ increases further, the system showcases the saddle-node $(LP_1)$ bifurcation with the normal form of coefficient $5.386942e^{+00}$, and at this state, one unstable fixed point and one stable fixed point both collapse and disappear, resulting in the system having a single stable fixed point. Subsequently, with an increase in $k_0$, we observe a supercritical Neimark-Sacker $(NS_2)$ (a dynamical state where a stable fixed point loses its stability and gives rise to an unstable fixed point accompanied by a stable limit cycle) with the normal form of coefficient $-6.111164e^{-01}$. Further, with an increase in $k_0$, we witness a supercritical Neimark-Sacker $(NS_3)$ (a dynamical state in which an unstable fixed point and the stable limit cycle collapse, giving rise to a stable fixed point) with the normal form of coefficient $-6.836661e^{-01}$. Through this process, the system demonstrated various bifurcations, including Neimark-Sacker, period-doubling, and saddle-node bifurcations (i.e., $PD_4, PD_5, NS_4, NS_5, PD_6, LP_3, LP_4, PD_7, NS_6$). In codimension-1 bifurcation analysis, we examine the impact of incorporating periodic electromagnetic radiation into the system. As the cosine function exhibits periodic behavior and a graph with a characteristic up-and-down pattern, a similar pattern is observed here.

Further, this study investigates the codimension-2 bifurcation analysis of the map $\mathcal{M}_{r,k}(x,\phi)$ in $(k_0,k_1)$-plane, illustrated in Fig. \ref{N_one}(b). This diagram incorporates codimension-1 bifurcation curves from Fig. \ref{N_one}(a), whereas the green, blue, and magenta curves are associated with Saddle-node $(LP)$, period-doubling $(PD)$, and Neimark-Sacker $(NS)$ bifurcation depicted in Fig. \ref{N_one}(a). For a sufficiently small value of parameter $k_1$, it presents the four general period-doubling (i.e., $GPD_{\{1,2,3,4\}} $) bifurcations with the creation of a closed symmetric structure. One notable thing we observe here is that although the map $\mathcal{M}_{r,k}(x,\phi)$ is asymmetric, it shows a symmetric pattern in terms of the locus of orbit for the four $GPD_{\{1,2,3,4\}}$ here.

As the parameter $k_1$ increases, a cusp point $(CP_1)$ emerges near the parameter $k_0=5$ on the locus of period-doubling bifurcation (i.e., the blue curve) and two cusp points $(CP_1)$ and $(CP_2)$ near the parameter $k_0=-2$ on the locus of saddle-node bifurcation (i.e., the green lines), along with this a fold-flip $(LPPD_1)$ bifurcation is also observed. Further, as the parameter $k_1$ approaches approximately $2.3$, we see two $1:3$ resonances $(R3)$ on the locus of Neimark-Sacker bifurcation (i.e., on the magenta line) and a fold-flip $(LPPD_2)$ bifurcation on the locus of saddle-node bifurcation. Further increase in $ k_1$, a cusp point $(CP_4)$ on the green line, and a $1:3$ resonance $(R3)$ on the magenta line is showcased.  With an further increase in $k_1$, we observe the cusp point $(CP_5)$, and the flip-fold $(LPPD_3)$ and $(LPPD_4)$ bifurcation. The points of $(LPPD_3)$ and $(LPPD_4)$ are the points where the locus of fold and flip bifurcation coincide. As the parameter $k_1$ increases further, a $1:3$ resonance $(R3)$ and $1:4$ resonance $(R4)$ emerge along the magenta line.  Also, a cusp point $(CP_6)$ is depicted on the green curve. Further, increase in $k_1$, the system showcases the $1:4$ resonance $(R4)$ and flip-fold $(LPPD_5)$ and $(LPPD_6)$ on the locus of saddle-node bifurcation.  Lastly, at a higher value of parameter $k_1$, the system shows multiple $1:1$ resonances $(R1)$ and a cusp point $(CP_7)$ for the map $\mathcal{M}_{r,k}(x,\phi)$.

\begin{figure}[ht!] 
    \centering
    \subfloat[Codimension-1]{%
        \includegraphics[width=0.5\textwidth,height=0.4\textwidth]{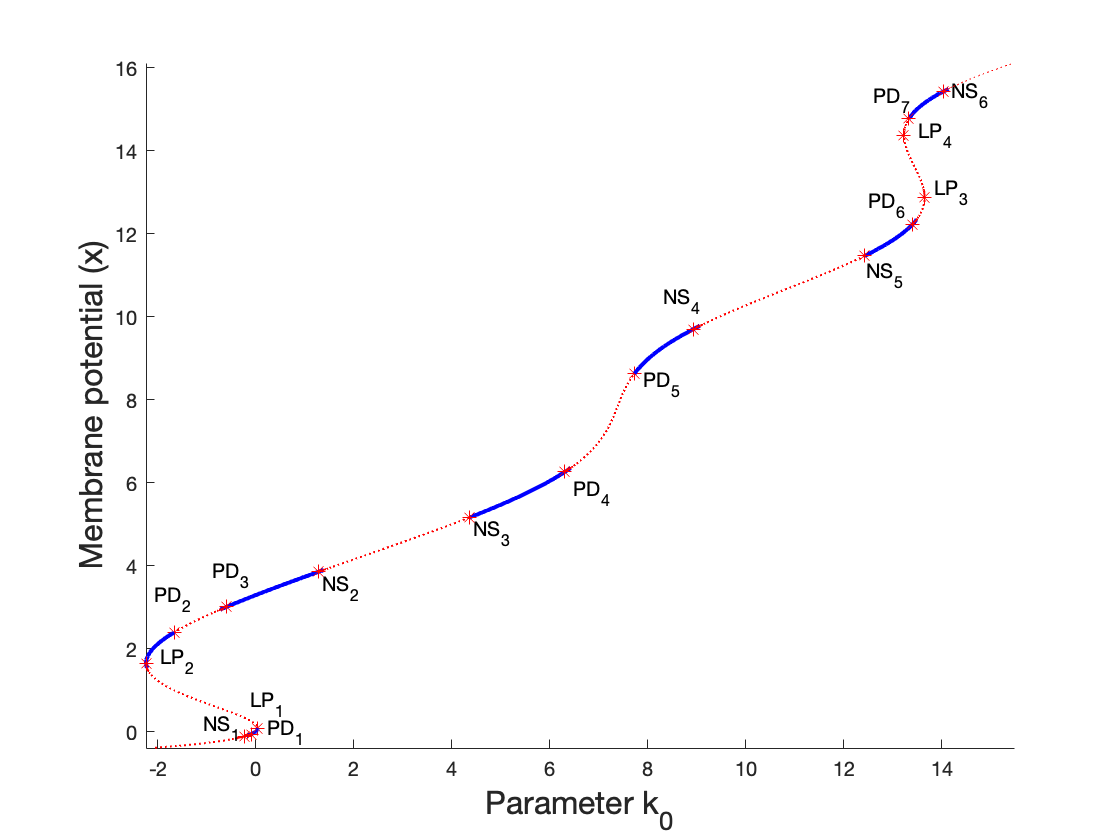}%
        }%
    \hfill%
    \subfloat[Codimension-2]{%
        \includegraphics[width=0.5\textwidth,height=0.4\textwidth]{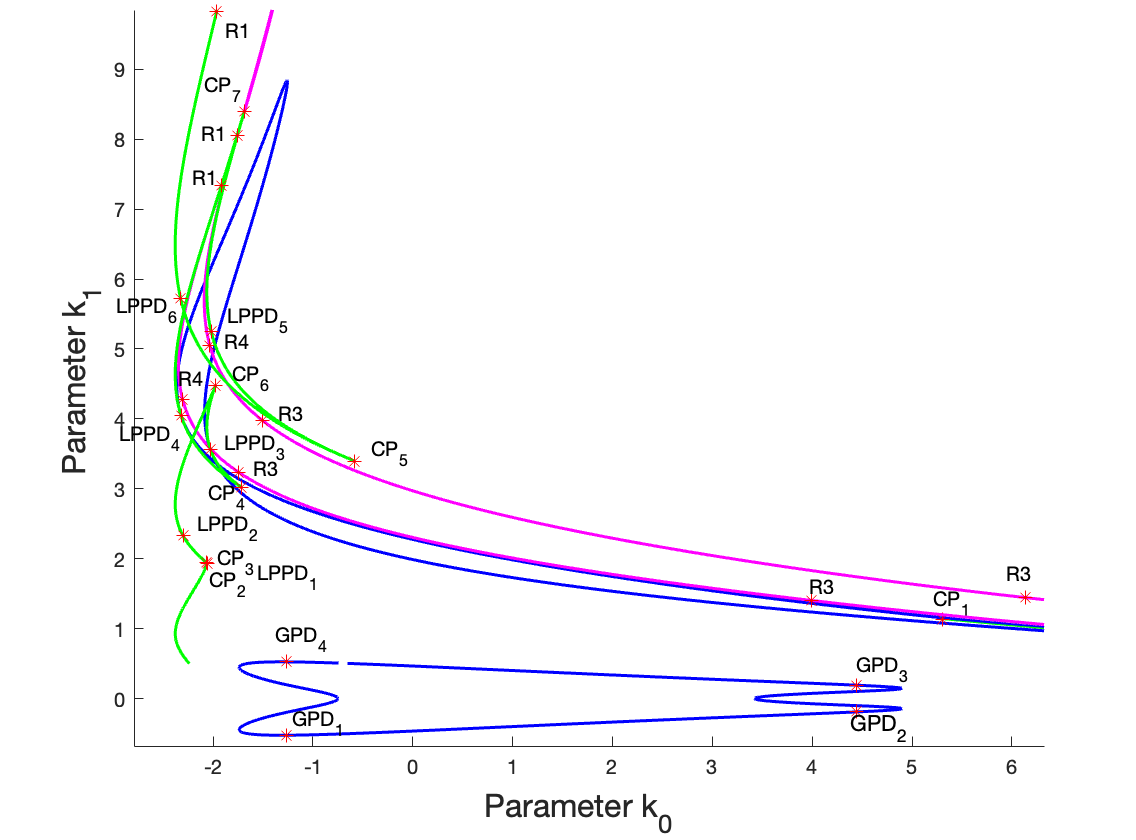}%
        }%
    \caption{(a) shows the codimension-1 bifurcation diagram of the map $\mathcal{M}_{r,k}(x,\phi)$ w.r.to $k_0$. The blue (thick) curve denotes the stable fixed point, while the red (dotted) curve denotes the unstable one. (b) shows the codimension-2 bifurcation diagram in $(k_0,k_1)$-parameter plane, where magenta, blue, and green curves correspond to the loci of the saddle-node (LP), period-doubling $(PD)$, and Neimark-Sacker $(NS)$ bifurcations, respectively.}\label{N_one}
\end{figure}

Now, we discuss the bifurcation study by treating $r$ and $k_2$ as the main parameters of the map $\mathcal{M}_{r,k}(x,\phi)$, as illustrated in Fig. \ref{N-two}. Start with the parameter $r$ for the codimension-1 bifurcation study in the $(r,x)$ plane. The blue (thick) curve represents the stable fixed point, while the red (thin dotted) curve represents the unstable one, as shown in the subfigure \ref{N-two}(a). For a sufficiently small value of the parameter $r$, we witness a single stable fixed point in the system. As $r$ increases, a saddle-node $(LP_2)$ bifurcation is depicted with the normal form of coefficient $2.23286e^{+00}$. Further, an increase of $r$ results in three fixed points: two are stable, and one is unstable. Upon further increase of $r$, we again see a saddle-node $(LP_1)$ bifurcation with the normal form of coefficient $-1.519996e^{-01}$ where an unstable fixed point from the upper side and a stable fixed point from the lower side collapse and disappear, and the system is left with only one stable fixed point.  With a further increase in parameter $r$, a supercritical period-doubling $(PD_1)$ bifurcation with the normal form of coefficient $2.200842e^{-01}$ is observed, leading the system to transform from a stable state to an unstable state. Subsequently, as the parameter $r$ increases, a Neimark-Sacker $(NS_1)$ bifurcation occurs, a neutral saddle.

This analysis is further extended to the codimension-2 bifurcation of the map $\mathcal{M}_{r,k}(x,\phi)$ by varying the parameters $r$ and $k_2$, as shown in Fig. \ref{N-two}(b). The red, black, and blue curves are associated with the saddle-node $(LP)$, Neimark-Sacker $(NS)$, and period-doubling $(PD)$ bifurcation, respectively, which is depicted in a codimension-1 bifurcation study in Fig. \ref{N-two}(a). For the sufficiently higher parameter $k_2$ value, we observe a $1:2$ resonance $(R2)$ on the locus of Neimark-Sacker $(NS)$ bifurcation, which is the black curve.  With a decrease in the parameter $ k_2$, we observe a two-fold-flip (i.e., $LPPD_1, LPPD_2$) bifurcation on the locus of saddle-node bifurcation.  Further, as the parameter $k_2$ decreases, we witness generalized period-doubling $(GPD_7)$ bifurcation on the locus of period-doubling $(PD)$ bifurcation, along with a $1:2$ resonance $(R2)$ bifurcation is depicted at the point of intersection of the loci of period-doubling $(PD)$ and Neimark-Sacker $(NS)$ bifurcation. Further, with a decrease in the parameter  $k_2$, we observe two $1:1$ resonances $(R1)$ and one cusp point $(CP_1)$ on the locus of saddle-node bifurcation.  Lastly, with a decrease in the parameter $k_2$, we obtain a set of flip-fold $(LPPD)$ and generalized period-doubling $(GPD)$ bifurcations.  We zoom into this region and illustrate in Fig. \ref{N-two}(c).  Once again, the impact of using periodic electromagnetic radiation is seen here in the form of the periodic type behavior of the bifurcation.

\begin{figure}[ht!] 
    \centering
    \subfloat[Codimension-1]{%
        \includegraphics[width=0.5\textwidth]{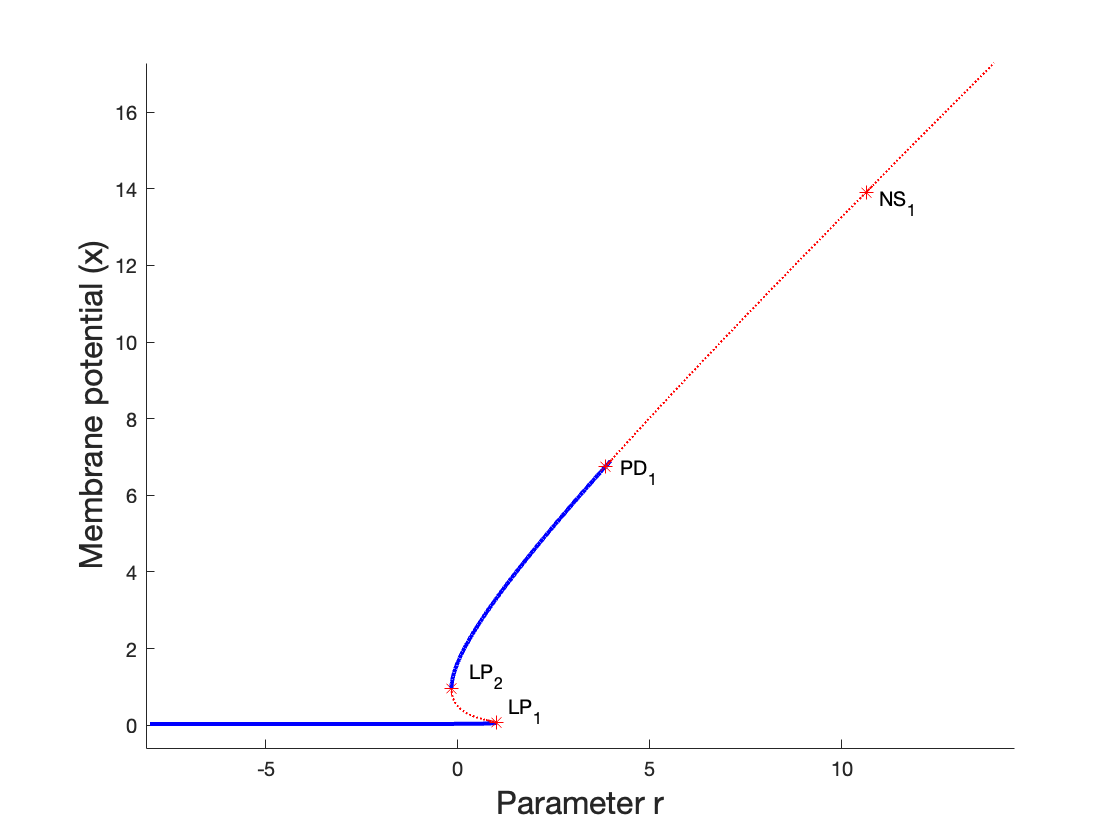}%
        }%
    %\hfill%
    \subfloat[Codimension-2]{%
        \includegraphics[width=0.5\textwidth]{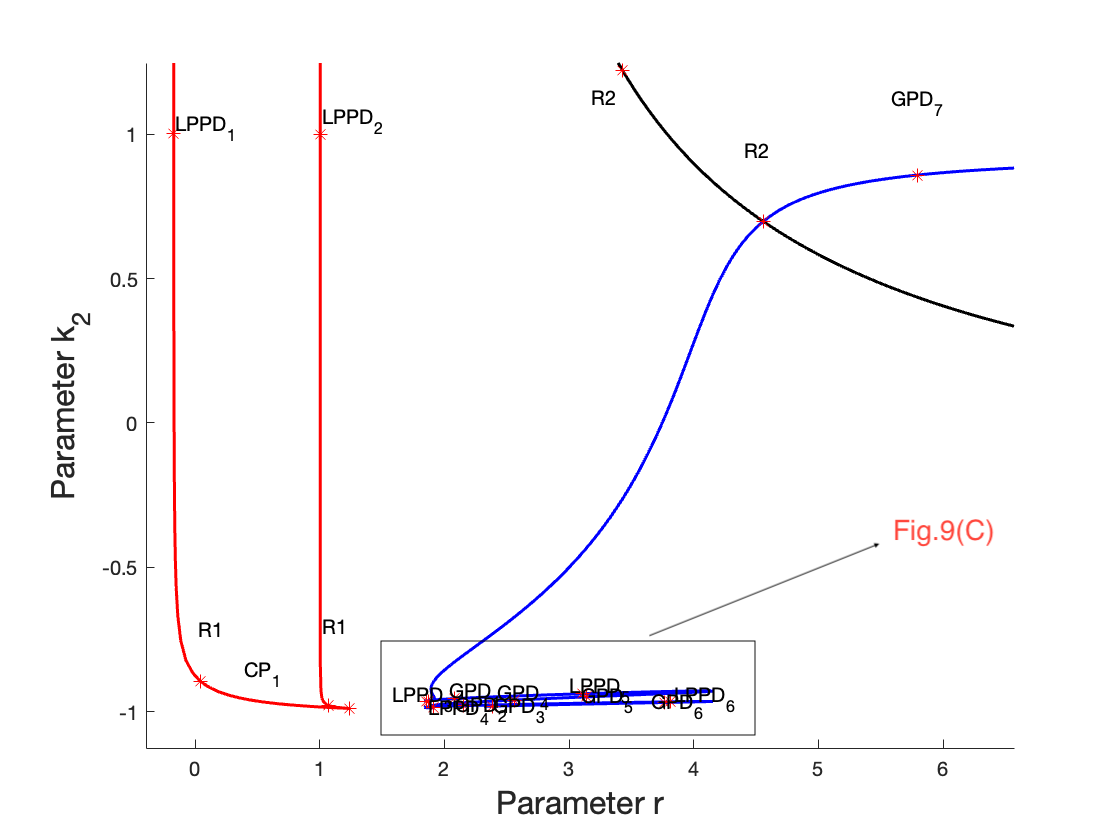}%
        }%
        \vfill%
    \subfloat[Zoom part of panel (b)]{%
        \includegraphics[width=0.5\textwidth]{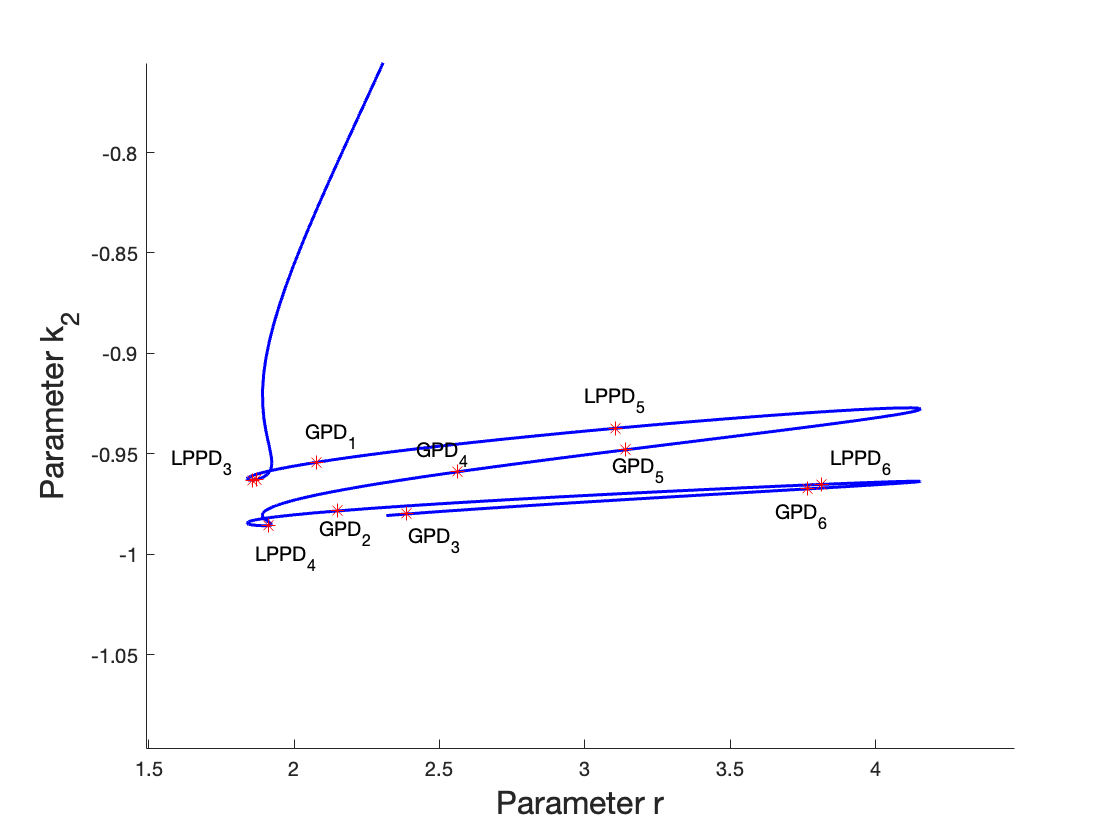}%
        }%
    \caption{In (a), Codimension-1 bifurcation diagram of the map $\mathcal{M}_{r,k}(x,\phi)$ with $r$ as bifurcation parameter. The thick blue curve denotes the stable fixed point, while the dotted red curve indicates the unstable one. In (b), the codimension-2 bifurcation diagram in the $(r,k_2)$ parameter plane, where magenta, blue, and green curves are the loci of the $NS$, $PD$, and  $LP$ bifurcations, respectively. In (c), the zoomed region of 9(b).}\label{N-two}
\end{figure}

% \vspace{-1cm}

%%%--------------------------------------
% \section{Manifolds and Chaotic attractor}

% \begin{figure}[ht!] 
%     \centering
%     \subfloat[Codimension-1]{%
%         \includegraphics[width=0.5\textwidth]{manifold_chaotic.png}%
%         }%
%     %\hfill%
%     \subfloat[codimension-2]{%
%         \includegraphics[width=0.5\textwidth]{manifolds.png}%
%         }%
%         \vfill%
%     \subfloat[codimension-2]{%
%         \includegraphics[width=0.5\textwidth]{stable_chaotic.png}%
%         }%
%     \caption{All the other parameter are set as $k_0=-0.7, k=1, _1=0.1, k_2=0.2, r=2.04$}\label{manifolds}
% \end{figure}  
%%%%--------------------------

\section{Multistability} In this section, we investigate the occurrence of coexisting multiple attractors or multistability phenomena for the map $\mathcal{M}_{r,k}(x,\phi)$. Multistability is an essential feature with important implications for both biological and dynamical systems. Biologically, the ability of a neuron's membrane potential to settle into different stable firing regimes depends on prior activity and input. In that case, this phenomenon is known as multistability, which indicates that neurons naturally evolve and become task-efficient according to need , as our bodies do over time. From a dynamical point of view, multistability indicates the presence of multiple stable invariant sets, such as an attractor, in a nonlinear system. Dynamical systems possess several important properties, one of which is that their long-term behavior is governed by the initial condition; a small difference in initial condition can lead to completely different outcomes.

Here, we discuss the coexistence of a chaotic attractor, a stable limit cycle, and a periodic attractor with the change in parameters $r$ and $k$. By scanning the parameter space, we obtain the coexistence of different attractors corresponding to the different initial conditions, while keeping all the parameters constant (i.e., $k_0=0.06, k_1=0.1, k_2=0.2, k=0.53, $ and $ r=2.7$). Such coexistence of different attractors is illustrated in Fig. \ref{multistability} along with the basin of attraction region corresponding to these attractors. In Fig. \ref{multistability}(a), a coexistence of the chaotic attractor (which is split into three parts) is shown in black color, a period five orbit marked by the red dot, and a stable limit cycle shown in blue color. We have combined all the attractors in a separate figure for clear visualization of the geometric structure of the attractors. Fig. \ref{multistability}(b) shows the basin of attraction region corresponding to the different attractors. We use fine grid points for $(x, \phi)$ to plot the basin of the attraction region, and we evaluate a sufficient number of iterations for each grid point to determine its convergence. A color is assigned to a grid point based on its convergence. The initial points that diverge to infinity are marked with red, those that converge to the chaotic attractor are marked with pink, those that converge to the stable limit cycle are marked with yellow, and those that go to the period-five attractor are assigned the green color. This process creates Fig. \ref{multistability}(b), and a highly intricate basin region plot with the corresponding attractors. Each attractor type is precisely positioned within its region of attraction.

\begin{figure}[ht!] 
    \centering
     \subfloat[Coexistence of multiple attractors]{%
        \includegraphics[width=0.55\textwidth]{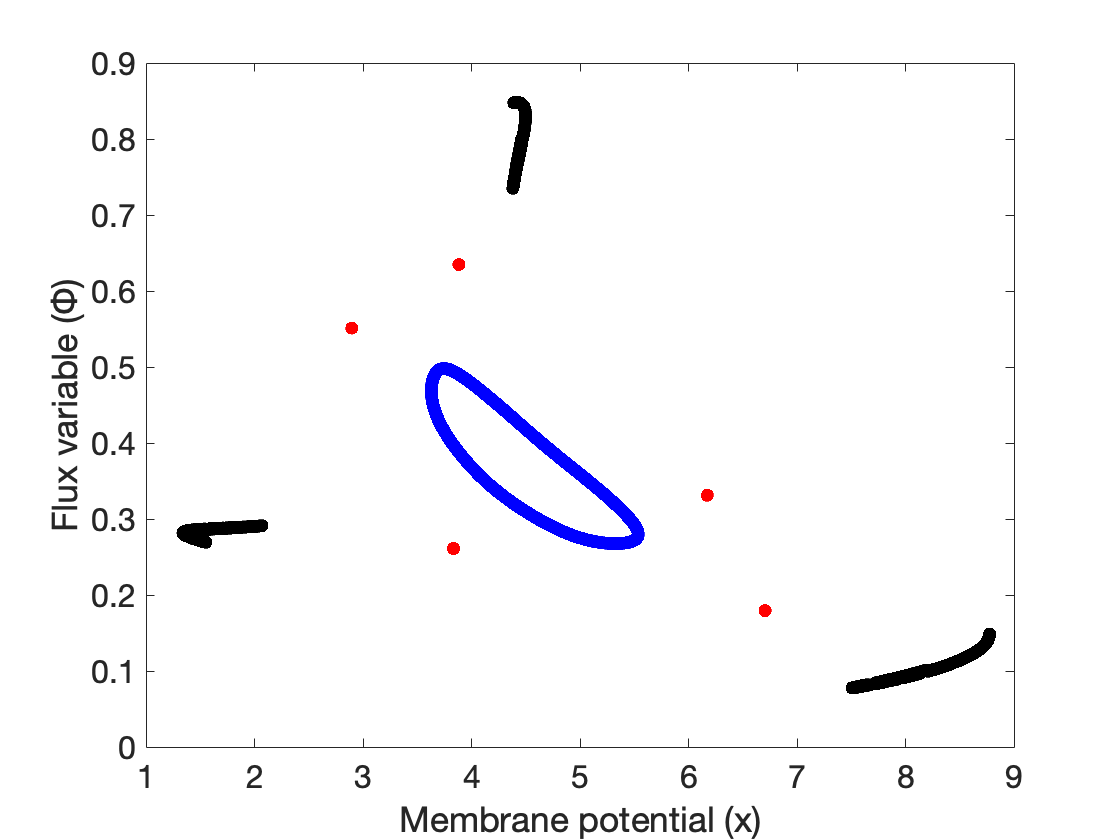}%
        }%
    %\hfill%
    \subfloat[Basin region for coexisted multi-attractors]{%
        \includegraphics[width=0.55\textwidth]{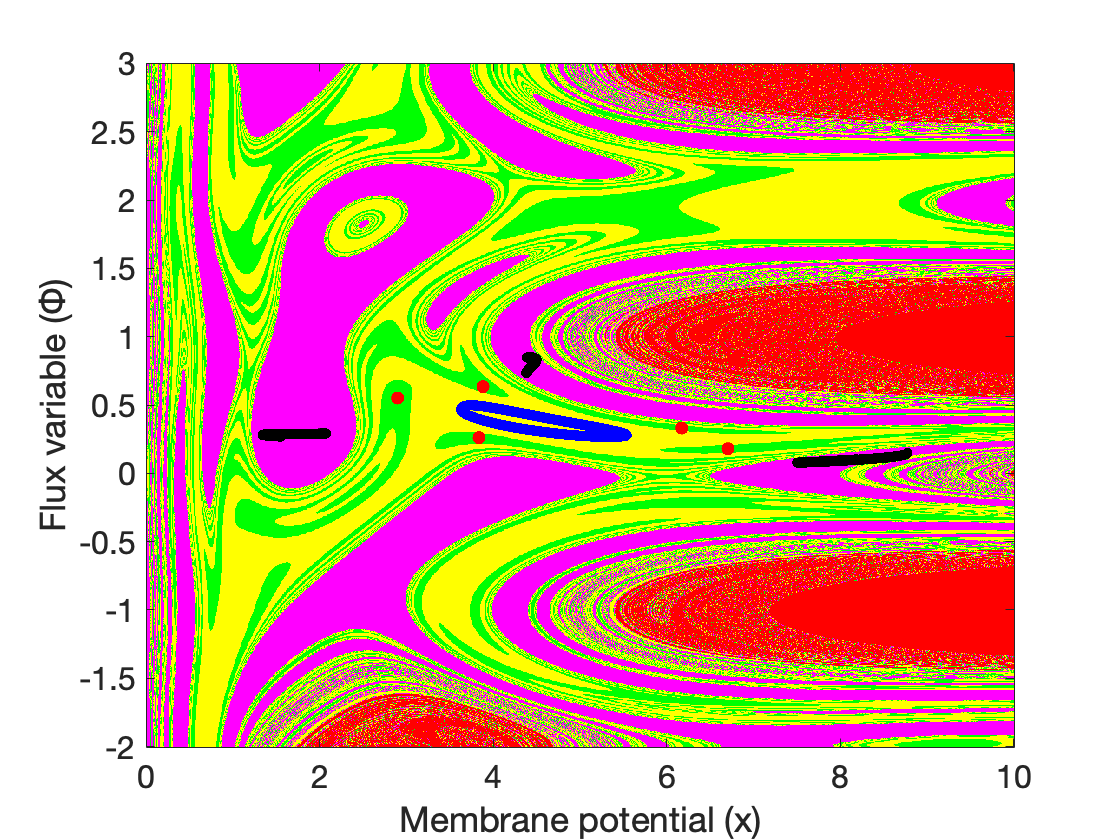}%
        }%
    \caption{Coexistence of stable limit cycle,  period five, and chaotic attractor of the map $\mathcal{M}_{r,k}(x,\phi)$ along with the basin region. (a), The coexistence of a stable limit cycle is shaped like a blue ring, a chaotic attractor in three divided portions in black color, and a period five attractor in a red dot. (b), The basin of the attraction region corresponding to the different attractors. Parameters are considered as: $k_0= 0.06, k_1=0.1.  k_2=0.2, k=0.53, $ and $ r=2.7$
  }\label{multistability}
\end{figure}

To study the complexity of the basin of the attracting region, we zoom in to the different parts of Fig. \ref{multistability}(b) near the chaotic attractor and illustrated in Fig. \ref{c_attractor}. Here, we focus on the chaotic attractor, which is marked in black color, and its basin region, which is indicated by pink color for those initial conditions that converge to the chaotic attractor under the iteration of the map $\mathcal{M}_{r,k}(x,\phi)$. We observe that the chaotic attractor is set properly in its basin of attraction.  As the chaotic attractor splits into different parts, its basin region is also split into many disconnected regions. Apart from the chaotic attractor and its basin region, we observe the basin regions for both the stable limit cycle and the period five cycle as well.

\begin{figure}[ht!] 
    \centering
    \subfloat[]{%
        \includegraphics[width=0.37\textwidth]{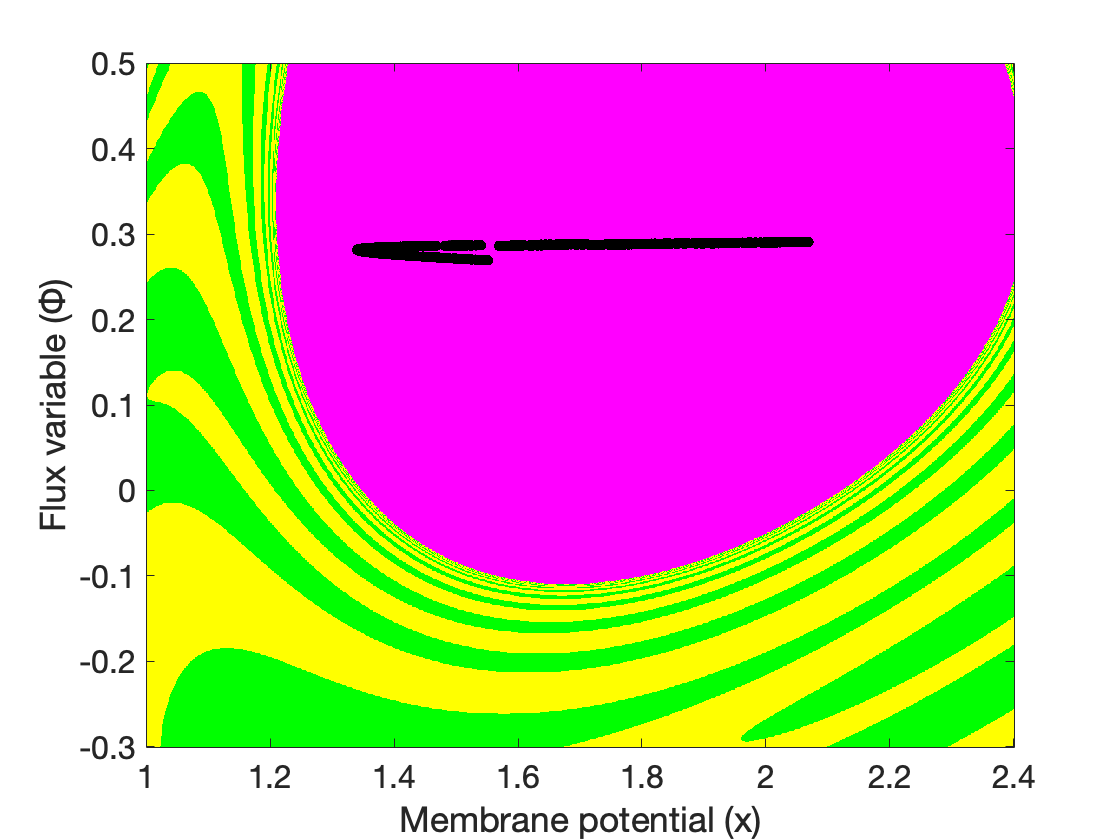}%
        }%
    %\hfill%
    \subfloat[]{%
        \includegraphics[width=0.37\textwidth]{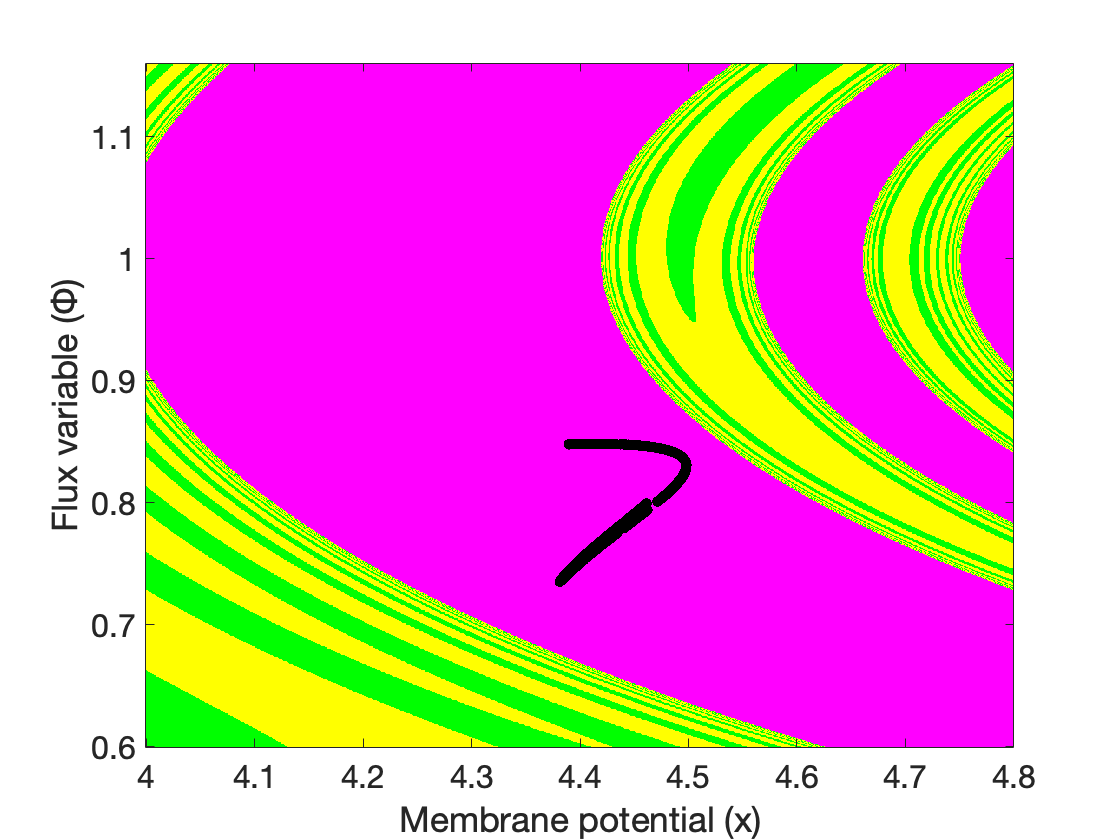}%
        }%
    %\hfill%
    \subfloat[]{%
        \includegraphics[width=0.37\textwidth]{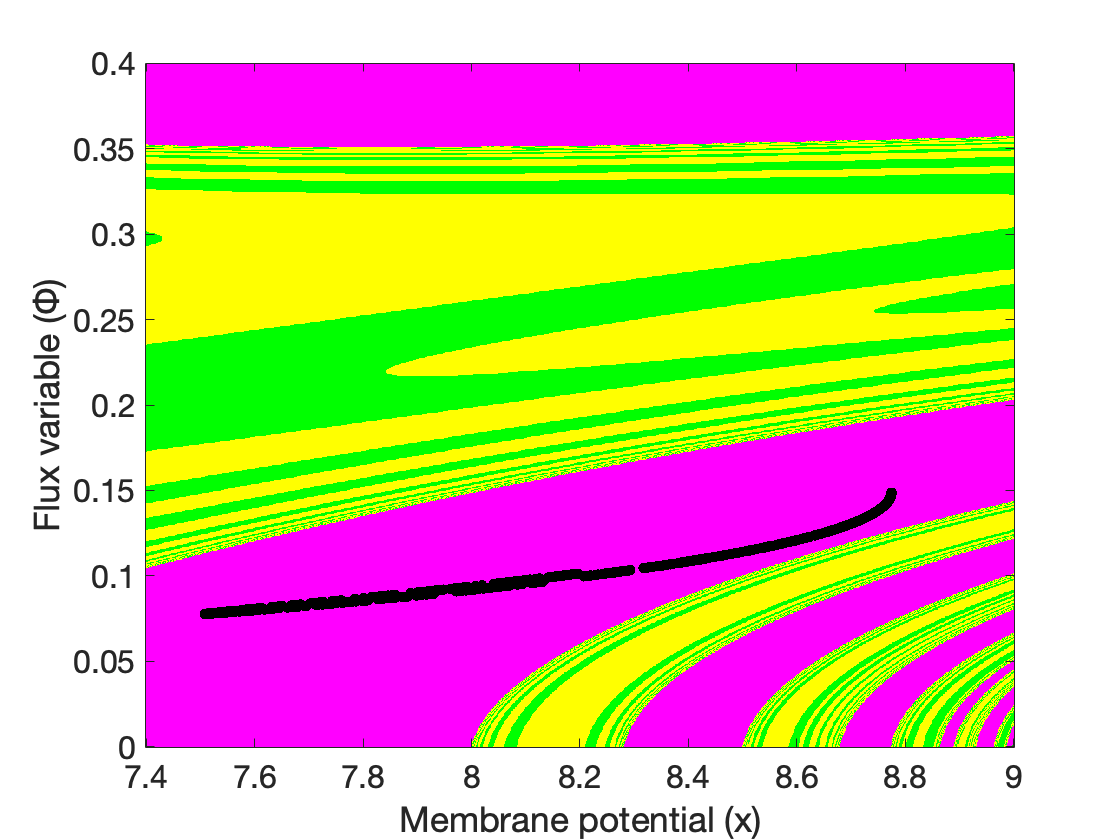}%
        }
    \caption{Illustration of zoomed part of Fig. \ref{multistability}(b) near the chaotic attractor along with its basin of attraction region.}\label{c_attractor}
\end{figure}

Further, we also discuss the zoom parts of Fig.\ref{multistability}(b) near the stable limit cycle and period five attractor, which is shown in Fig.\ref{cycle_five}. In Fig. \ref{cycle_five}(a), a period five is illustrated together with its corresponding basin of attraction, which is shown in green color. The period five attractor is seen surrounding the stable limit cycle and properly situated in its basin of attraction. On the right-top part, we observe the variation in different natures of attractor and their corresponding basin of attraction region, along with the region of escape to infinity. All of this structure is evident in a small region, which highlights the high complexity of the map $\mathcal{M}_{r,k}(x,\phi)$. In Fig. \ref{cycle_five}(b), a stable limit cycle is seen completely sitting in its basin of attraction region, which is shown by the yellow color for those initial conditions that are converging to the limit cycle under the iteration of the map $\mathcal{M}_{r,k}(x,\phi)$. In addition, we observe basins of attraction for both the chaotic attractor and the period-five attractor. This complexity makes the map $\mathcal{M}_{r,k}(x,\phi)$ more prominent to study its dynamical properties.

\begin{figure}[ht!]
    \centering
     \subfloat[Period five attractor with basin region]{%
        \includegraphics[width=0.55\textwidth]{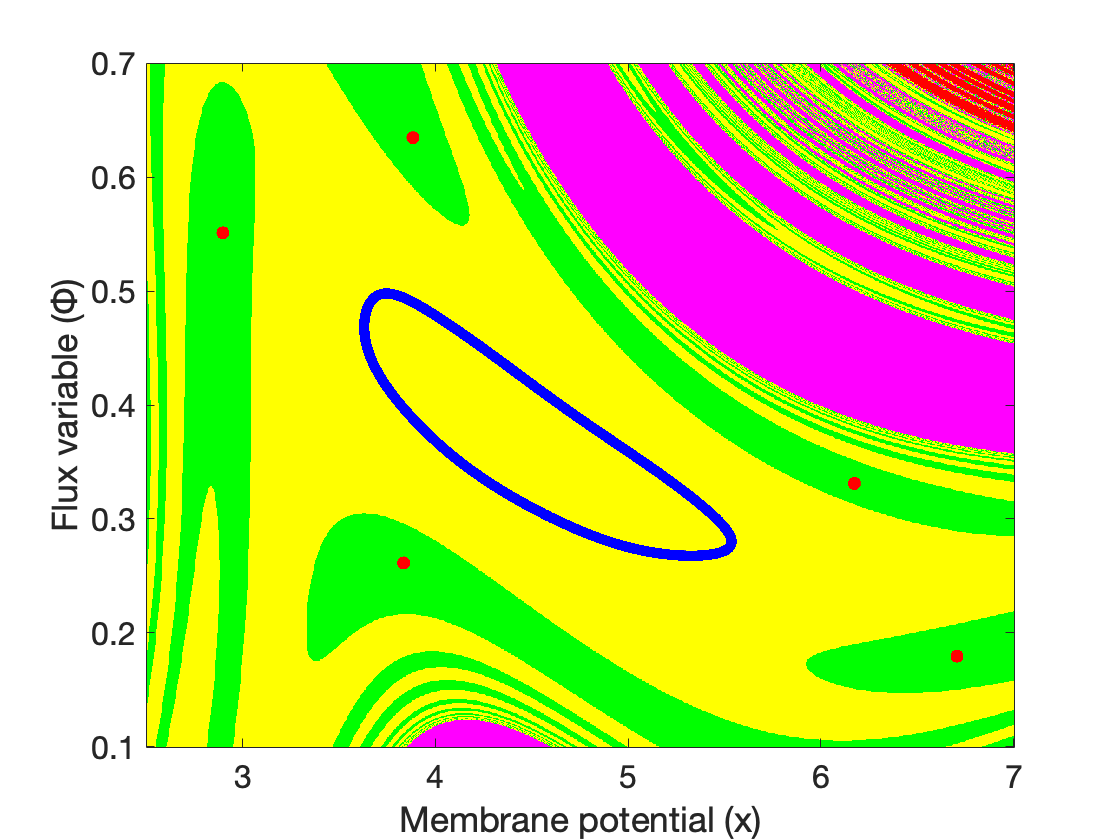}%
        }%
    %\hfill%
    \subfloat[Stable limit cycle with basin region]{%
        \includegraphics[width=0.55\textwidth]{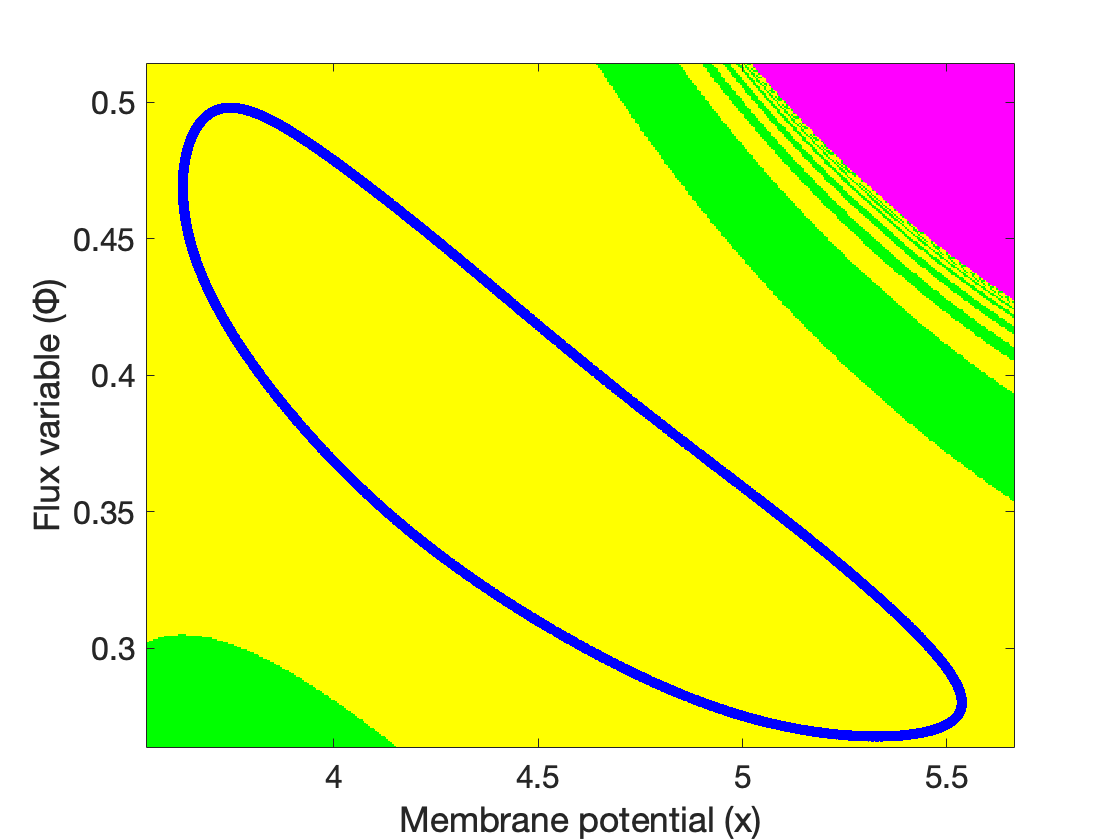}%
        }%
    \caption{(a) Detailed view region near the period-five attractor and its basin of attraction. (b) zoomed region near the stable limit cycle and its basin of attraction in Fig. \ref{multistability}(b).}\label{cycle_five}
\end{figure}

.

%%---------------------------------------------------

\section{Spiking and Bursting}

As stated by IZH in \cite{izhikevich2007dynamical}, a neuron fires a spike because it is near a transition called a bifurcation. In other words, when a neuron's membrane potential is activated, the neuron model exhibits a qualitative change in its dynamical behavior. Therefore, studying the spiking and bursting patterns in neuron maps provides valuable insight into the fundamental ways in which neurons communicate. Neuron maps are used to reproduce these behaviors and help us analyze the underlying nonlinear dynamics that give rise to the firing patterns. Different patterns (e.g., tonic spiking, periodic spiking, chaotic firing) are associated with distinct functional roles in the brain, including sensory processing, memory encoding, and rhythmic activity. Studying these patterns helps us to uncover the qualitative change in the dynamical effect of neuron information processing. In our study, we illustrated different spiking and bursting patterns, which are shown by the map $\mathcal{M}_{r,k}(x,\phi)$, as illustrated in Fig.\ref{spiking}. Different types of firing patterns are obtained by varying the parameters $k$ and $r$, while keeping the other parameter fixed as $k_0=0.1, k_1=0.1, $ and $ k_2=0.2$. 

\begin{figure}[ht!] 
    \centering
    \subfloat[Regular spike, $k=-0.5$, $r=0.2$]{%
        \includegraphics[width=0.5\textwidth,height=0.2\textwidth]{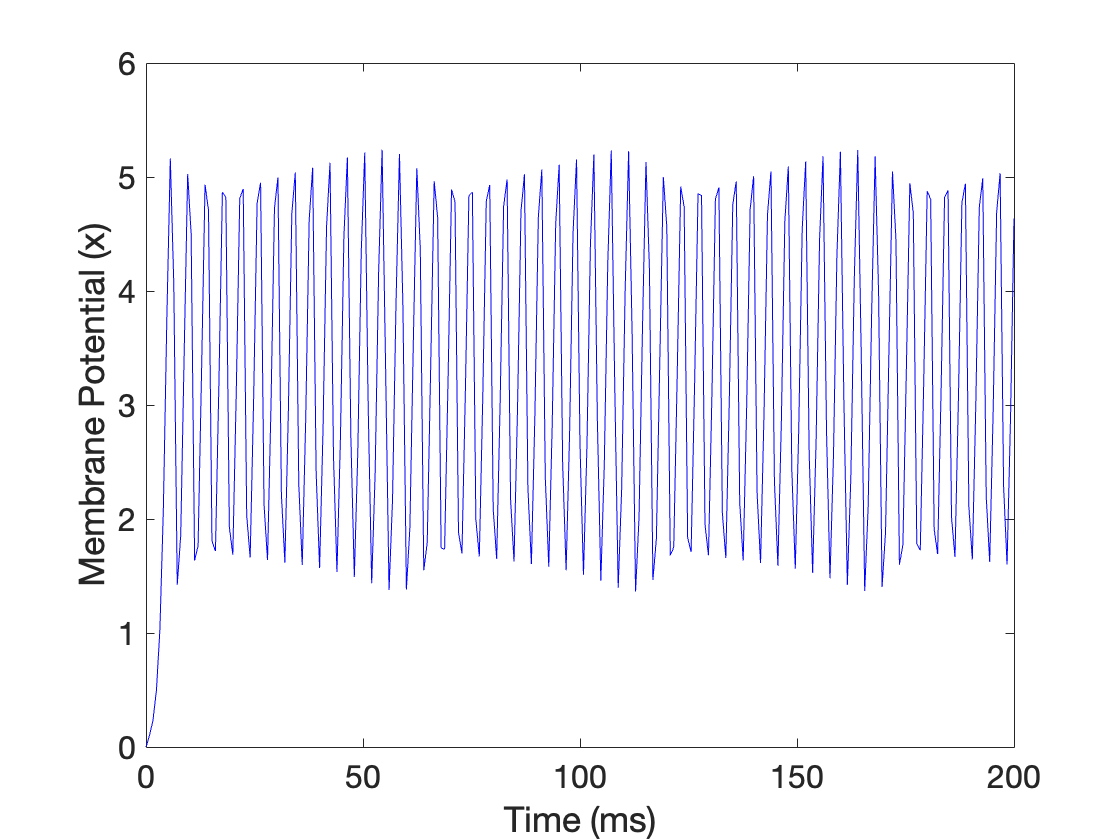}%
        }%
    %\hfill%
    \subfloat[Tonic spike, $k=-0.5$, $r=0.6$]{%
        \includegraphics[width=0.5\textwidth,height=0.2\textwidth]{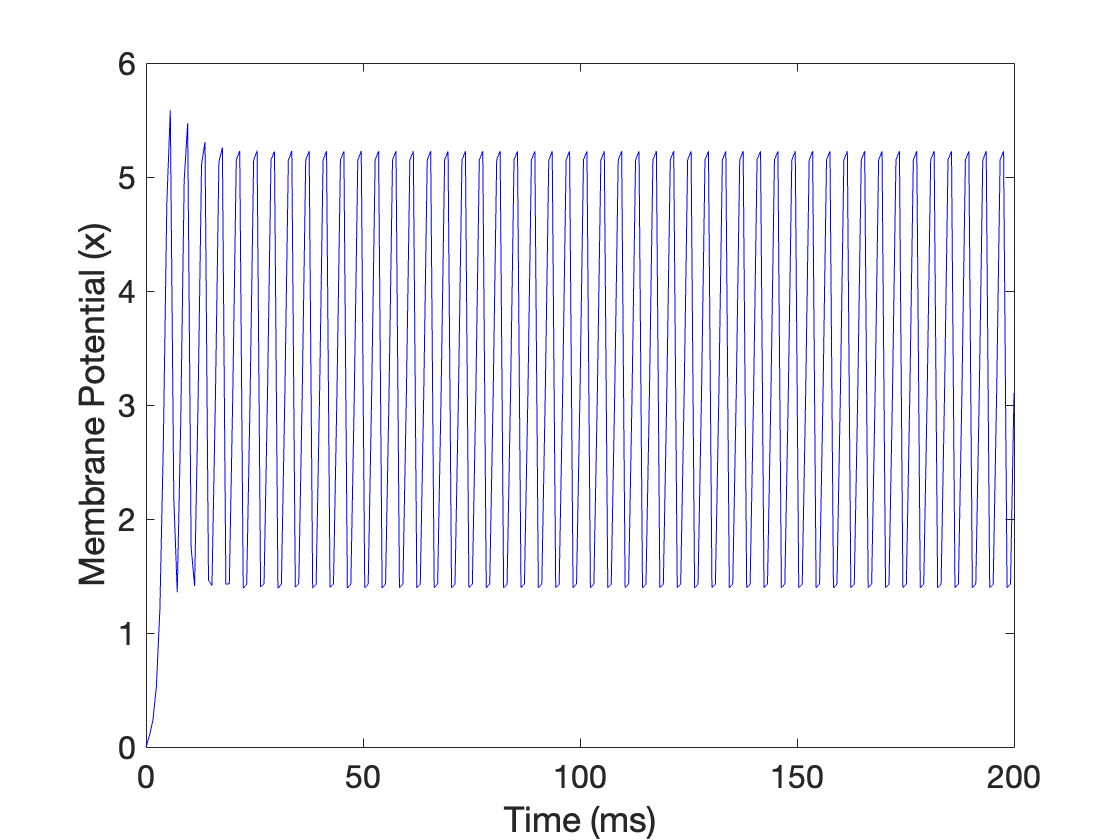}%
        }%
        \vfill%
    \subfloat[Chaotic Bursting, $k=-0.5$, $r=2.2$]{%
        \includegraphics[width=0.5\textwidth,height=0.2\textwidth]{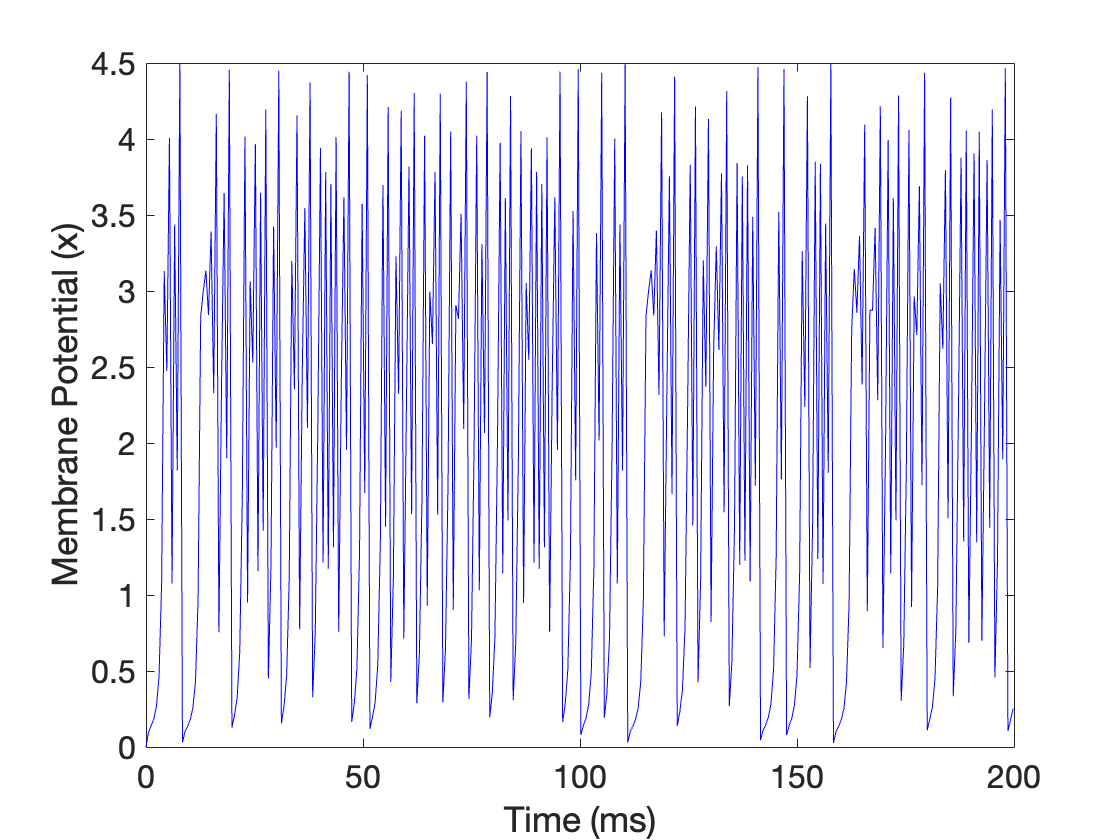}%
        }%
    %\hfill%
    \subfloat[Periodic Bursting, $k=0.1$, $r=3.9$]{%
        \includegraphics[width=0.5\textwidth,height=0.2\textwidth]{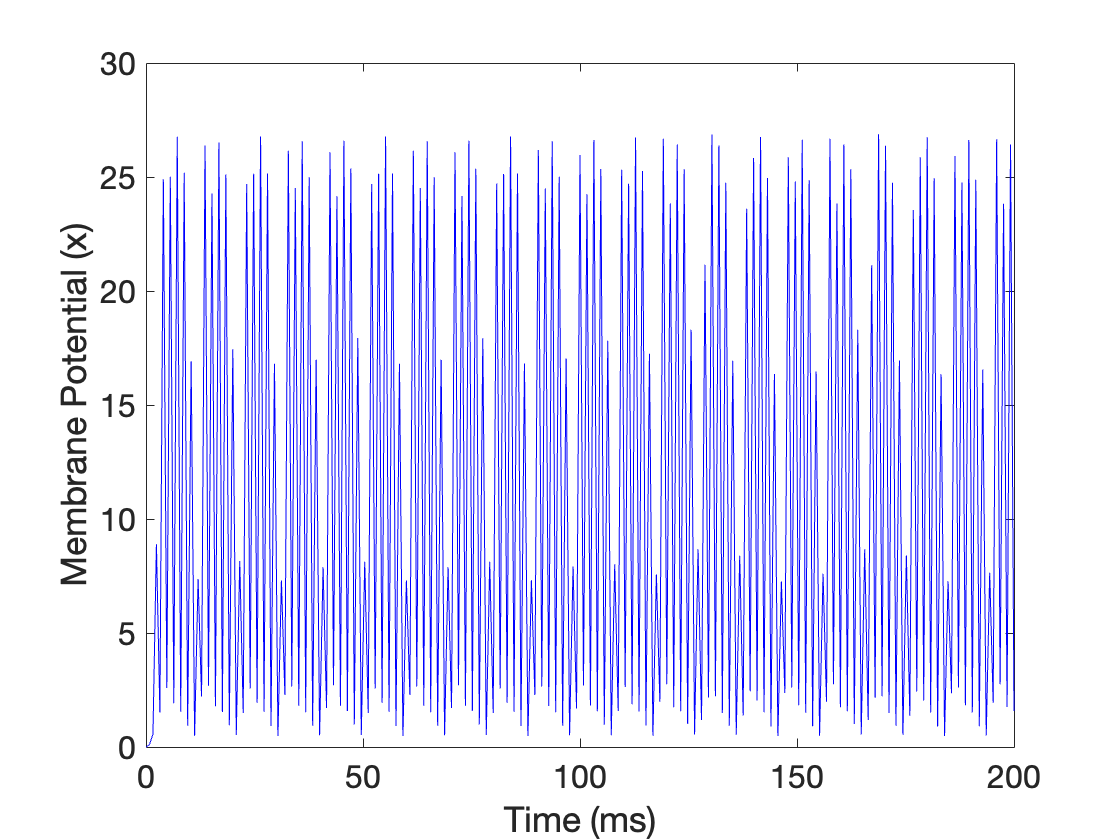}%
        }%
        \vfill
         \subfloat[Phasic Bursting, $k=1$, $r=1.4$]{%
        \includegraphics[width=0.5\textwidth,height=0.2\textwidth]{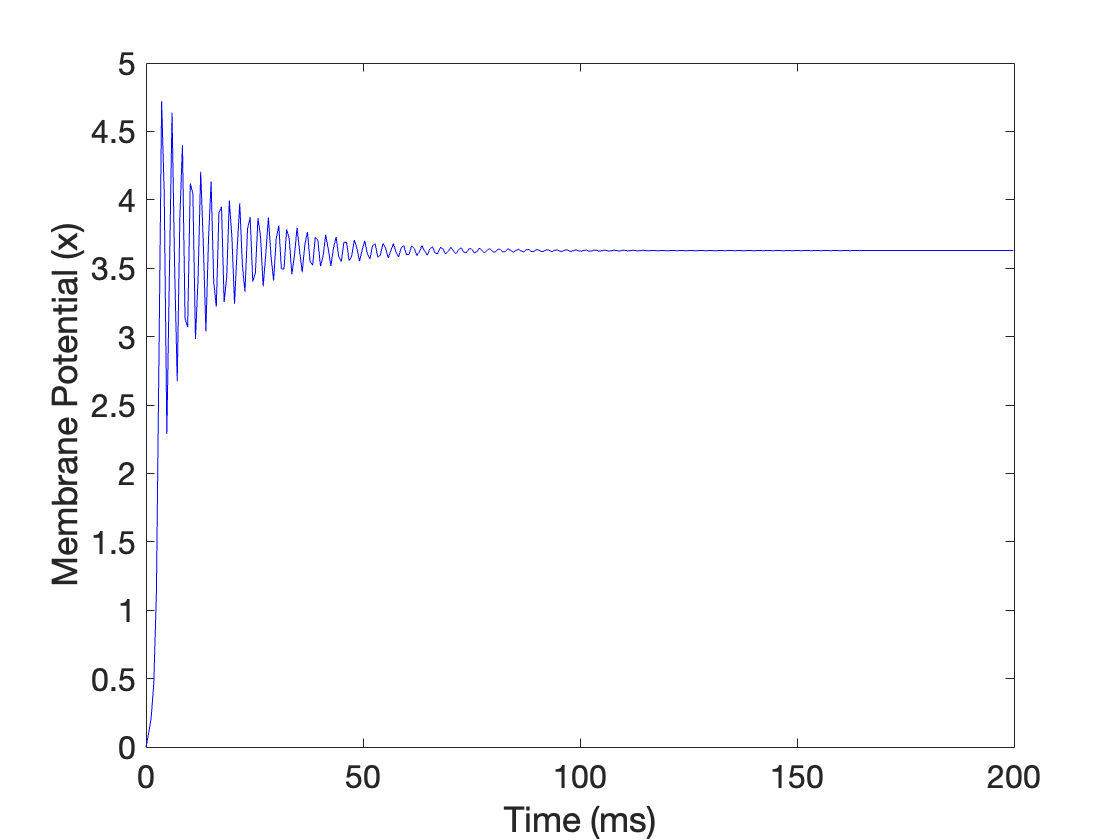}%
        }%
    \caption{Various firing patterns of the map $\mathcal{M}_{r,k}(x,\phi)$ w.r.to the parameters $k$ and $r$. Parameters are considered as: $k_0=0.1, k_1=0.1, $ and $ k_2=0.2$.}\label{spiking}
\end{figure}

\noindent In Fig. \ref{spiking}(a), we witness the regular spiking of the map $\mathcal{M}_{r,k}(x,\phi)$ for the flux parameter $k=-0.5$ and the recovery parameter $r=0.2$. With an increase in the recovery parameter at $r=0.6$ in Fig. \ref{spiking}(b), we observe the tonic spiking patterns followed by some high membrane potential spikes initially.  Furthermore, the bursting patterns are observed in the map by increasing the parameter $r=2.2$, i.e., the system exhibits chaotic bursting behavior for the parameters in Fig. \ref{spiking}(c). With an increase in flux parameter $ k=0.1$ and for the recovery parameter $r=3.9$ in Fig. \ref{spiking}(d), a periodic bursting pattern of the map is observed. Finally, with a further increase in the flux parameter $ k=1$, and for the recovery parameter $r=1.4$ in Fig. \ref{spiking}(e), the system exhibits phasic bursting patterns. 

Hence, different firing patterns for the map $\mathcal{M}_{r,k}(x,\phi)$ can be achieved by varying the parameters $k$ and $r$. Initially, we observe regular spiking patterns; as we continuously increase the parameters, we observe chaotic bursting, then periodic bursting, and finally return to phasic bursting.

%%-----------------------------------------

\section{Chaotic attractor and correlation dim}\label{chaotic_attractor}

In this section, we examine the chaotic attractor of the map $\mathcal{M}_{r,k}(x,\phi)$ with respect to the parameter $k$. Although chaotic behavior appears to be unpredictable, it has a fractal geometric structure, indicating that the system's evolution is bounded and governed by a complex yet deterministic rule, rather than pure randomness. To capture the intrinsic unpredictability of systems sensitive to initial conditions, the study of chaotic attractors is particularly helpful. In this study, we pick different discrete values of flux coupling parameter $k$ and obtain the chaotic attractors of the map $\mathcal{M}_{r,k}(x,\phi)$, as shown in Fig.\eqref{chaotic_attractors}. For $k=1.61$, the system exhibits the initial stage of the chaotic attractor. As the parameter $k$ increases, the chaotic attractor also evolves and becomes bigger and denser. For the parameter $k=1.748$, we observe the fully evolved form of the chaotic attractor for the map $\mathcal{M}_{r,k}(x,\phi)$.

The map exhibits a bounded but non-periodic trajectories, confirming the existence of a compact invariant set with sensitive dependence on initial conditions. Unlike purely stochastic signals, the generated attractors posses a similar fractal geometry, which we verify visually and quantify using the correlation dimension. It is a technique that is used to determine the chaotic attractor's fractional dimension, with respect to the fine range of flux parameter $k$. The calculated correlation dimension is illustrated in Fig. \ref {corr_dim}. For the same discrete values of parameter $ k$ as in Fig. \ref{chaotic_attractors}, the calculated correlation dimension is tabulated in Table \ref{Table_correlation}. This analysis provides the quantitative evidence that the attractor complexity increase in magnitude relative to parameter variation, rather than collapsing and escaping. The chaotic attractors present in the study resemble the shape of a balloon, so we refer them as \textit{Balloon chaotic attractors}.

\begin{figure}[ht!] 
    \centering
    \subfloat[$k=1.61$]{%
        \includegraphics[width=0.5\textwidth,height=0.3\textwidth]{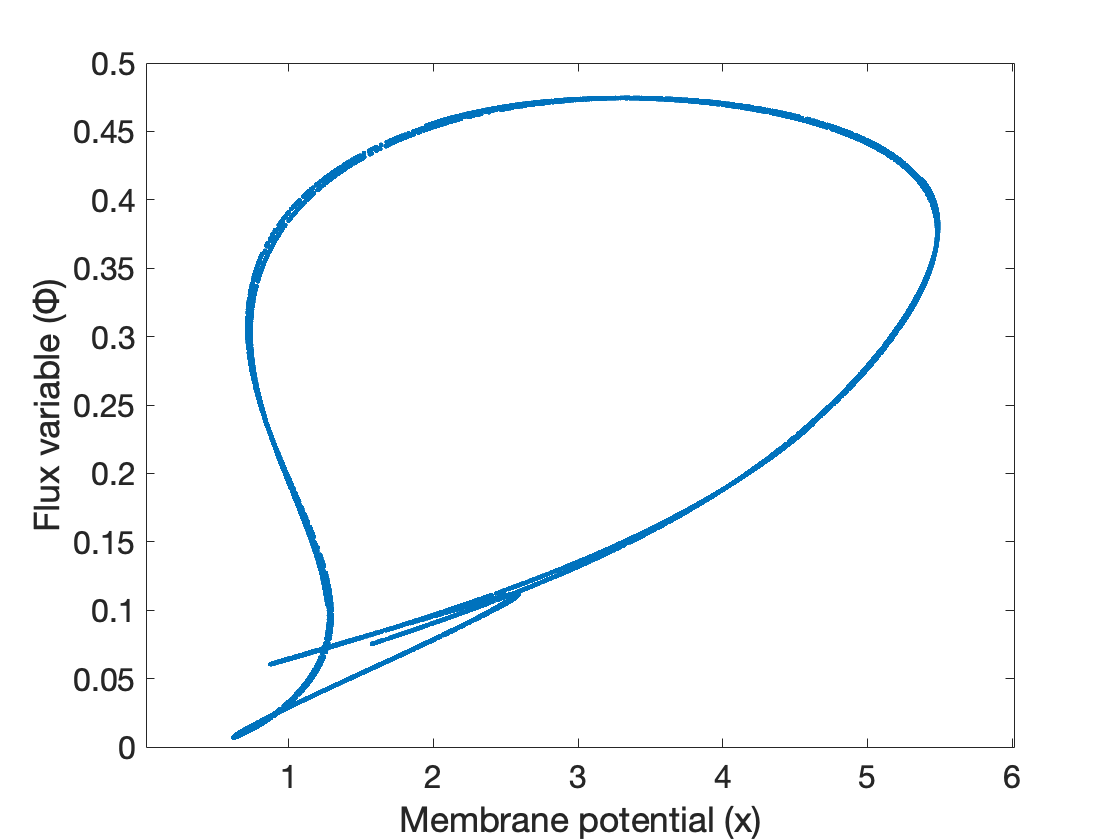}%
        }%
    %\hfill%
    \subfloat[$k=1.62$]{%
        \includegraphics[width=0.5\textwidth,height=0.3\textwidth]{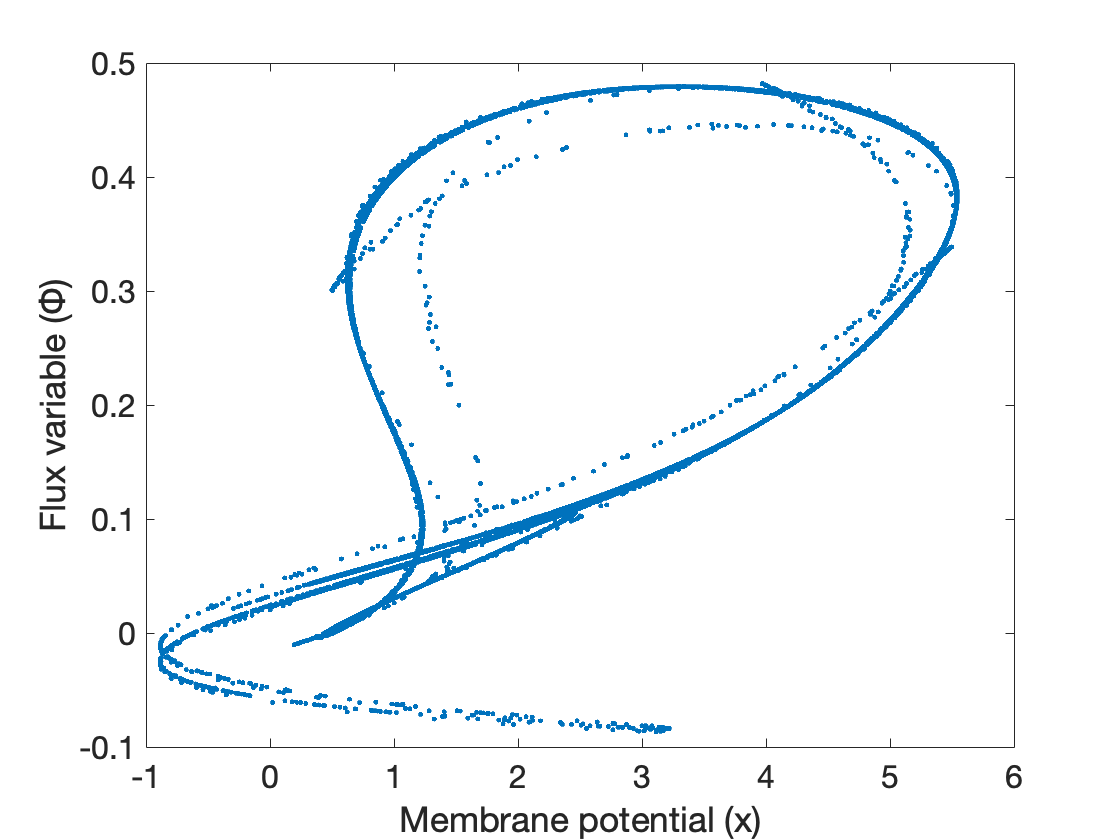}%
        }%
        \vfill
         \subfloat[$k=1.698$]{%
        \includegraphics[width=0.5\textwidth,height=0.3\textwidth]{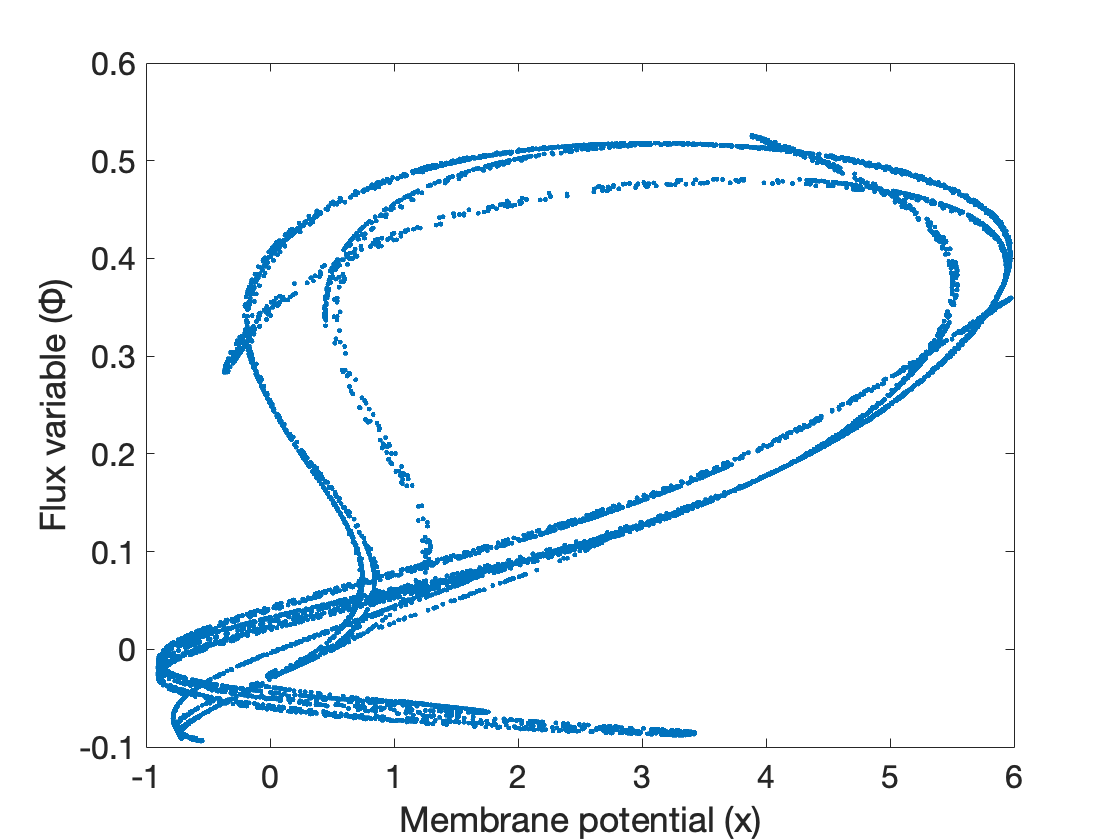}%
        }%
    %\hfill%
    \subfloat[$k=1.748$]{%
        \includegraphics[width=0.5\textwidth,height=0.3\textwidth]{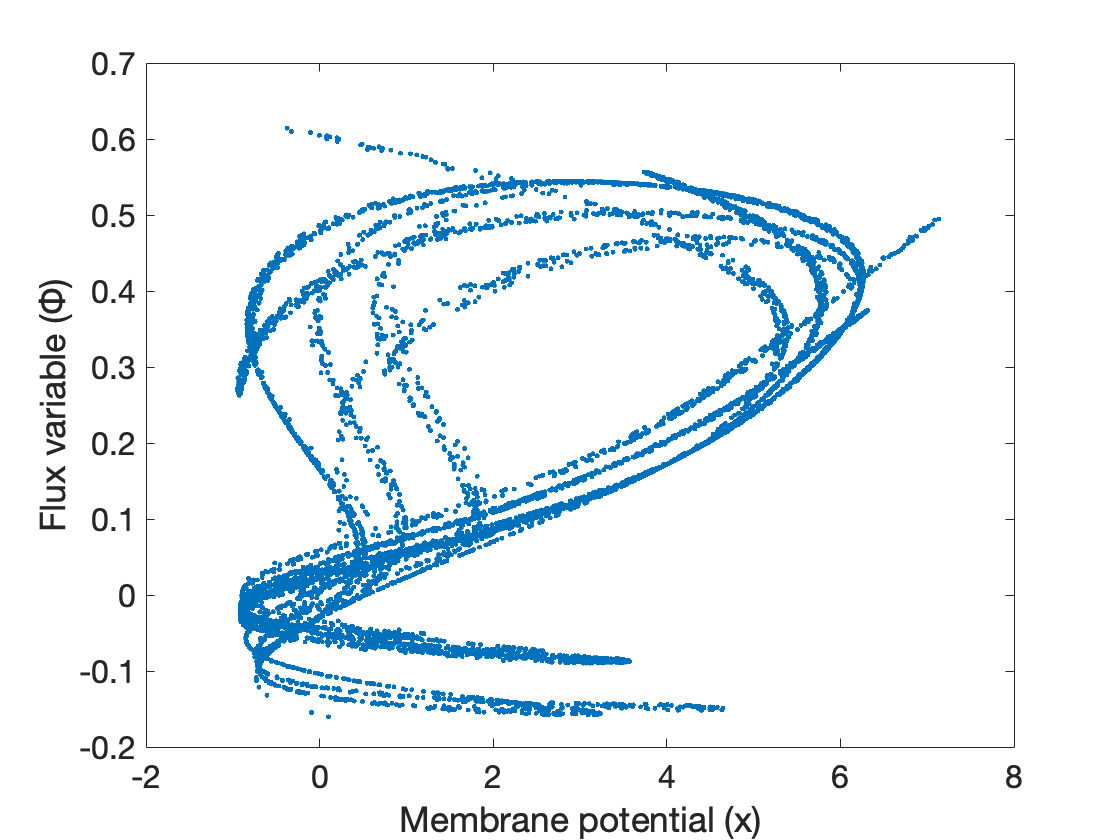}%
        }%
    \caption{Chaotic attractor of the map $\mathcal{M}_{r,k}(x,\phi)$ and its evaluation in four stages are shown.  Parameters are considered as: $k_0=-0.7, k_1=0.1, k_2=0.2$ and $ r=1.03$.}\label{chaotic_attractors}
\end{figure}

\begin{figure}[ht!] 
    \centering
    \subfloat[]{%
        \includegraphics[width=0.7\textwidth,height=0.4\textwidth]{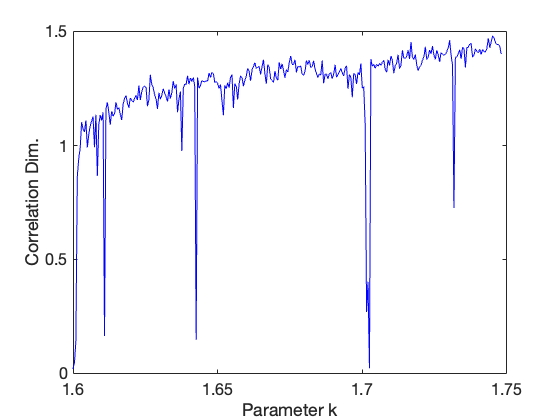}%
        }%
    \caption{Correlation dimension of the map $\mathcal{M}_{r,k}(x,\phi)$ with respect to the parameter $k$. Parameters are considered as: $k_0=-0.7, k_1=0.1, k_2=0.2,$ and $ r=1.03$.}\label{corr_dim}
\end{figure}

% \vspace{-5cm}
\begin{table}[ht!]
\centering
\caption{Corr. dim. of the map $\mathcal{M}_{r,k}(x,\phi)$ for discrete values of parameter $k$}
\vline\begin{tabular}{||c||c||}
\hline
\textbf{Parameter $k$} & \textbf{Correlation Dim.}\\
\hline
$1.61$ & $1.1705$ \\
 \hline
$1.62$ & $1.1882$\\ 
\hline
$1.698$ &  1.3310 \\ 
\hline
$1.748$ & 1.4358  \\ 
\hline
\end{tabular}\vline\label{Table_correlation}
\end{table}
 % \vspace{-5cm}

%%----------------------------
% \newpage 

\section{Ring-star network } After examining the dynamics of a single neuron under the electromagnetic radiation, we extend our analysis to a network of neurons. Recent studies illustrate the coexistence of both states, meaning the coexistence of synchronized and unsynchronized states in a network of neurons, as well as different spatiotemporal patterns \cite{muni2022discrete,ashwin2025global,radushev2026metric}. Investigating the network behavior of neurons is crucial for understanding various types of spatiotemporal patterns, including synchronized, unsynchronized, chimera states, and traveling waves. A network of neurons using a ring-star setup, which is governed by the map $\mathcal{M}_{r,k}(x,\phi)$, is shown in Fig. \ref{circle}.

\begin{figure}[ht!] 
    \centering
    \subfloat[]{%
        \includegraphics[width=0.6\textwidth]{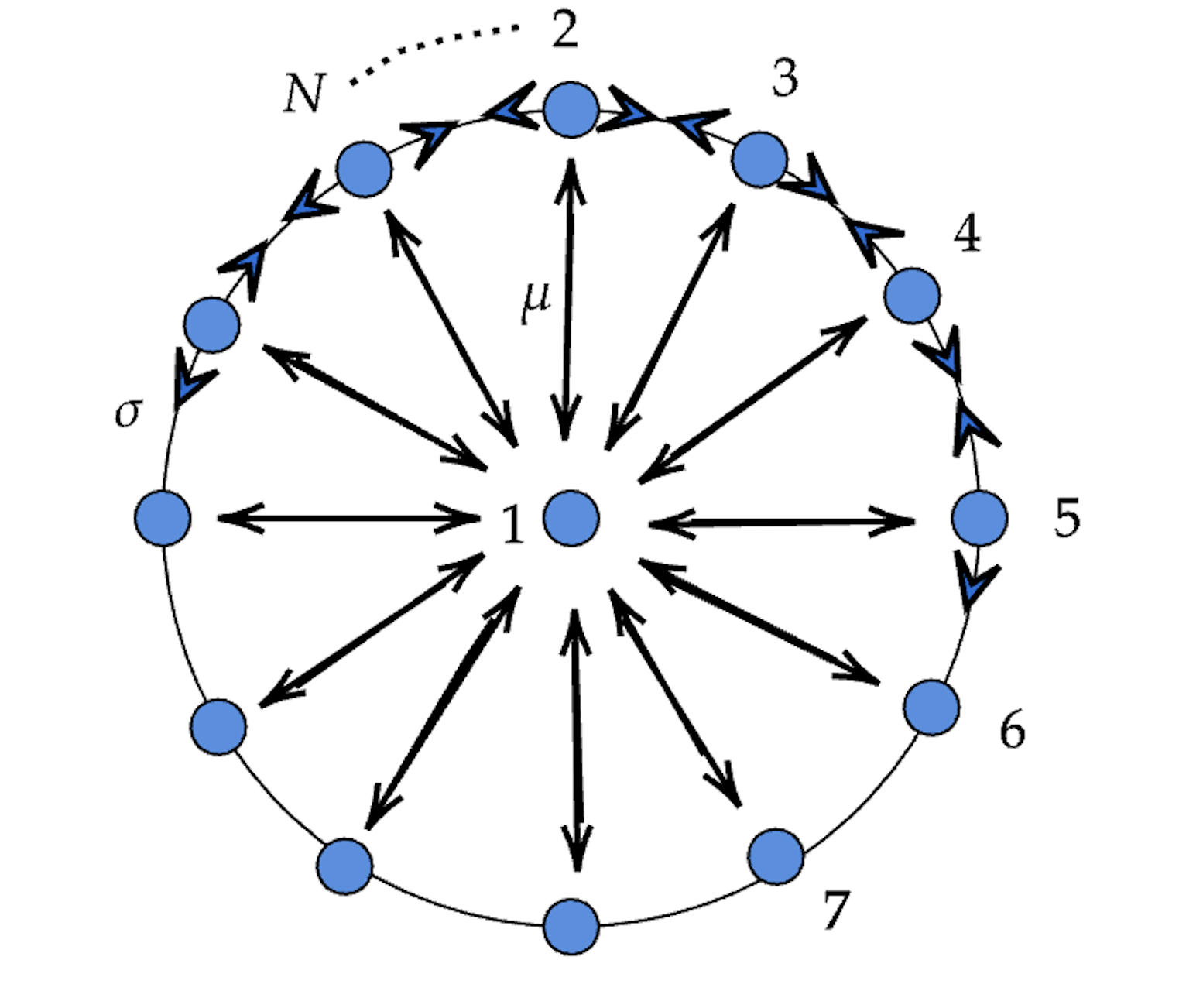}%
        }%
    \caption{A structure diagram of a neuron's network under the ring-star topology, governed by the $\mathcal{M}_{r,k}(x,\phi)$. }\label{circle}
\end{figure}

\noindent The parameters $\sigma$ and $\mu$ represent the coupling strength of the ring and star configuration, respectively, and each node $m=1,2,..., N$ in this network is the map $\mathcal{M}_{r,k}(x,\phi)$. The double-sided arrow shows the bidirectional connection between neurons in the ring-star network. The mathematical formation of this ring-star neuron's network is as follows:
\begin{align}
   \begin{split}
    x_{m}(n+1)&=x_m^2(n)e^{\left(r-x_m(n)\right)}+k_0+kx_m(n)\;\{\cos(\pi\phi_m(n))\}+\mu[x_m(n)-x_1(n)]\\
    &+\frac{\sigma}{2R}\sum_{i=m-R}^{m+R}[x_i(n)-x_m(n)],\\
    \phi_{m}(n+1)&=k_1x_{m}(n)-k_2\phi_{m}(n).
   \end{split}\label{general}
\end{align}

\noindent The equation of the central node is obtained from the above equation as follows:
\begin{align}\label{central}
\begin{split}
    x_{1}(n+1)&=x_1^2(n)e^{\left(r-x_1(n)\right)}+k_0+kx_1(n)\{\cos(\pi\phi_1(n))\}+\mu\sum_{i=1}^{N}[x_i(n)-x_{1}(n)],\\
    \phi_{1}(n+1)&=k_1x_{1}(n)-k_2\phi_{1}(n).
    \end{split}
\end{align}

\noindent With the boundary conditions as:
\begin{align}\label{condition}
\begin{split}
x_{m+N}\;(n)&=x_{m}(n),\\
\phi_{m+N}\;(n)&=\phi_{m}(n).
\end{split}
\end{align}

\noindent In the network equations, the parameter $\mu$ denotes the coupling strength of the star network, while $\sigma$ represents the coupling strength of the ring network. The symbol $R$ denotes the number of neighbor nodes, and the network size considered for the simulations is $ N = 100$. For this analysis, we consider the three cases by varying the coupling parameters. In the first case for the ring network, the star-network coupling strength is considered zero (i.e., $\mu=0$), and the parameter $\sigma$ is the control parameter in this setup. In the second case, we use the nonzero coupling parameter to investigate the ring-star behavior, where both coupling parameters are decisive factors. Lastly, we take the ring coupling parameter zero (i.e., $\sigma=0$) to study the star network in which $\mu$ is the central parameter of the network. We set the coupling range $R=10$ throughout this paper and consider the random initial condition for the network study. In each figure of this section, each figure present three complementary view of network: $(i)$ rightmost shows the change in membrane potential of neurons over the time (i.e., how the current in membrane potential is change with time), $(ii)$ the middle one displays the final membrane potential values of all neuron in the network (those value of membrane potential which we find at the end time state from rightmost), $(iii)$ the leftmost shows the relation between nodes to each other (i.e., the relation between the end-state membrane potential value shown in middle one).

Note that a color bar is provided in the leftmost and rightmost subfigures in this section. On the leftmost, the color bar indicates the absolute difference between the neuron's end-state membrane potential, while the rightmost indicates the absolute membrane potential for each neuron over time. Moreover, we use a random initial condition in the network study of the map $\mathcal{M}_{r,k}(x,\phi)$ and observe that the system exhibits non-trivial and more complex behavior.

% \newpage   
\subsection{Ring Network} 

\begin{figure}[ht!] 
    \centering
    \subfloat[]{%
        \includegraphics[width=0.37\textwidth]{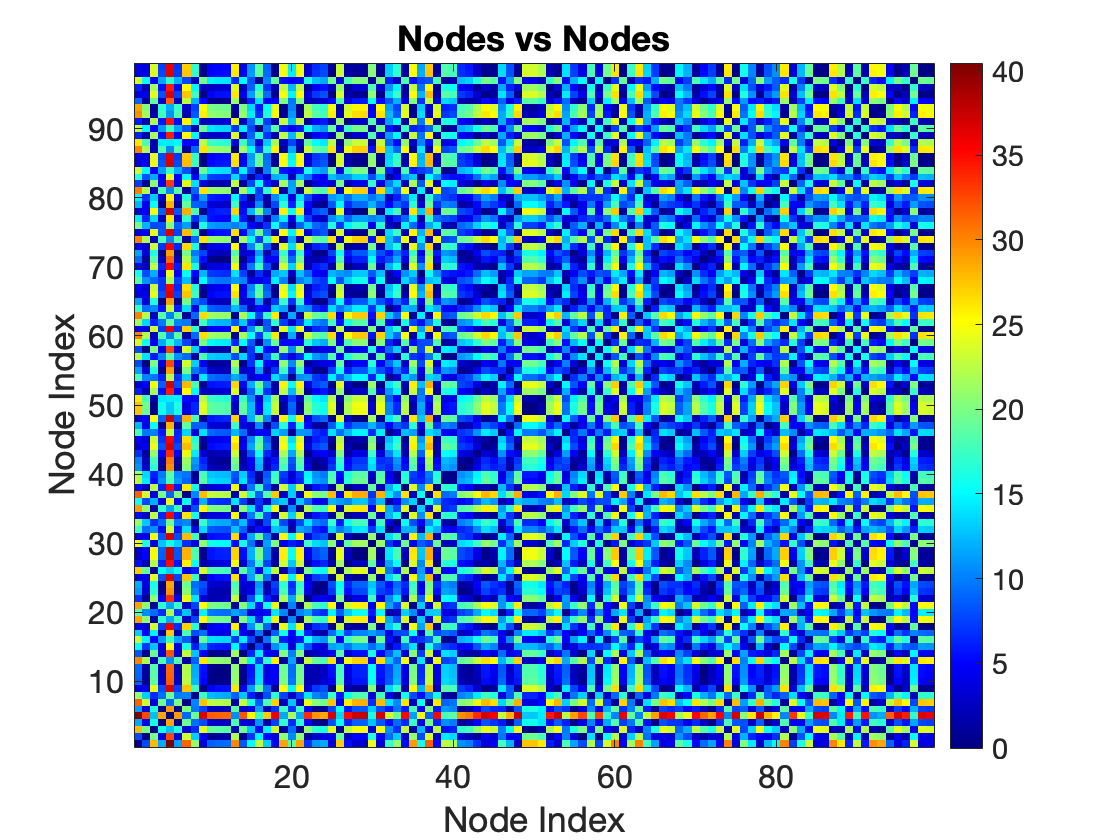}%
        }%
    %\hfill%
    \subfloat[\tiny{Unsynchronized state, $\sigma$=0.01}]{%
        \includegraphics[width=0.37\textwidth]{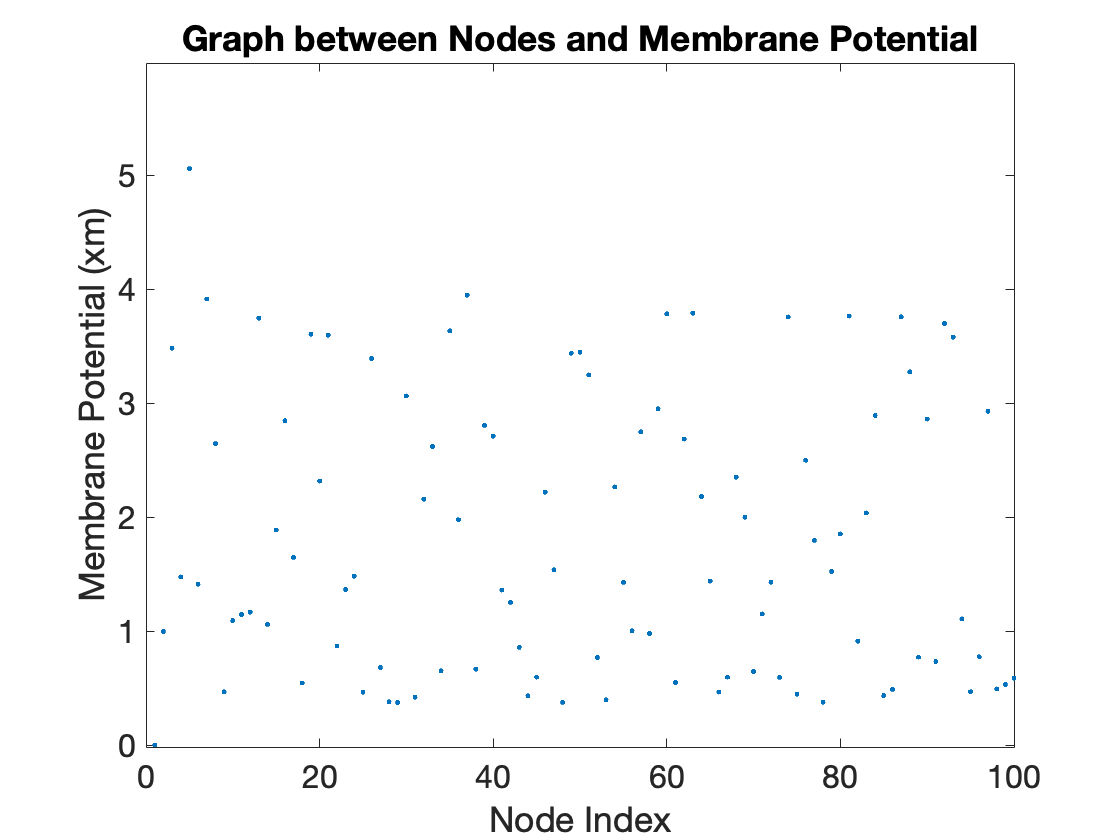}%
        }%
    %\hfill%
    \subfloat[]{%
        \includegraphics[width=0.37\textwidth]{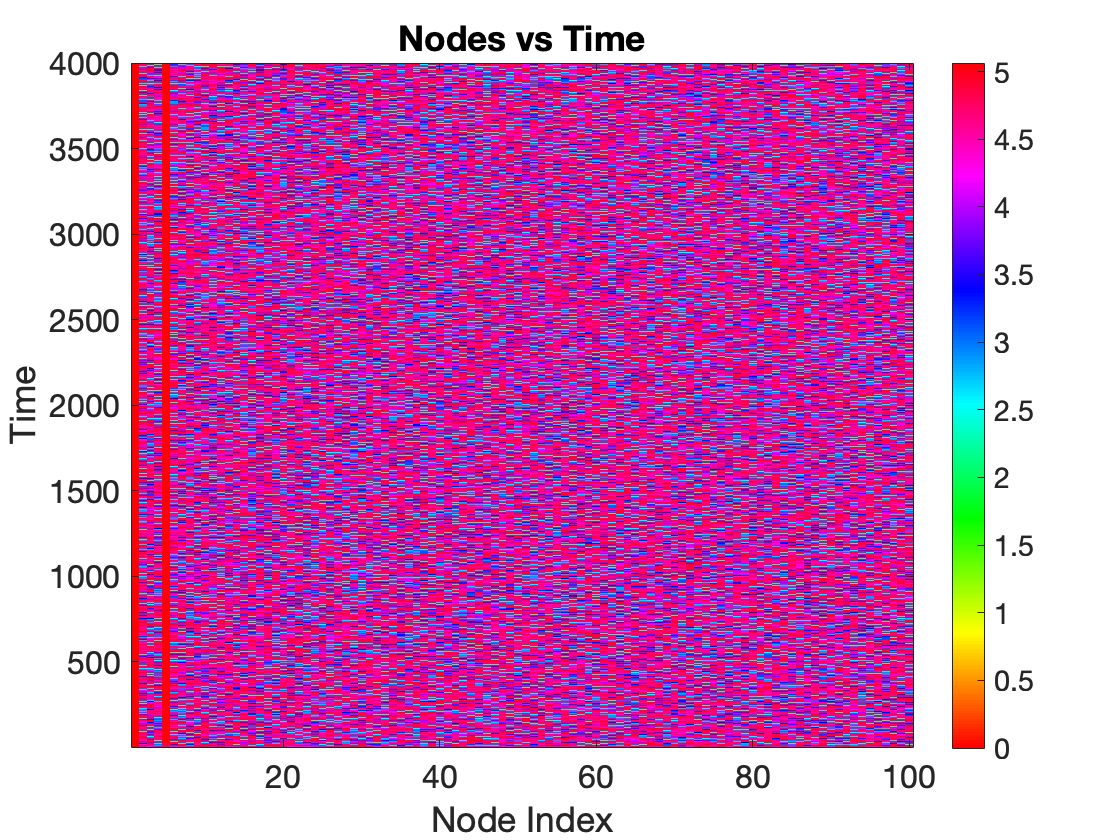}%
        }
        \vfill%
     \subfloat[]{%
        \includegraphics[width=0.37\textwidth]{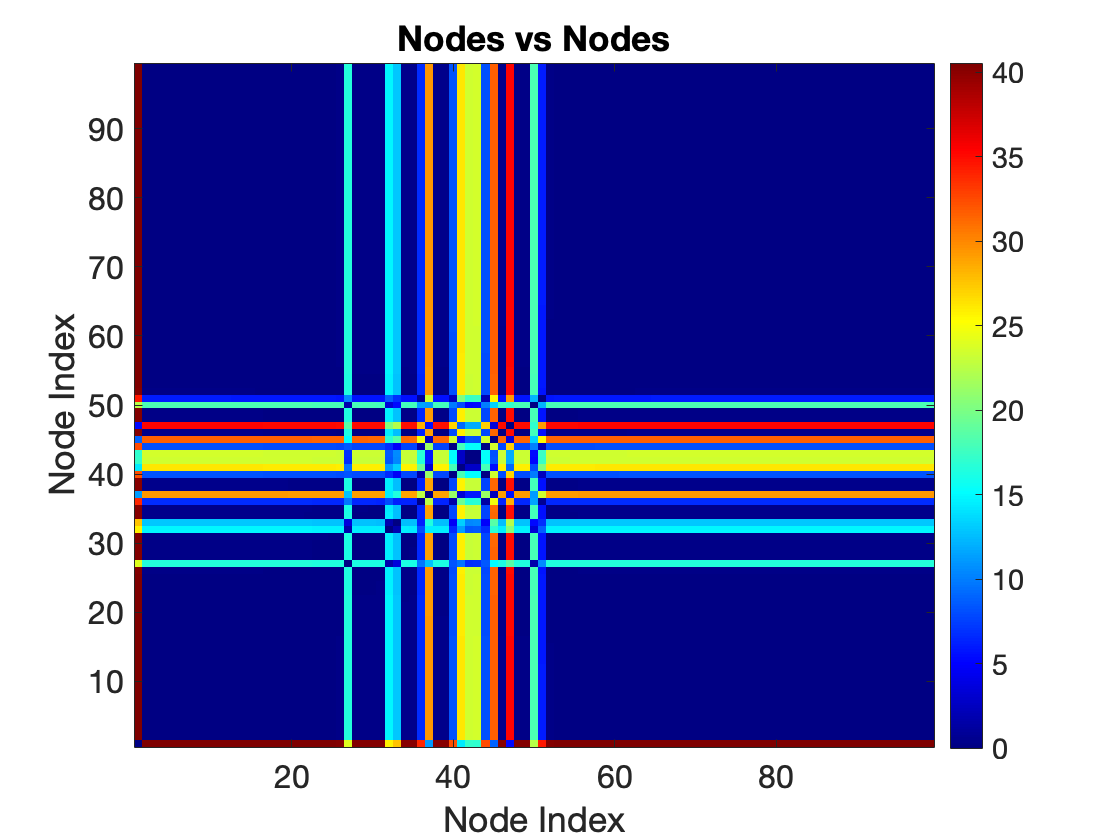}%
        }%
    %\hfill%
    \subfloat[\tiny{Chimera state, $\sigma$=0.057}]{%
        \includegraphics[width=0.37\textwidth]{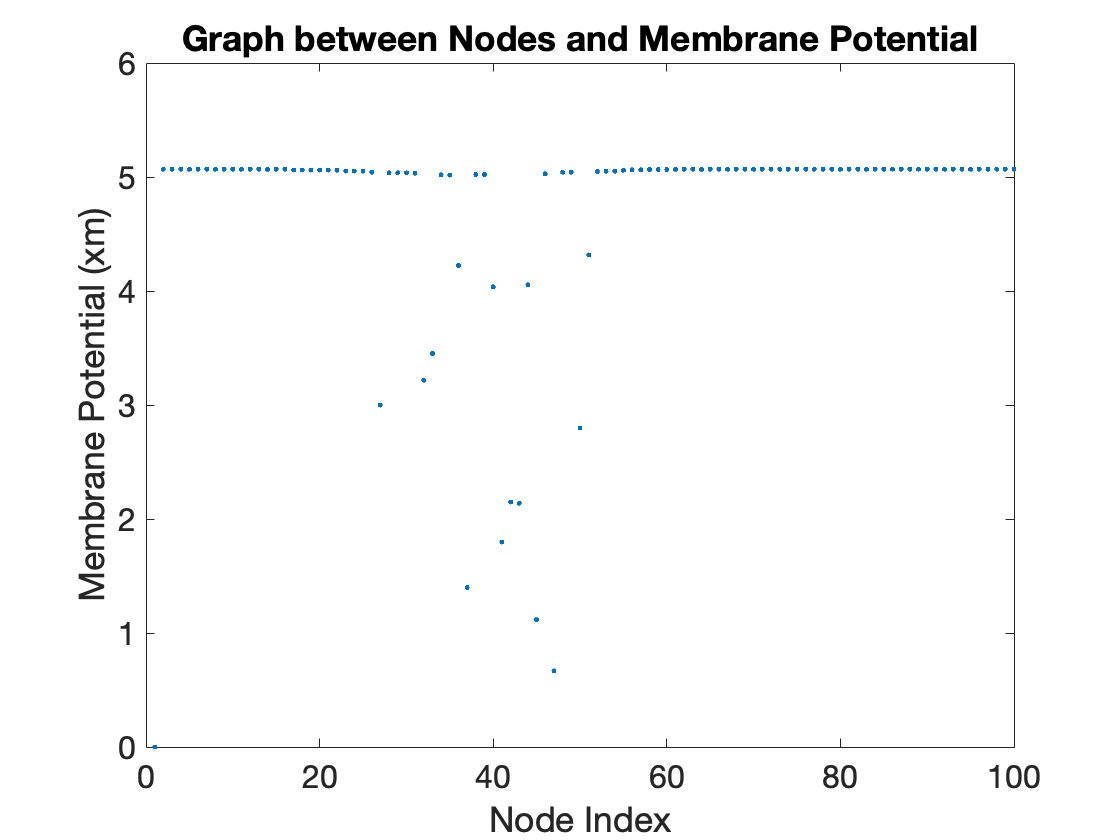}%
        }%
    %\hfill%
    \subfloat[]{%
        \includegraphics[width=0.37\textwidth]{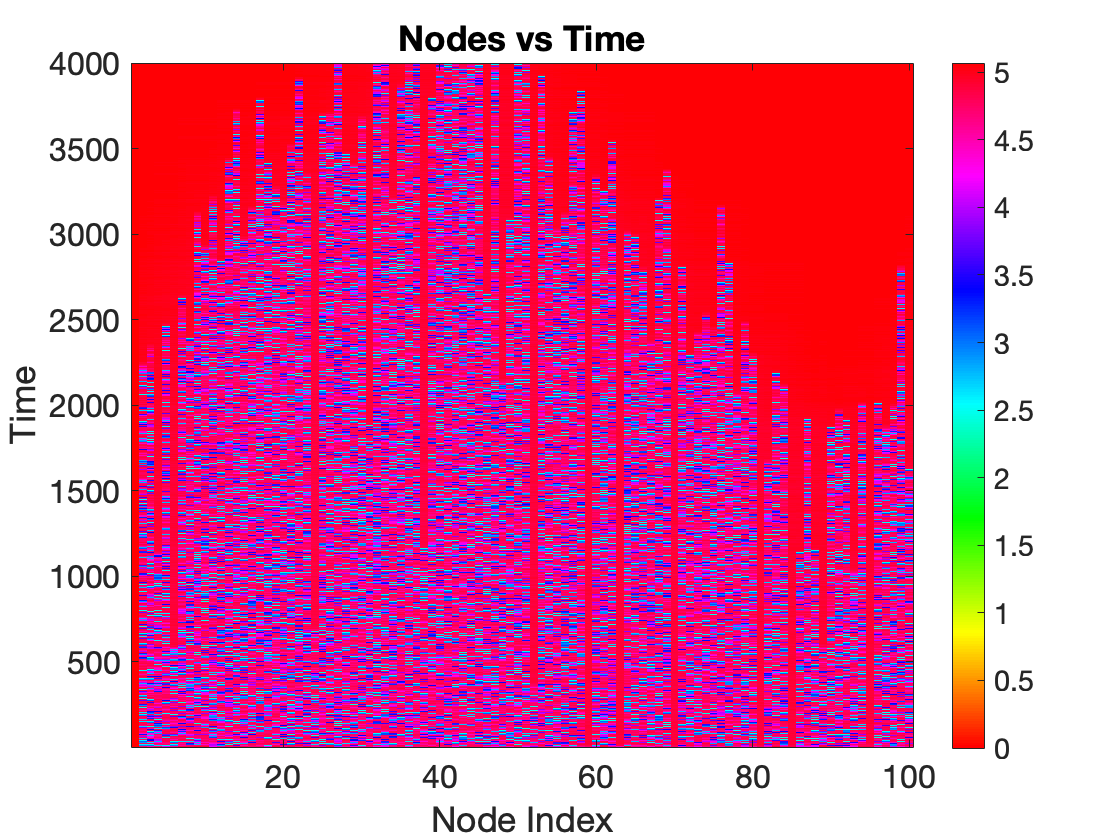}%
        }%
        \vfill%
     \subfloat[]{%
        \includegraphics[width=0.37\textwidth]{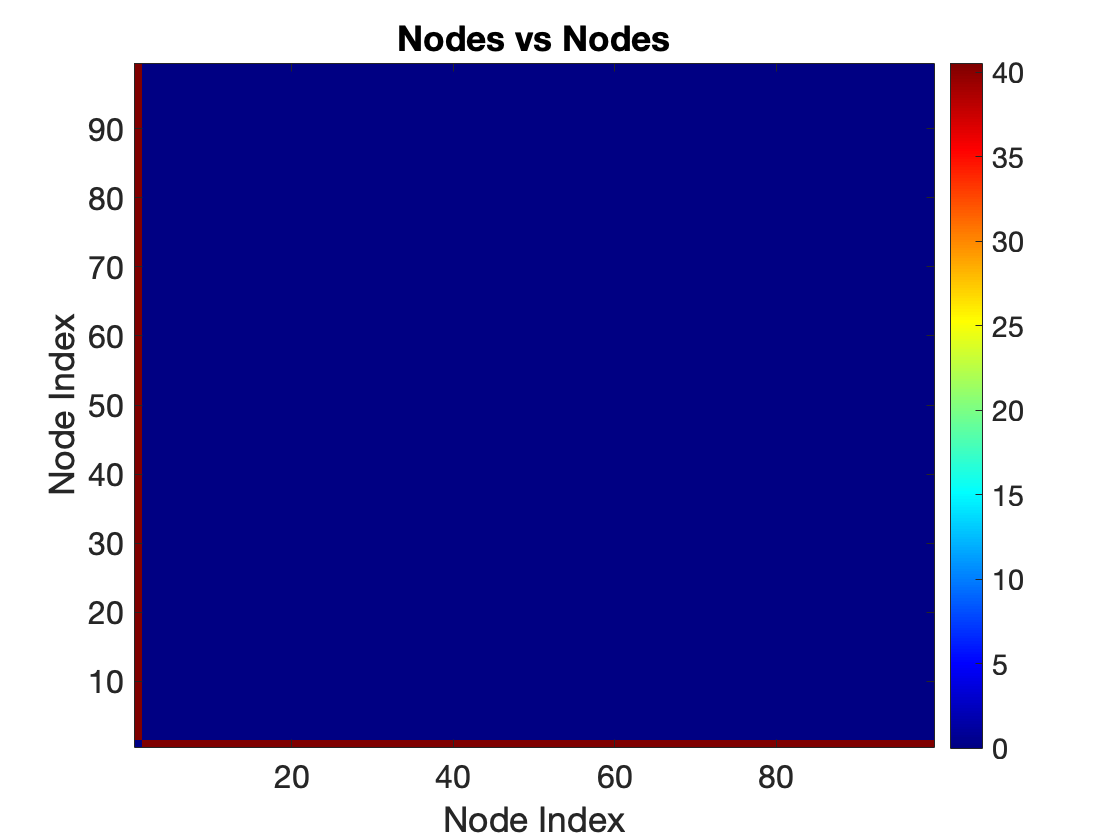}%
        }%
    %\hfill%
    \subfloat[\tiny{Synchronized state, $\sigma$=0.1}]{%
        \includegraphics[width=0.37\textwidth]{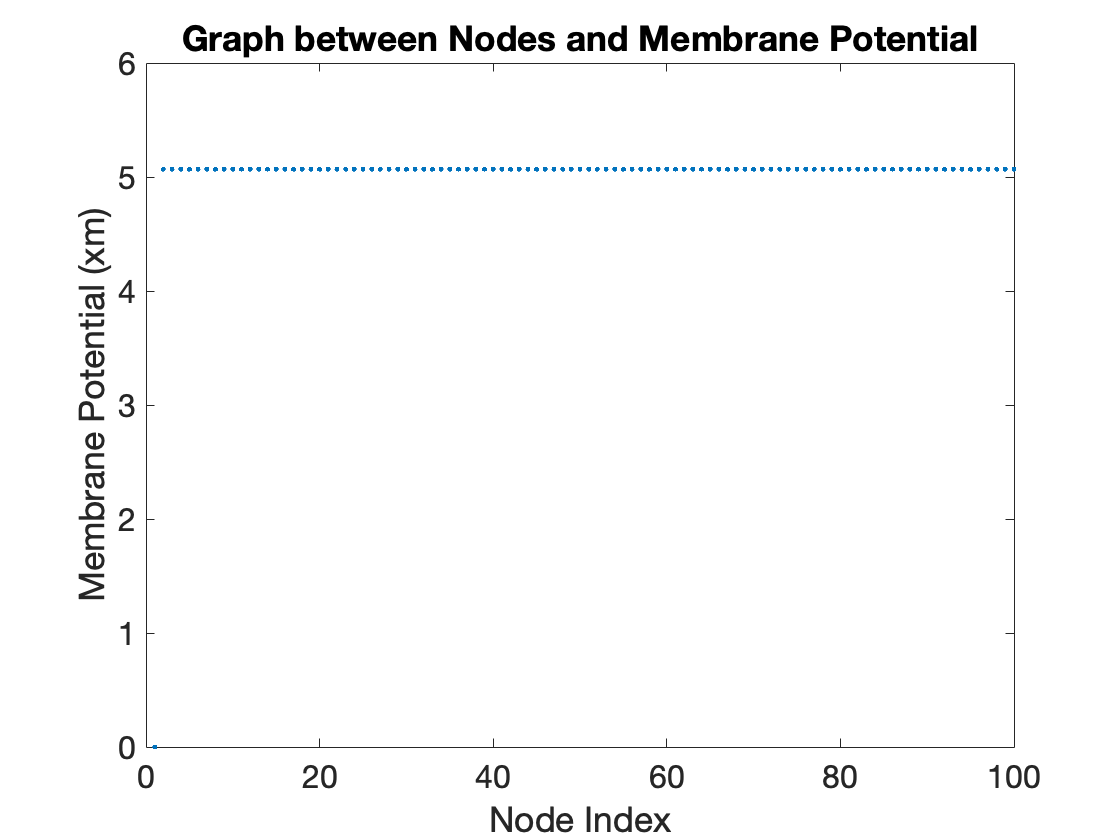}%
        }%
    %\hfill%
    \subfloat[]{%
        \includegraphics[width=0.37\textwidth]{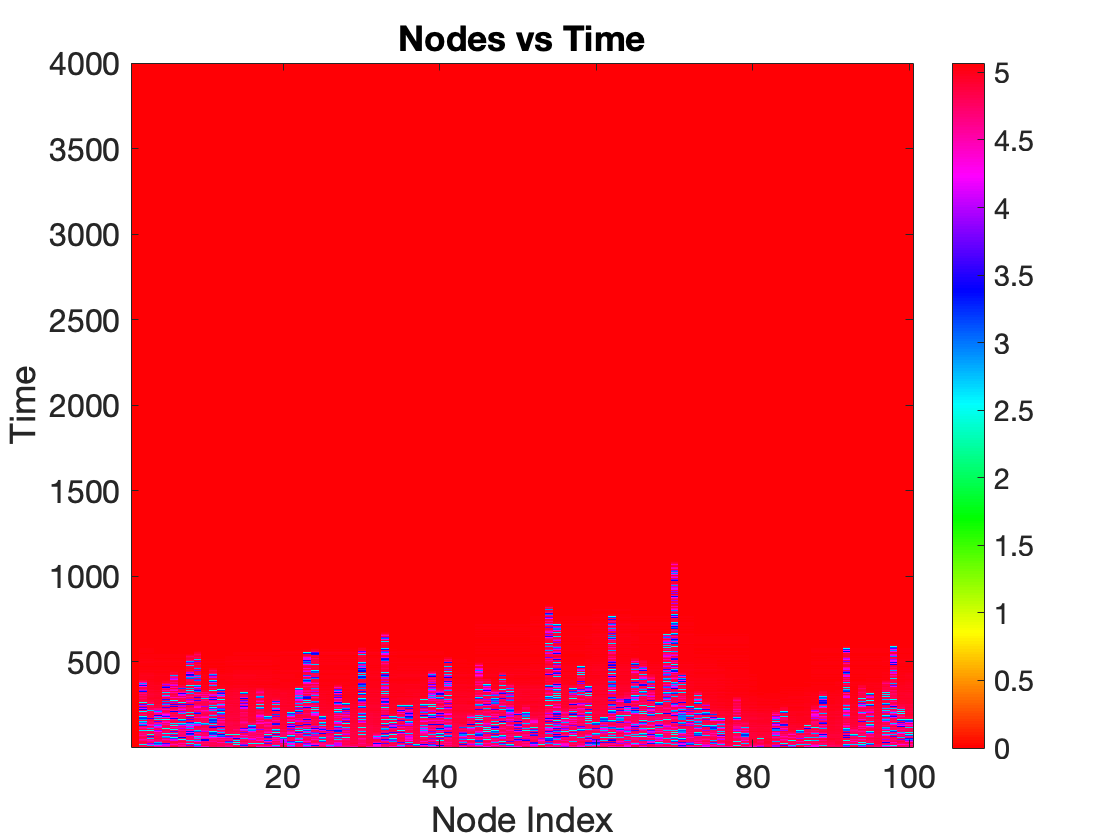}%
        }%
    \caption{Spatiotemporal patterns under the ring network of the map $\mathcal{M}_{r,k}(x,\phi)$ by varying $\sigma$. The first row showcases the unsynchronized behavior among neurons. The second row shows a chimera state. The third row shows the synchronized behavior. Parameters are considered as: $k_0 =2, k_1=0.3, k_2=0.5, k=-0.5,$ and $ r=1.2$}\label{ring}
\end{figure}

First, we examine the ring network in which there is no connection between outer nodes and the central node (i.e., $\mu=0$). So, $\sigma$ is the only controlling parameter in this network. By varying the parameter $\sigma$ from $0.01$ to $0.1$, we observe distinct spatiotemporal patterns, as illustrated in Fig.\ref{ring}. In the subfigure\ref {ring} first row  (i.e., (a),(b),(c)) for the parameter $\sigma=0.01$, we witness the unsynchronized behavior of the network of map $\mathcal{M}_{r,k}(x,\phi)$.  In subfigure \ref{ring} second row (i.e., (d),(e),(f)) for the parameter $\sigma=0.057$, a state between the synchronized and unsynchronized state is observed, which is called the chimera state. A chimera state is a phenomenon that is found in many places in the real world, such as dolphin sleeping patterns, where one eye is closed, and one eye is open (basically, in this situation, half of the brain is in a sleeping state and half of the brain is in an active state). So, we are interested in finding such a phenomenon in a neuron's network exhibited by the map $\mathcal{M}_{r,k}(x,\phi)$. Lastly, in the third-row subfigure \ref{ring} (i.e., (g),(h),(i)) for the value of parameter $\sigma=0.1$, a synchronized behavior is observed in the network dynamics.

% \newpage (i.e., a dynamical state where the continuous and piece-wise traveling wave both occur)

 % treating both coupling strengths $\mu$ and $\sigma$ as nonzero control parameters.

\subsection{Ring-star Network}

\begin{figure}[ht!] 
    \centering
    \subfloat[]{%
        \includegraphics[width=0.37\textwidth]{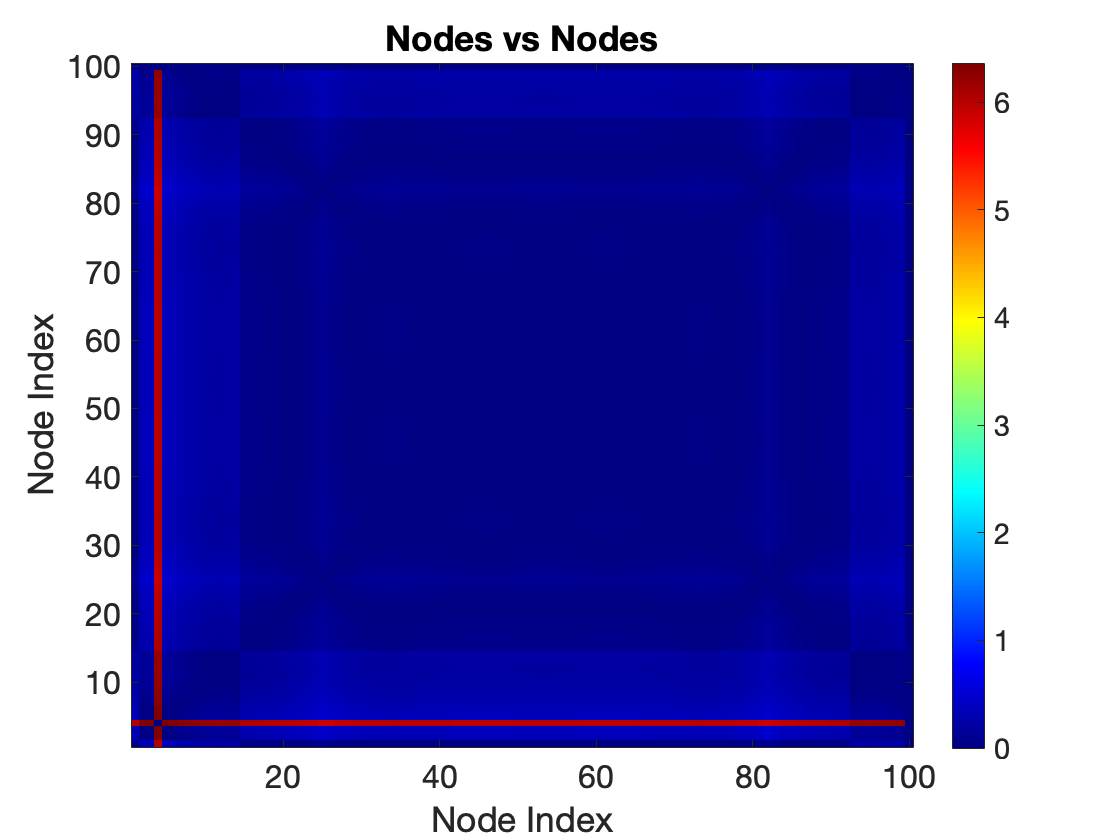}%
        }%
    %\hfill%
    \subfloat[\tiny{Continuous traveling wave, $\mu=0.00007, \sigma=0.16$}]{%
        \includegraphics[width=0.37\textwidth]{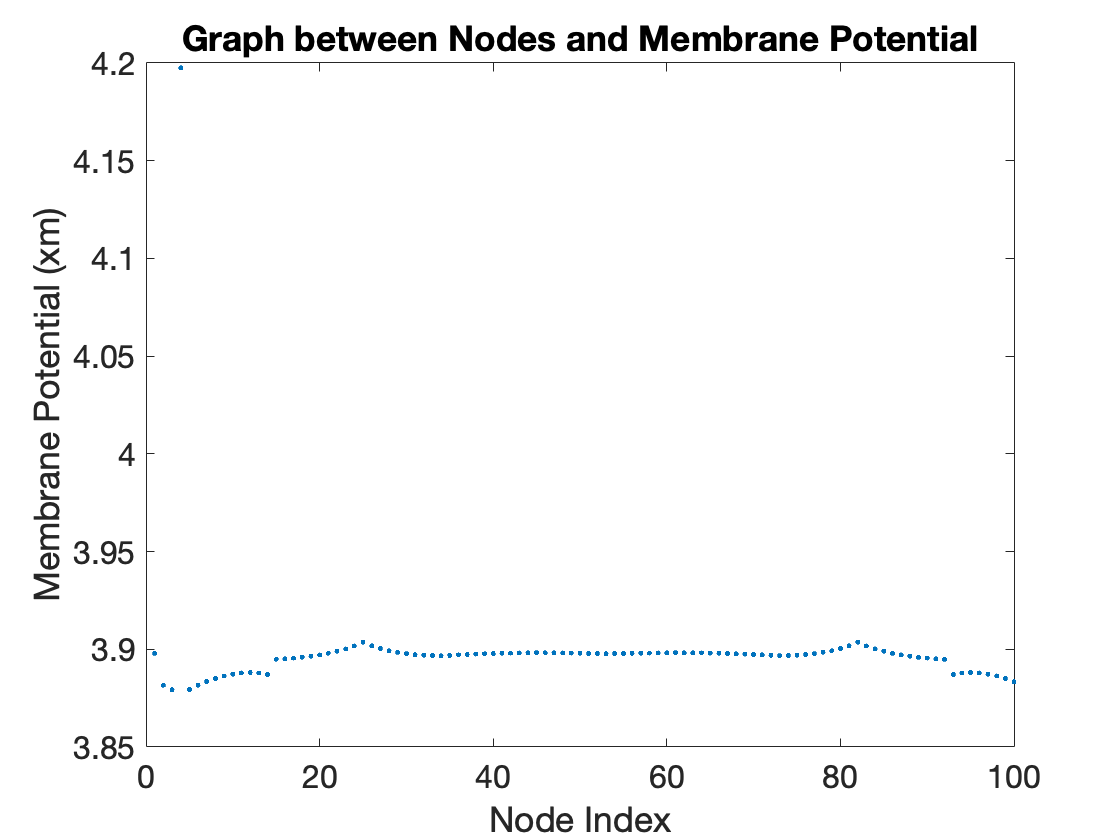}%
        }%
    %\hfill%
    \subfloat[]{%
        \includegraphics[width=0.37\textwidth]{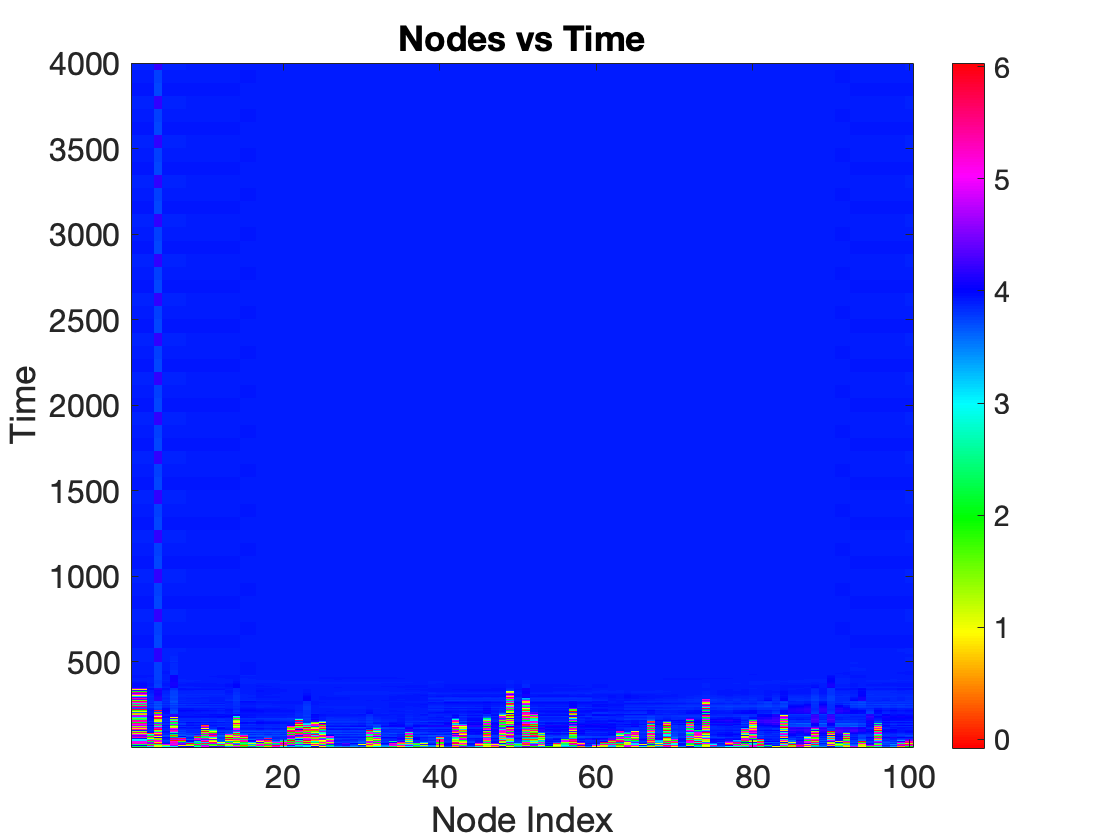}%
        }
        \vfill%
     \subfloat[]{%
        \includegraphics[width=0.37\textwidth]{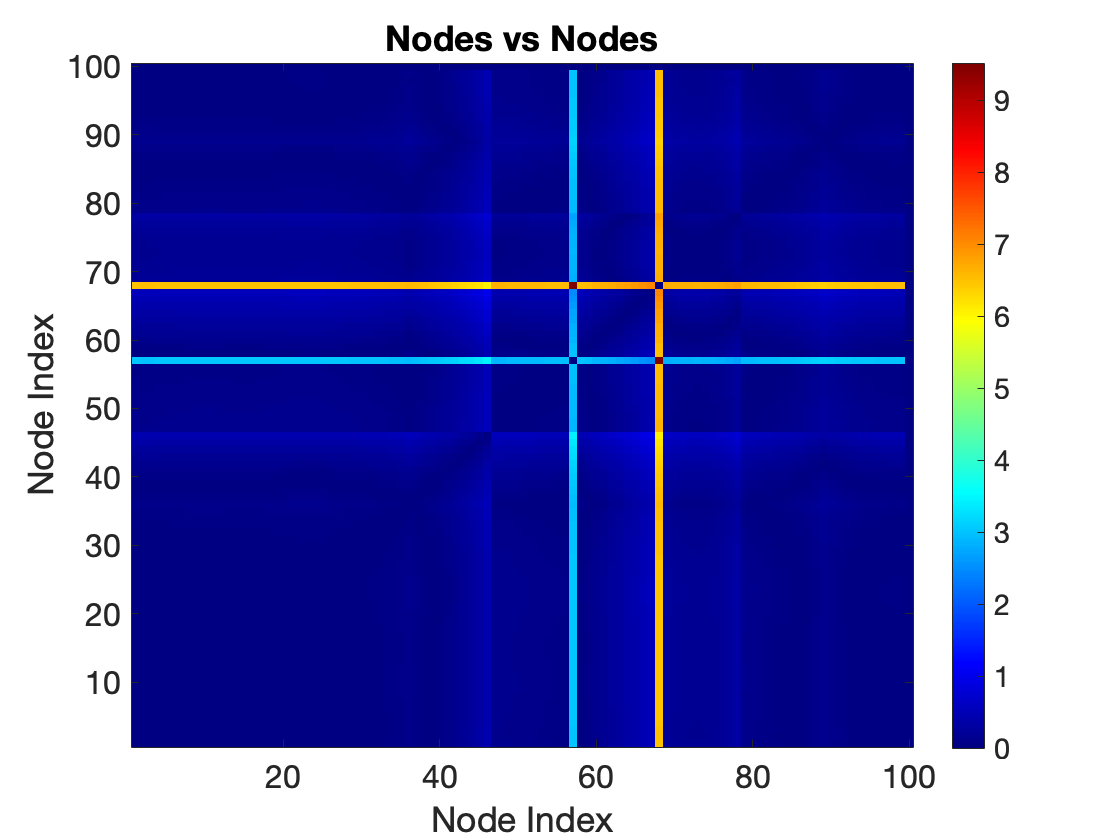}
        }%
    %\hfill%
    \subfloat[\tiny{Chimera state, $\mu$=0.00001, $\sigma$=0.15}]{%
        \includegraphics[width=0.37\textwidth]{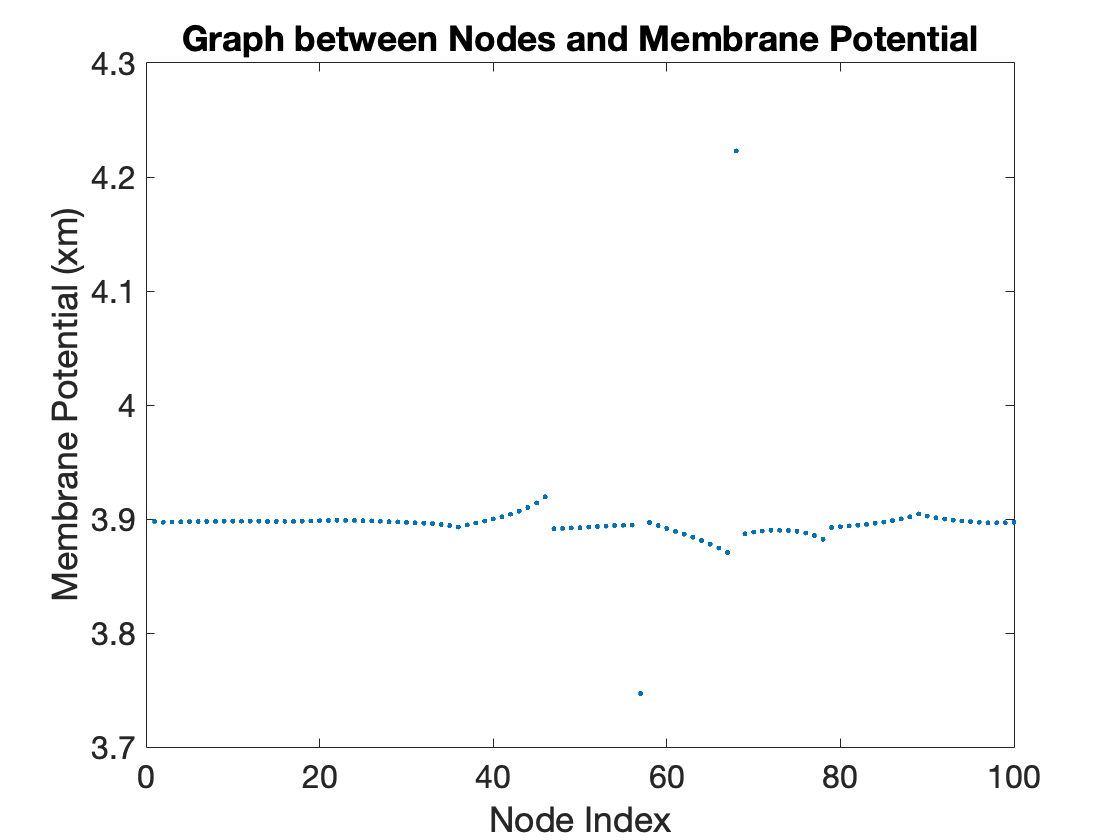}%
        }%
    %\hfill%
    \subfloat[]{%
        \includegraphics[width=0.37\textwidth]{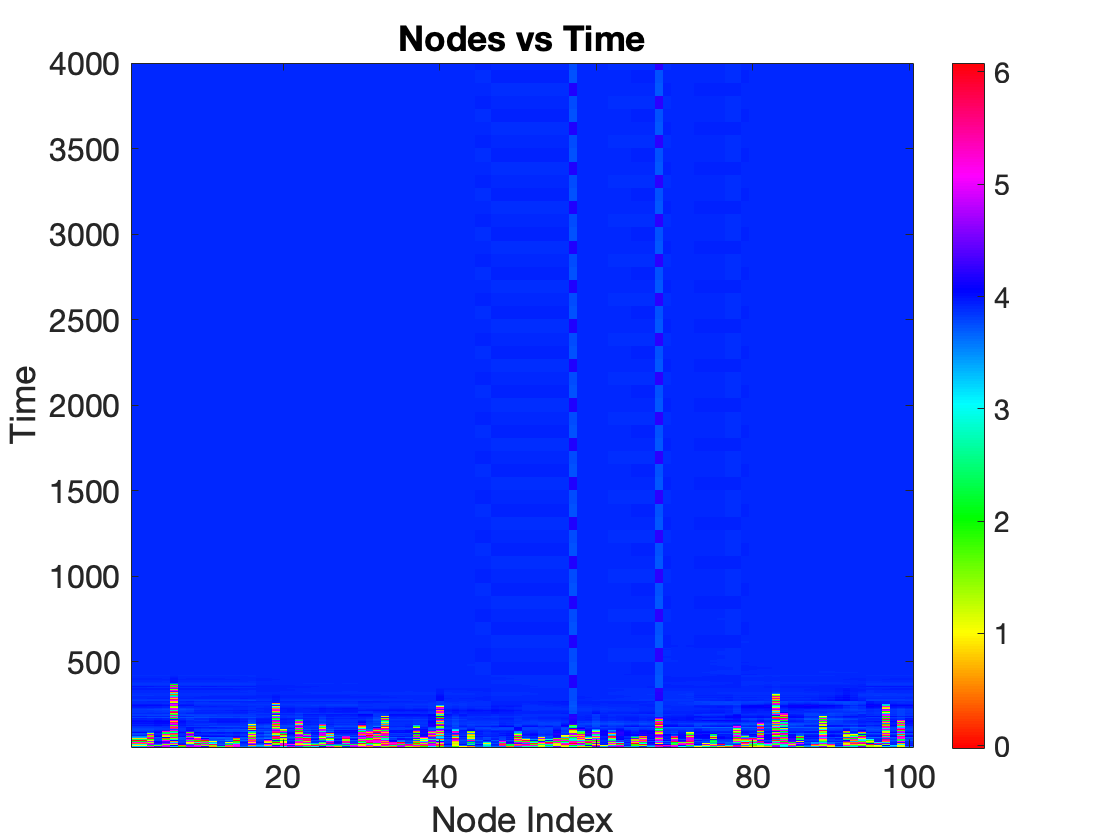}%
        }%
        \vfill%
     \subfloat[ ]{%
        \includegraphics[width=0.37\textwidth]{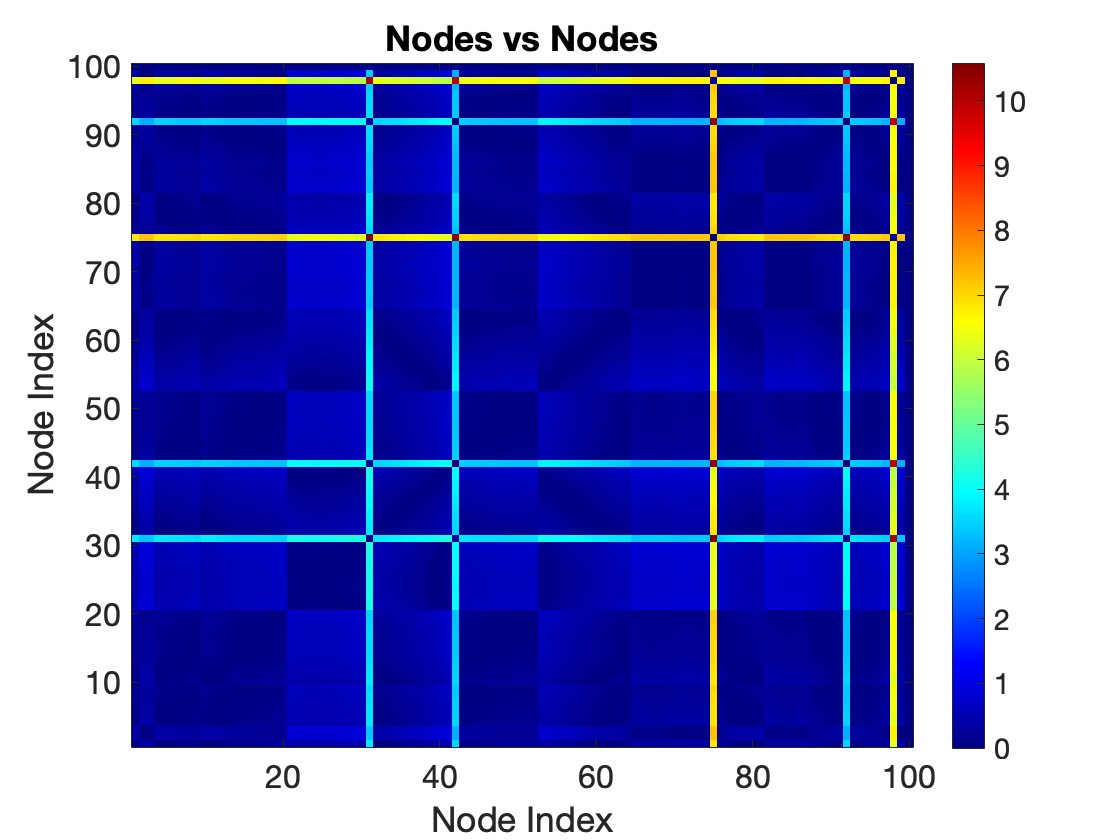}%
        }%
    %\hfill%
    \subfloat[\tiny{Piecewise traveling pattern, $\mu$=0.0001, $\sigma$=0.14}]{%
        \includegraphics[width=0.37\textwidth]{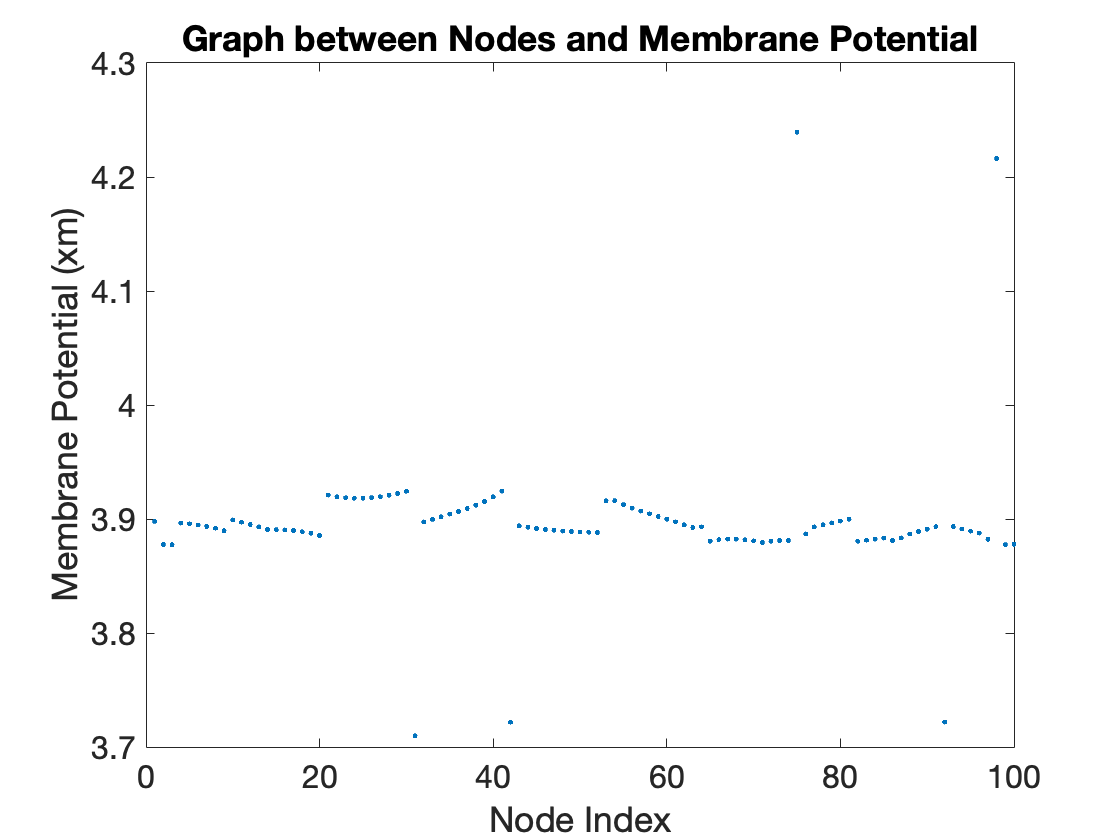}%
        }%
    %\hfill%
    \subfloat[]{%
        \includegraphics[width=0.37\textwidth]{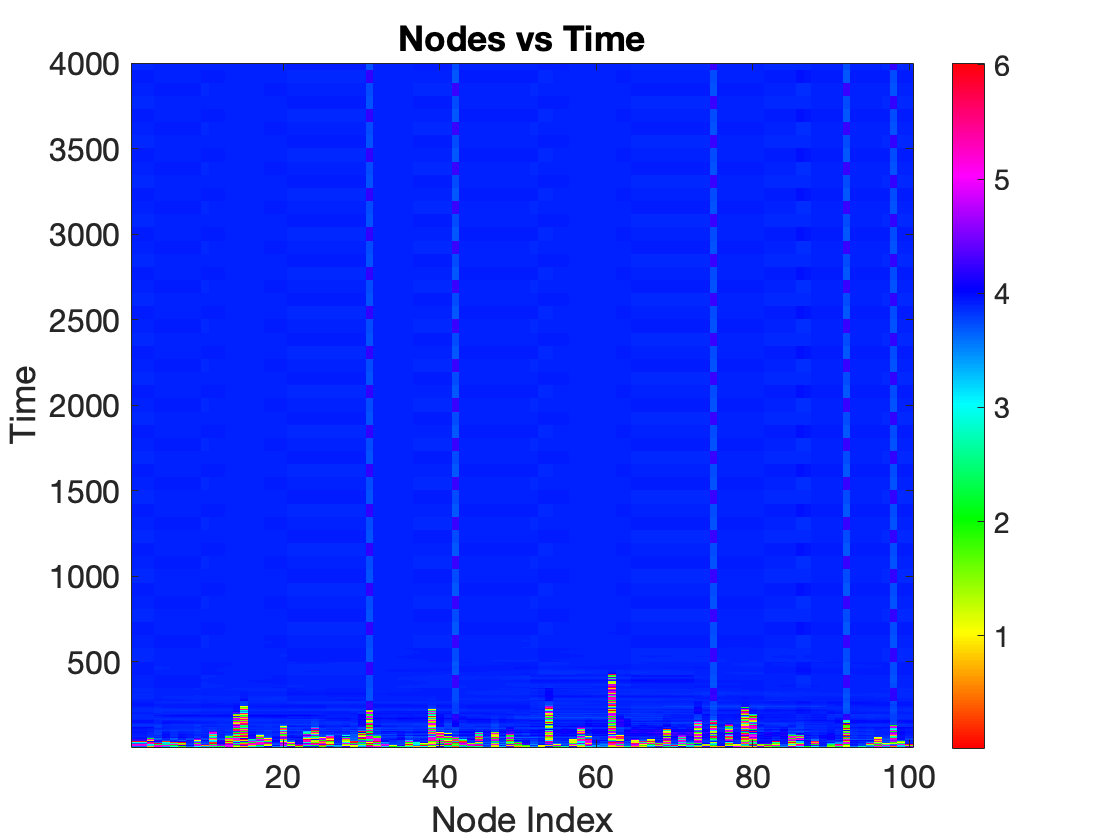}%
        }%
    \caption{Spatiotemporal patterns of the the map $\mathcal{M}_{r,k}(x,\phi)$ in a ring-star network with the variation in parameters $\mu$ and $\sigma$. The first row shows a continuous wave pattern.  The second row shows a chimera pattern. Finally, the third row shows that a piecewise pattern emerges among the nodes. Parameters are considered as: $k_0=0.7,  k_1=0.8215, k_2=-0.39, k=-0.2,$ and $ r=2.15.$}\label{ringstar}
\end{figure}

We investigate the ring-star network of the map $\mathcal{M}_{r,k}(x,\phi)$, treating both coupling strengths as control parameters (i.e., consider the coupling strength $\mu$ and $\sigma$ both non-zero in this case). With the variation in the coupling strength, we illustrate different spatiotemporal patterns, which are showcased in Fig.\ref{ringstar}. In the first row of Fig.\ref{ringstar} (i.e., (a),(b),(c)) for the parameter $\mu=0.00007 $ and $ \sigma=0.16$, the system showcased the continuous traveling wave patterns in the membrane potential of neurons.  In the second row of the Fig.\ref{ringstar} (i.e., (d),(e),(f)) for $\mu=0.00001$ and $ \sigma=0.15$, a chimera state is observed in the network of the map $\mathcal{M}_{r,k}(x,\phi)$. Lastly, in the third row of the Fig.\ref{ringstar} (i.e., (g),(h),(i)) for $\mu=0.0001$ and $ \sigma=0.14$, a piece-wise traveling pattern is observed in the ring-star network of the map $\mathcal{M}_{r,k}(x,\phi)$.

\subsection{Star Network}

\begin{figure}[p] 
    \centering
    \subfloat[]{%
        \includegraphics[width=0.37\textwidth]{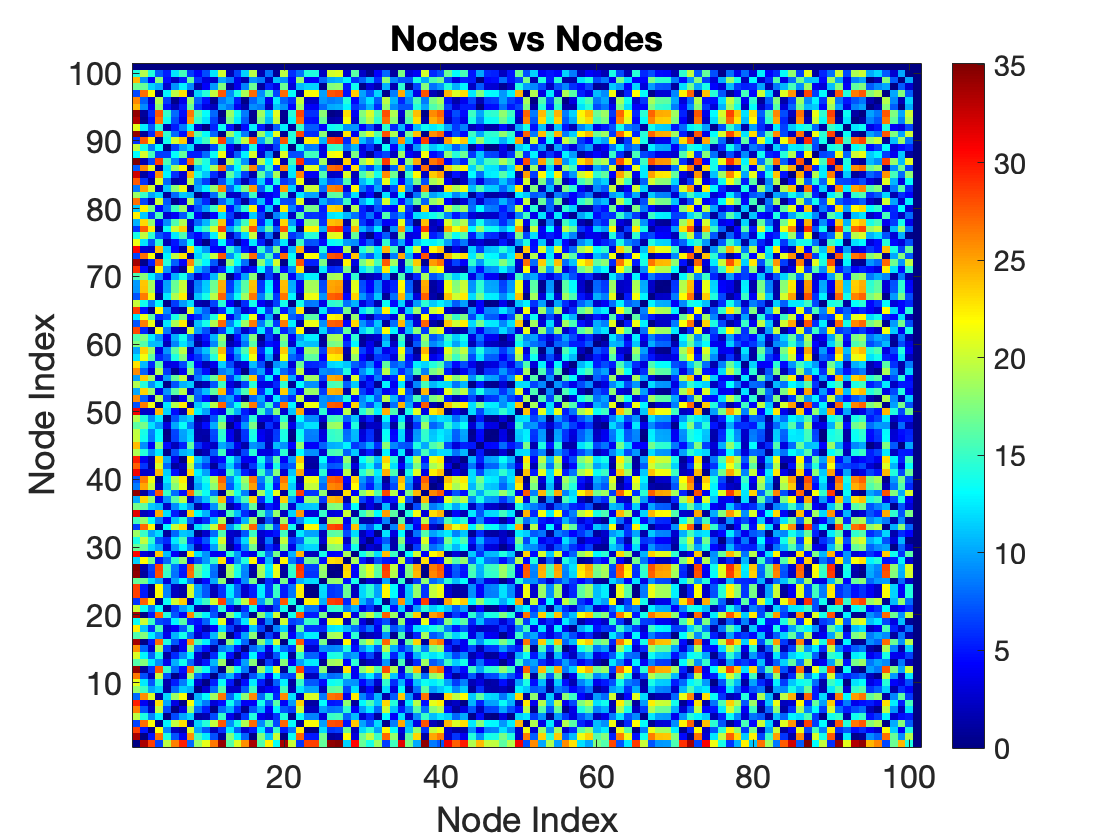}%
        }%
    %\hfill%
    \subfloat[\tiny{Unsynchronized state, $\mu$=0.001, $r$=1.2}]{%
        \includegraphics[width=0.37\textwidth]{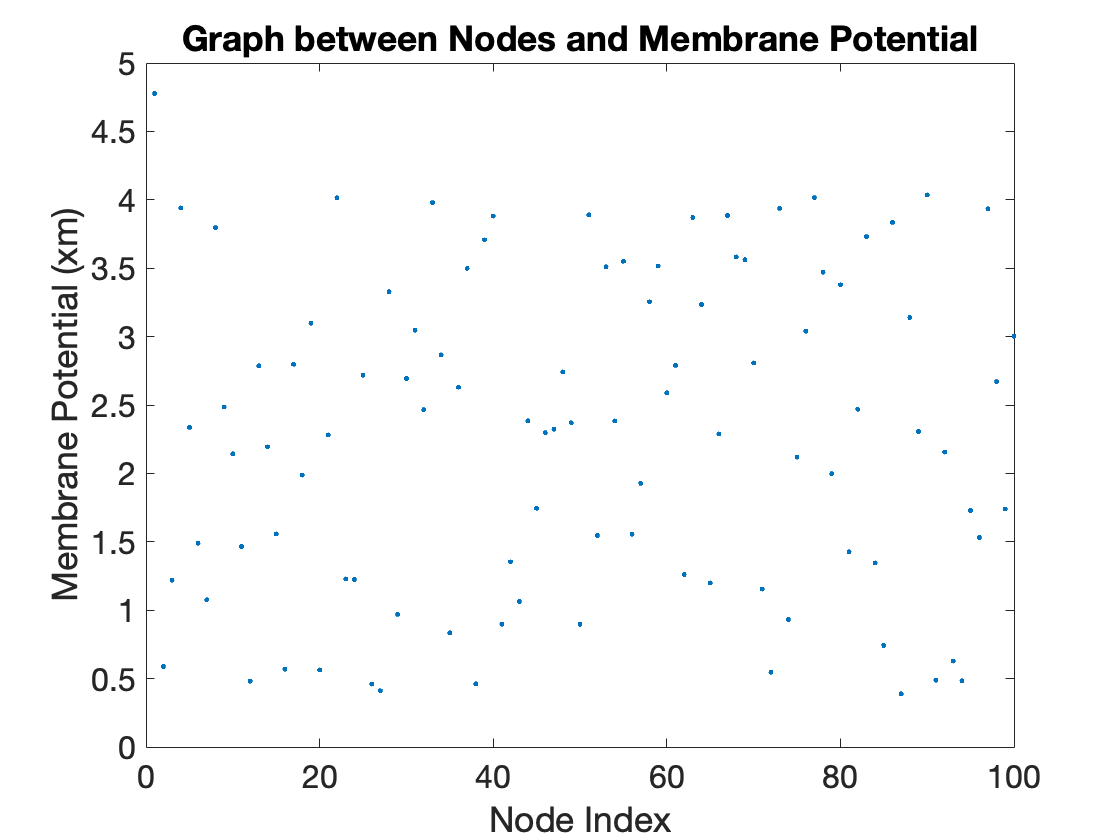}%
        }%
    %\hfill%
    \subfloat[]{%
        \includegraphics[width=0.37\textwidth]{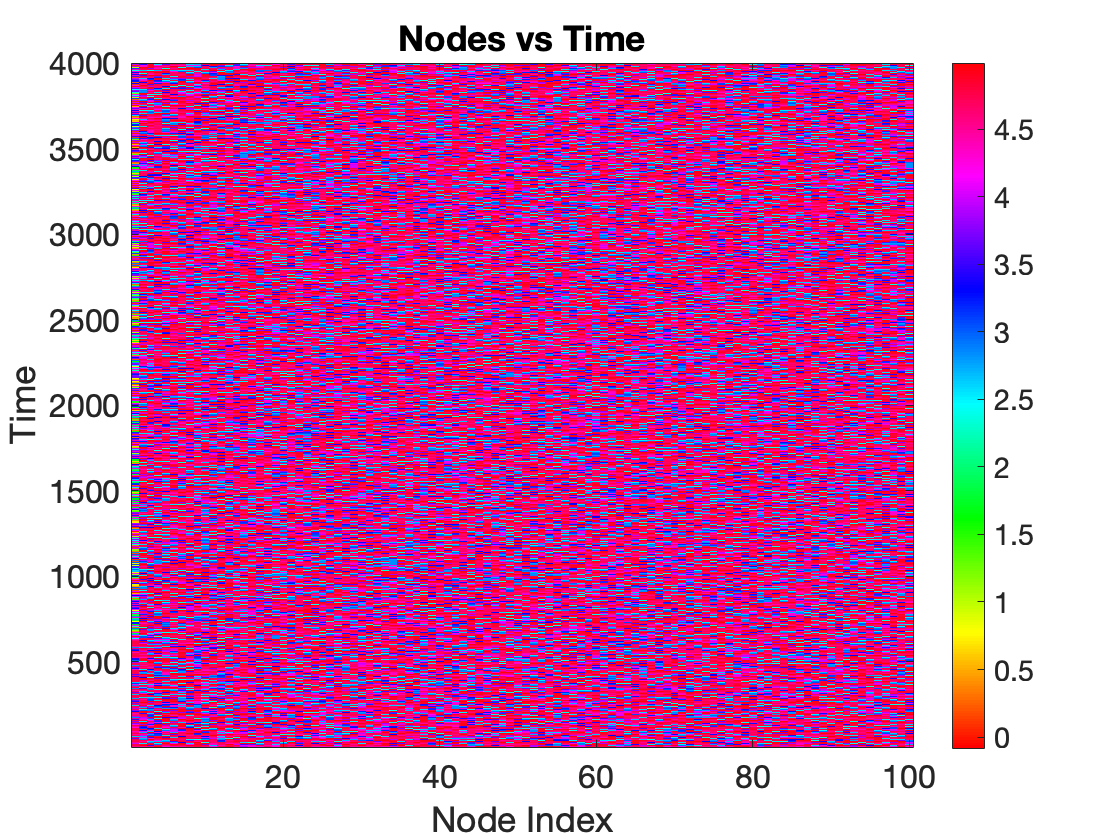}%
        }
        \vfill%
     \subfloat[]{%
        \includegraphics[width=0.37\textwidth]{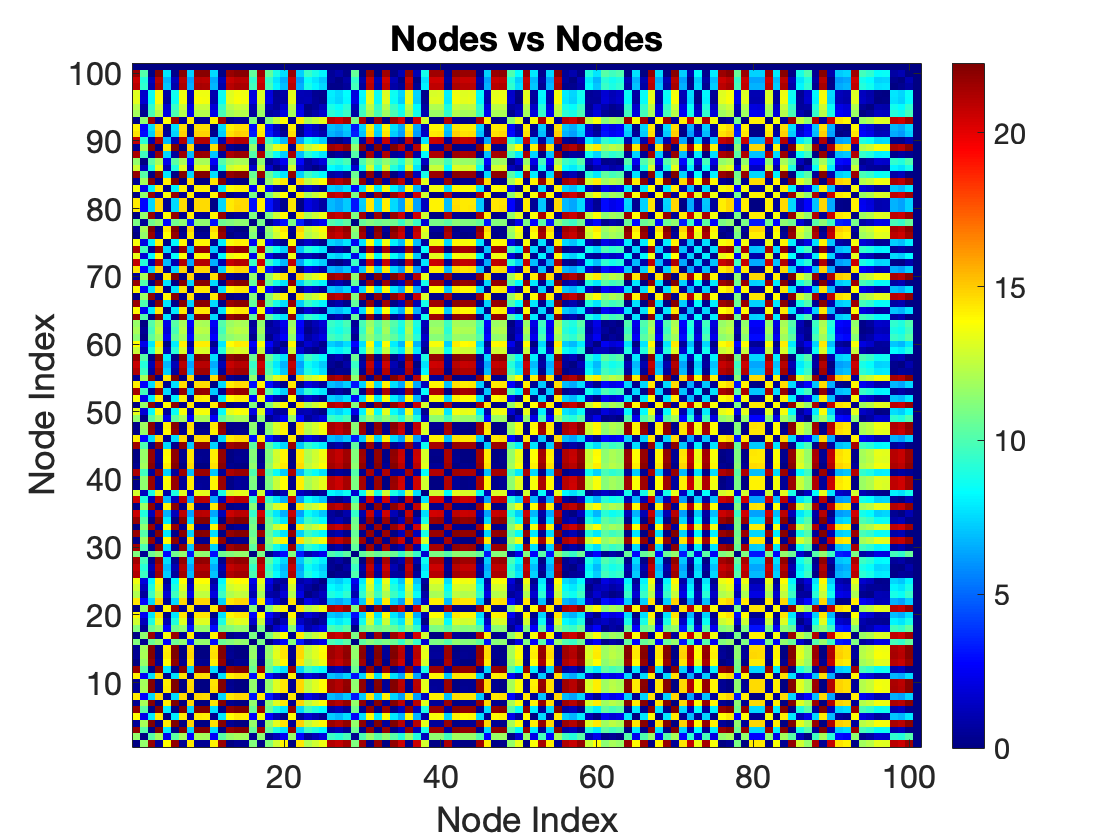}%
        }%
    %\hfill%
    \subfloat[\tiny{Transition state, $\mu$=0.0007, $r$=1}]{%
        \includegraphics[width=0.37\textwidth]{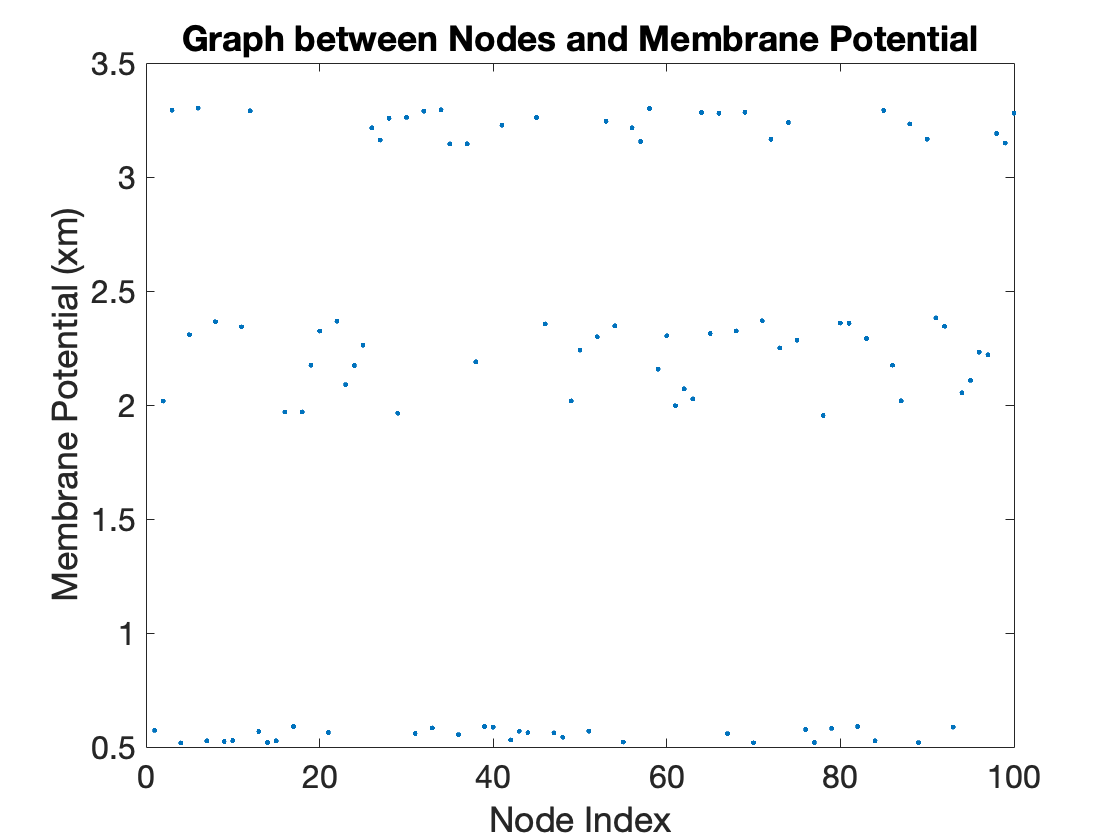}%
        }%
    %\hfill%
    \subfloat[]{%
        \includegraphics[width=0.37\textwidth]{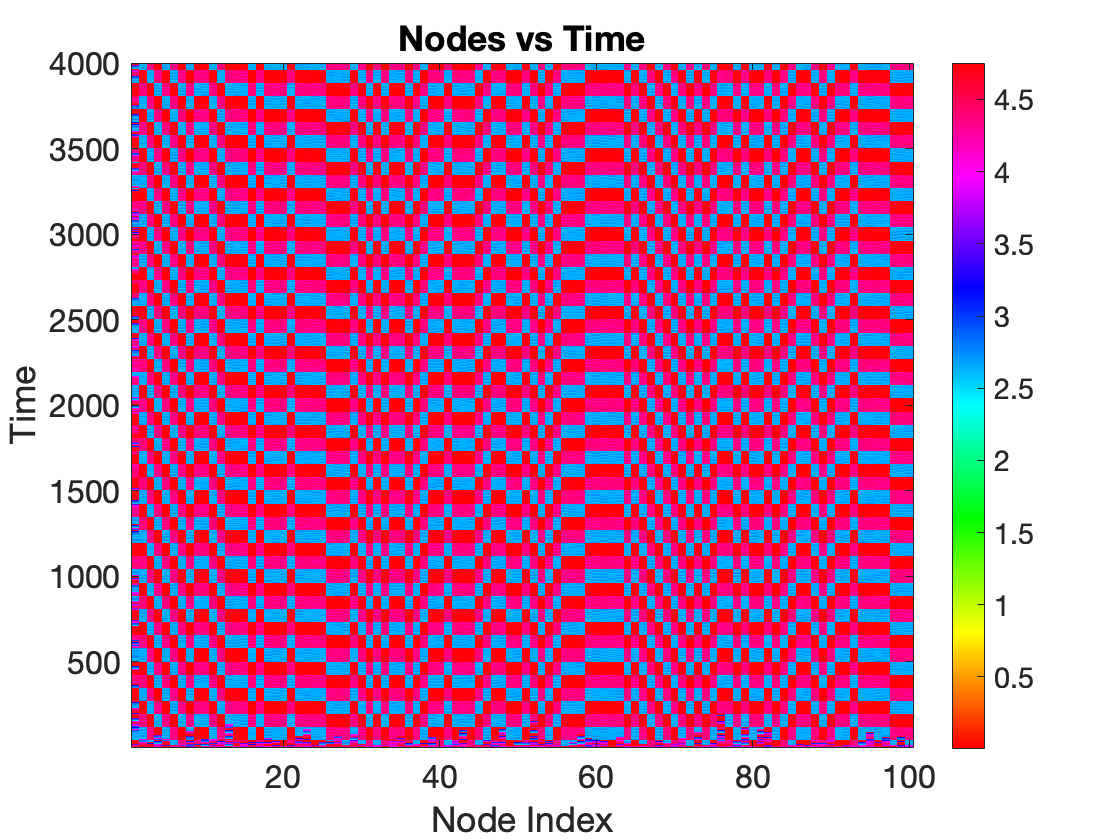}%
        }%
        \vfill%
     \subfloat[]{%
        \includegraphics[width=0.37\textwidth]{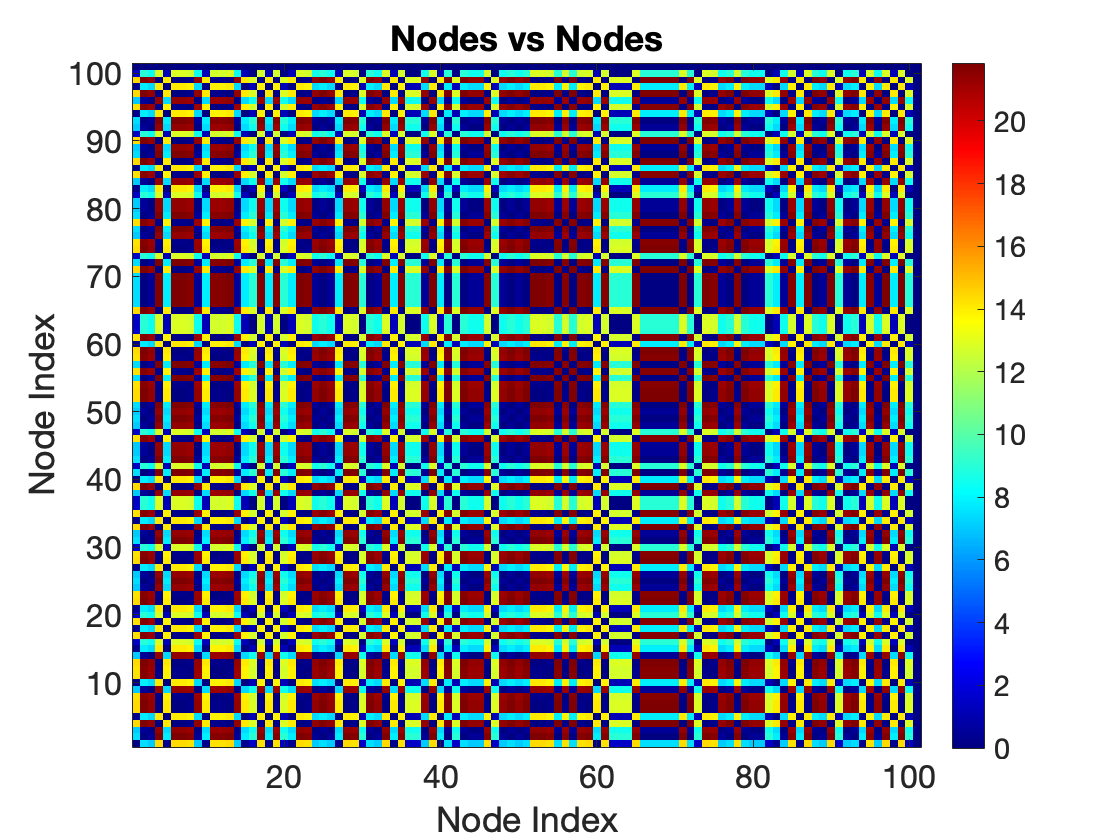}%
        }%
    %\hfill%
    \subfloat[\tiny{Five-clustered state, $\mu$=0.0003, $r$=0.98}]{%
        \includegraphics[width=0.37\textwidth]{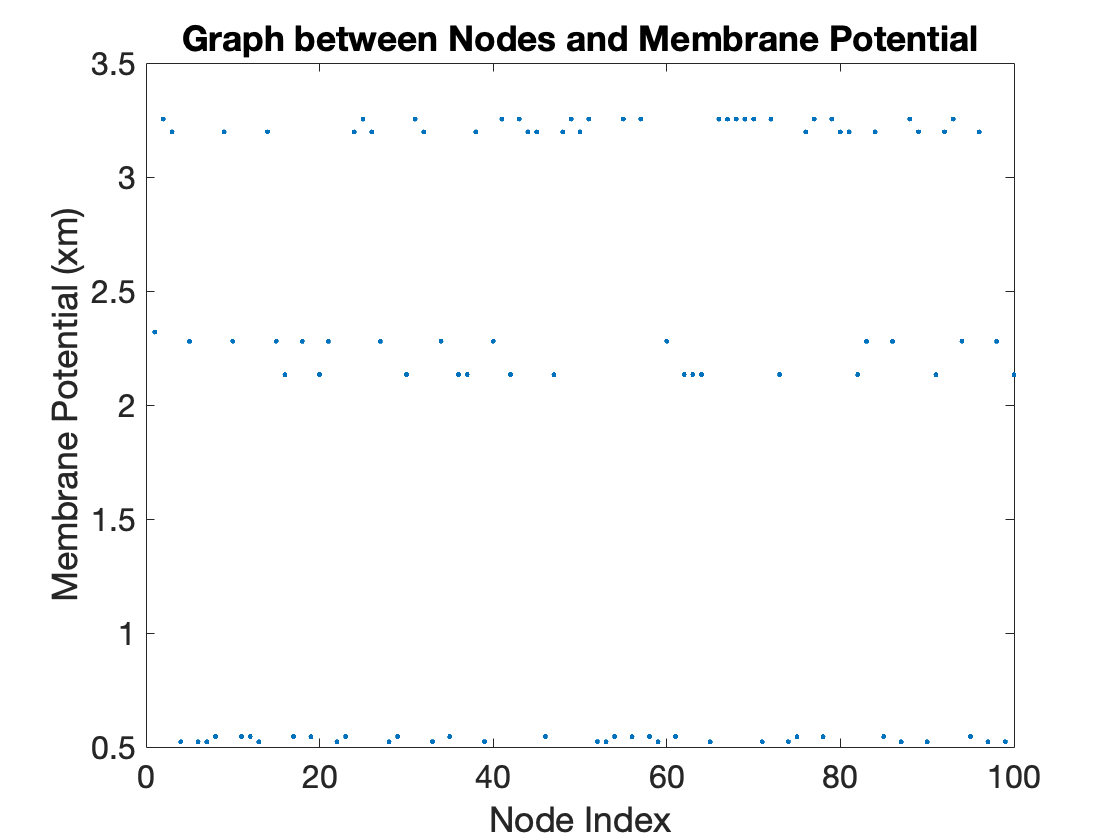}%
        }%
    %\hfill%
    \subfloat[]{%
        \includegraphics[width=0.37\textwidth]{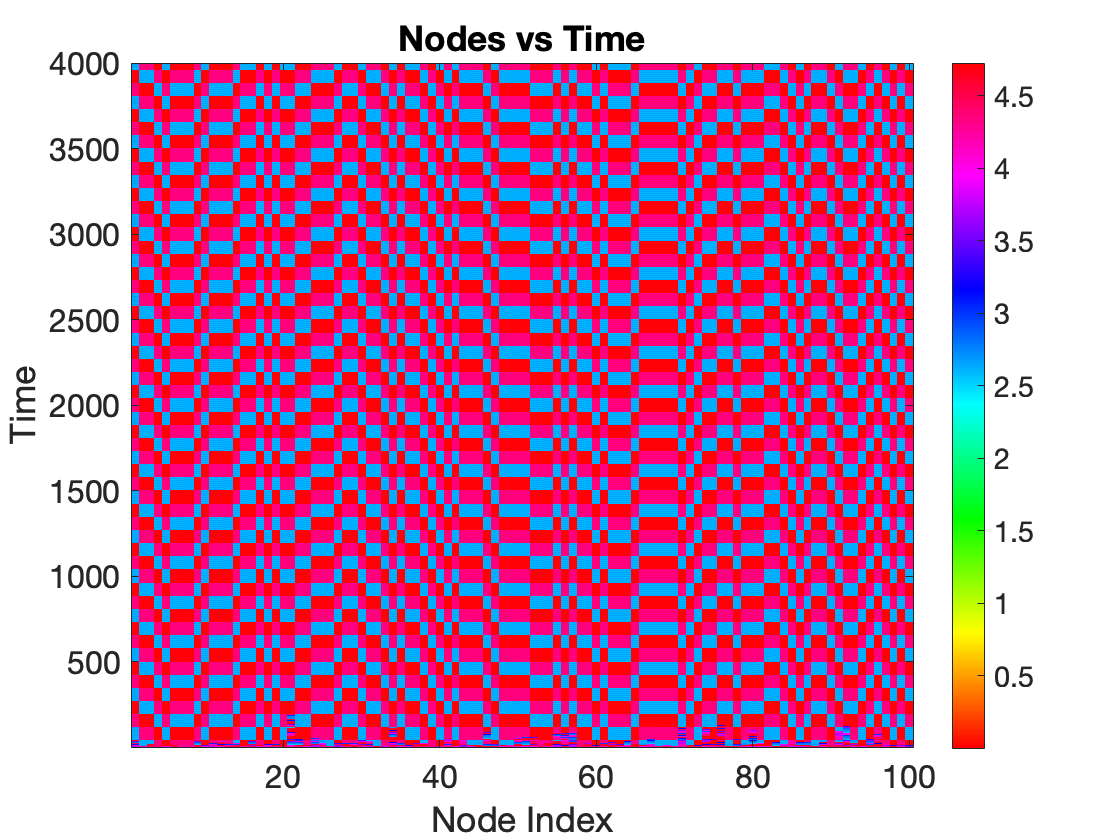}%
        }%
        \vfill%
     \subfloat[]{%
        \includegraphics[width=0.37\textwidth]{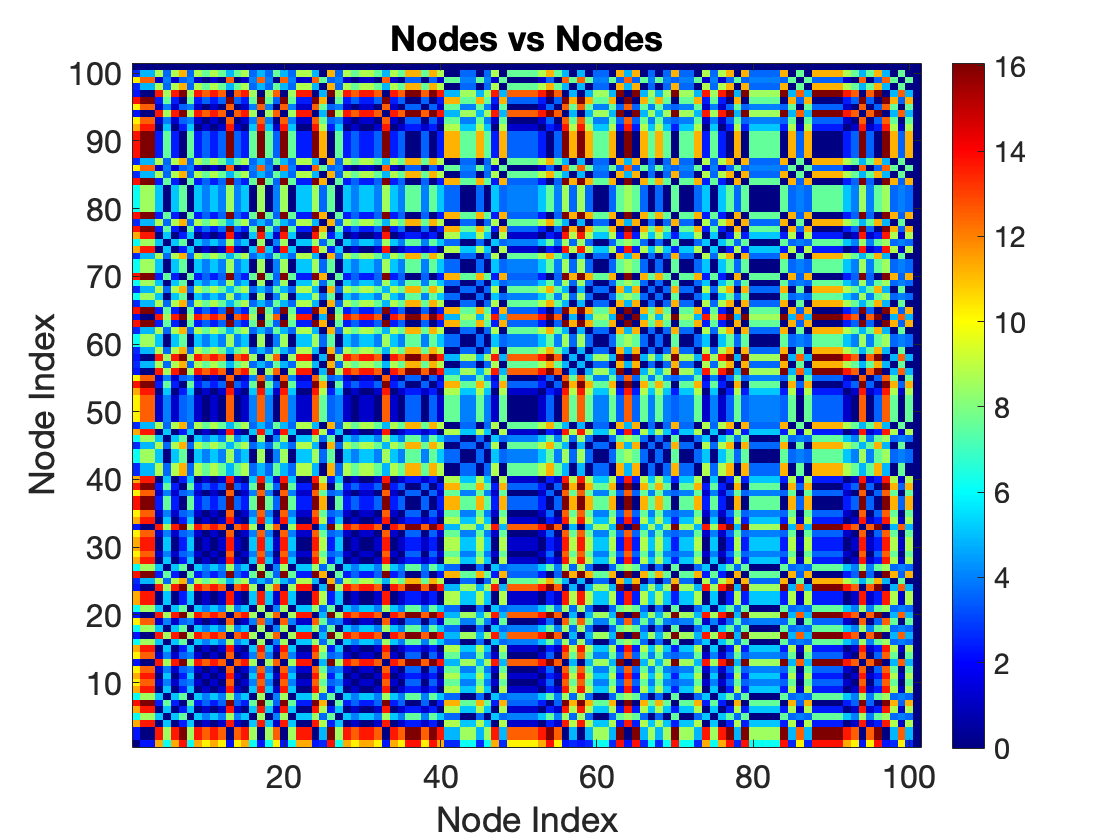}%
        }%
    %\hfill%
    \subfloat[\tiny{ Six-clustered state, $\mu$=0.0001, $r$=0.8}]{%
        \includegraphics[width=0.37\textwidth]{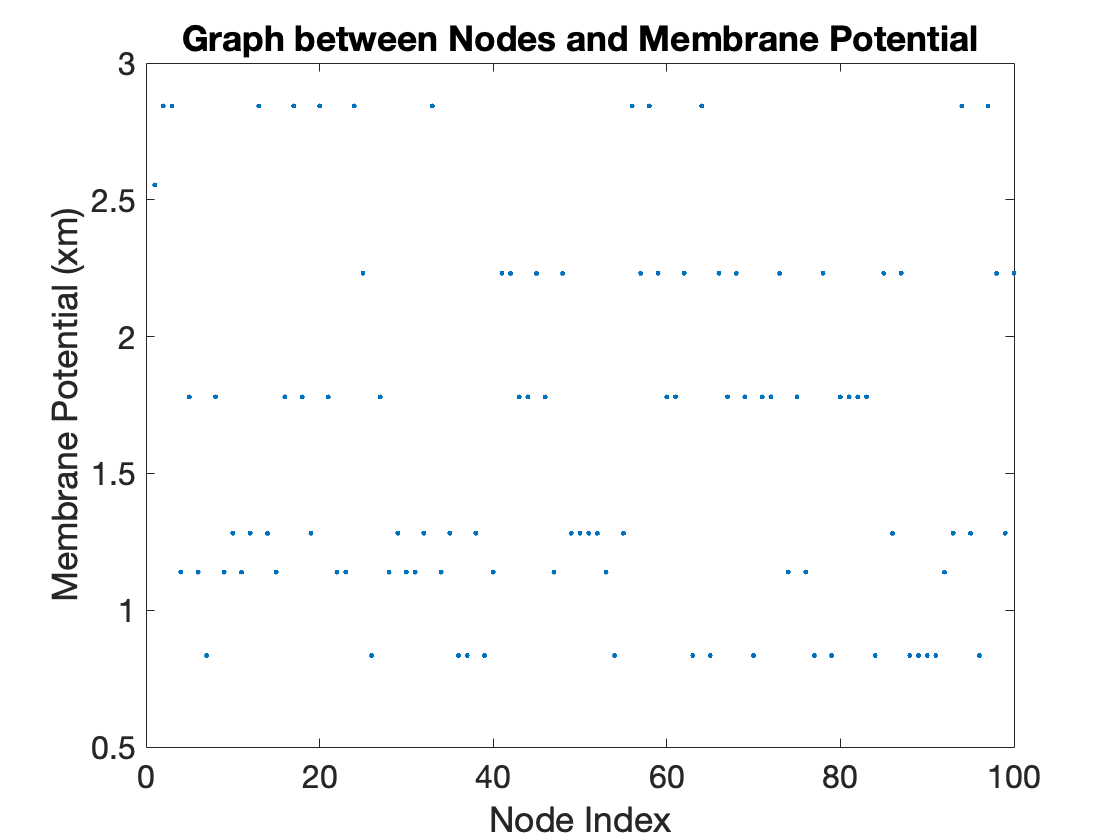}%
        }%
    %\hfill%
    \subfloat[]{%
        \includegraphics[width=0.37\textwidth]{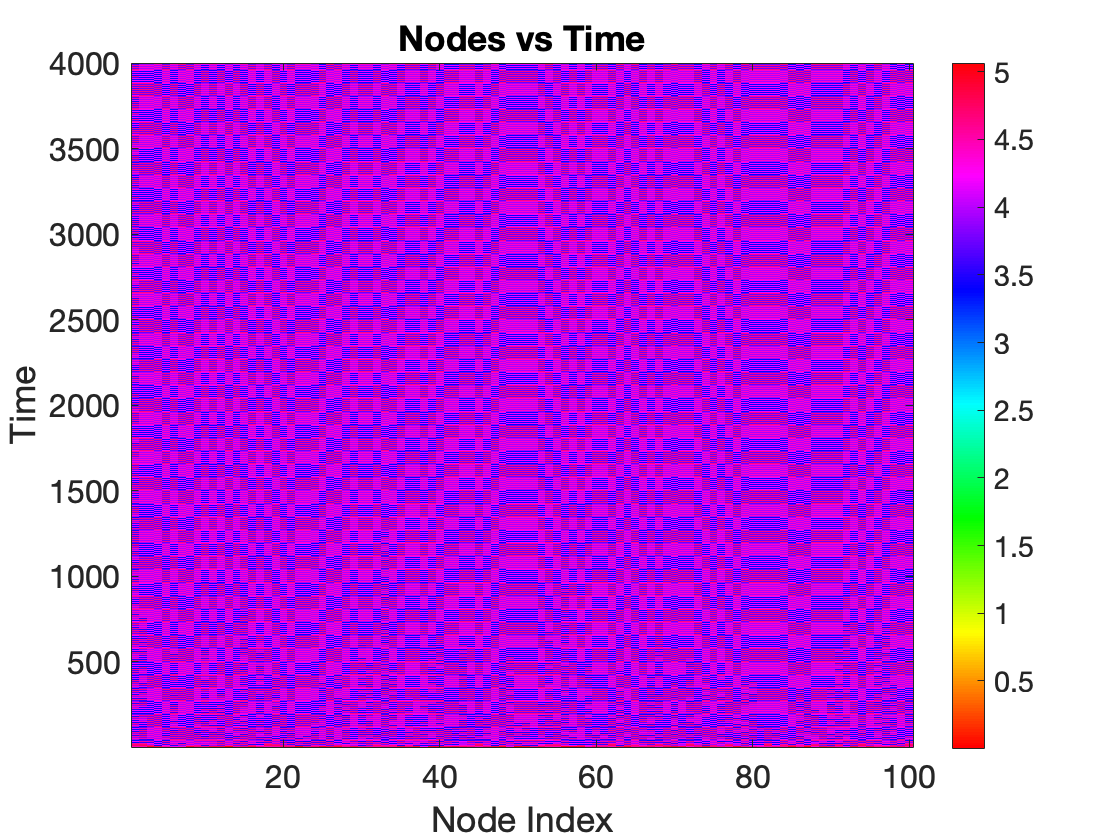}%
        }%
    \caption{Spatiotemporal pattern under the star network of the map $\mathcal{M}_{r,k}(x,\phi)$ by varying control parameter $\mu$. The first row illustrates the unsynchronized behavior of the nodes, while the second row depicts the transition state. The third row illustrates the five clustered states. The fourth row shows a six-clustered synchronized state. Parameters are considered as: $k_0=2, k_1=0.3, k_2,=0.5,$ and $ k=-0.5.$}\label{star}
\end{figure}

We discuss the scenario when there is no connection between the ring network, and all the nodes are connected to each other through the central node (i.e., ring coupling strength $\sigma=0$). Now, $\mu$ is the control parameter in this network, and by varying $\mu$ between $0.0001$ to $0.001$, we witness different spatiotemporal patterns, as shown in Fig.\ref{star}. In the first row of the Fig.\ref{star} (i.e., (a),(b),(c)) for the parameter $\mu=0.001$, we witness the unsynchronized state in the star network of the map $\mathcal{M}_{r,k}(x,\phi)$. In the second row of Fig.\ref{star} (i.e., (d),(e),(f)) for the parameter $\mu=0.0007$, we witness the transition state in the network dynamics.  In the third row of Fig.\ref{star} (i.e., (g),(h),(i)) for the parameter $\mu=0.0003$, the system exhibits the five clustered states in the star network.  Finally, in the fourth row of the Fig.\ref{star} (i.e., (j),(k),(l)) for the parameter $\mu=0.0001$, we observe a high number of clustered states in the network dynamics of the map $\mathcal{M}_{r,k}(x,\phi)$.

\subsection{Multi-chimera state}   
Next, we examine the key spatiotemporal patterns that appear while exploring the network governed by map $\mathcal{M}_{r,k}(x,\phi)$ under different topologies. Until now, chimera states have been observed in both ring and ring-star networks. However, an important question arises: Does this network exhibit more than one chimera state? To answer this, we vary the coupling strengths in ring-star configurations and found that the system exhibits a multi-chimera state. In particular, we identified a chimera pattern consisting of eight coherent groups for the parameter values $\sigma=0.41$ and $\mu=0.0004$. This is noteworthy because most existing studies focused only on single chimera states or their variants \cite{muni2022dynamical,kumar2026study}.

\begin{figure}[ht!] 
    \centering
    \subfloat[$\sigma=0.41$, $\mu=0.0004$]{%
        \includegraphics[width=0.37\textwidth]{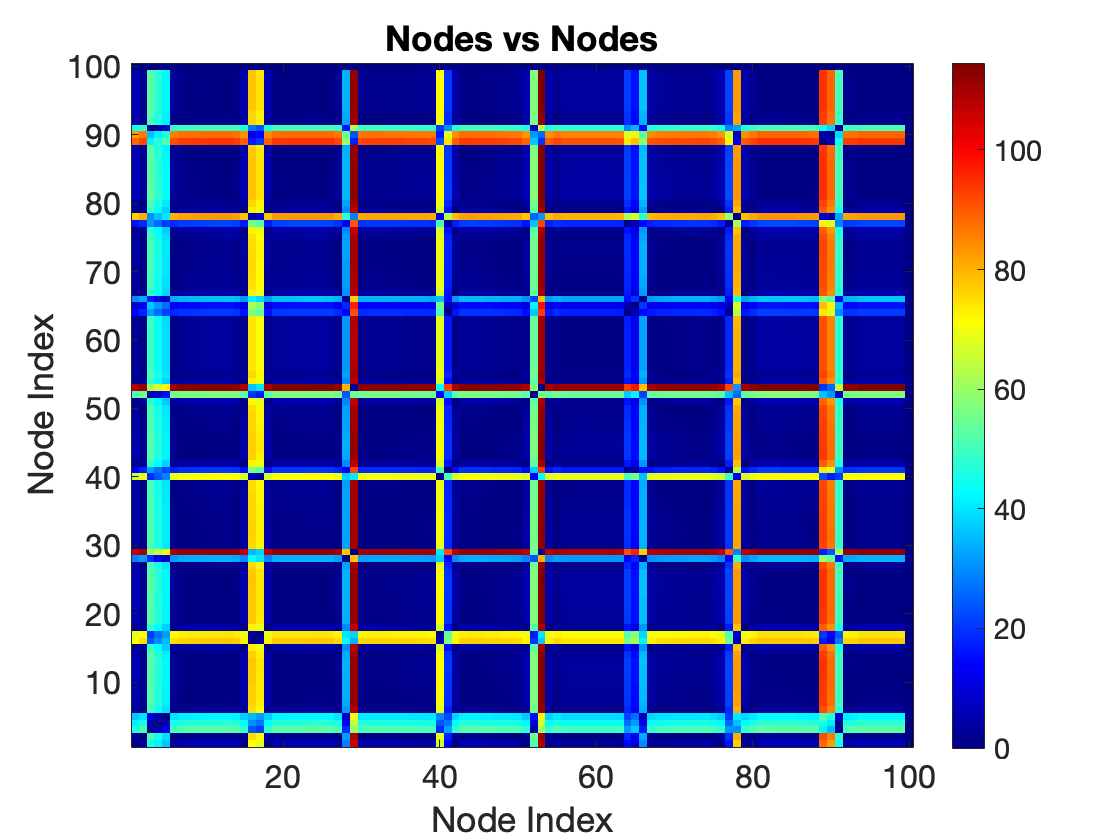}%
        }%
    %\hfill%
    \subfloat[Multi-chimera state]{%
        \includegraphics[width=0.37\textwidth]{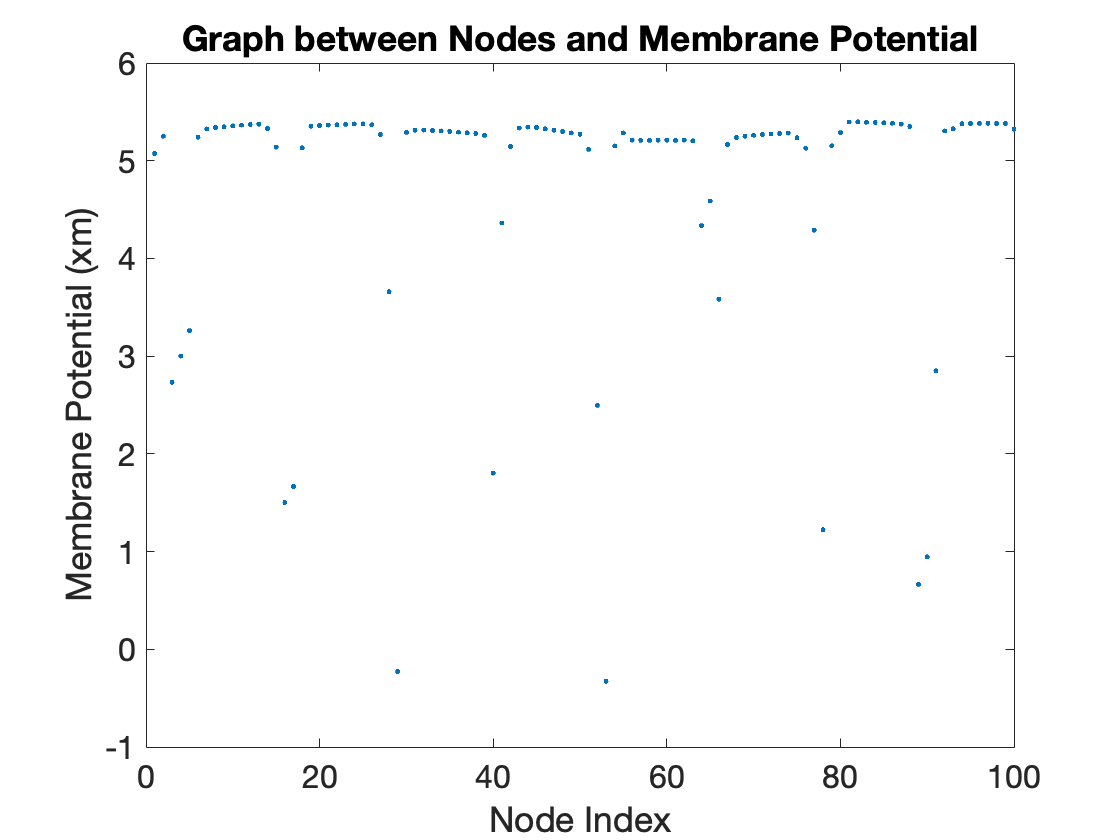}%
        }%
    %\hfill%
    \subfloat[Eight coherent groups]{%
        \includegraphics[width=0.37\textwidth]{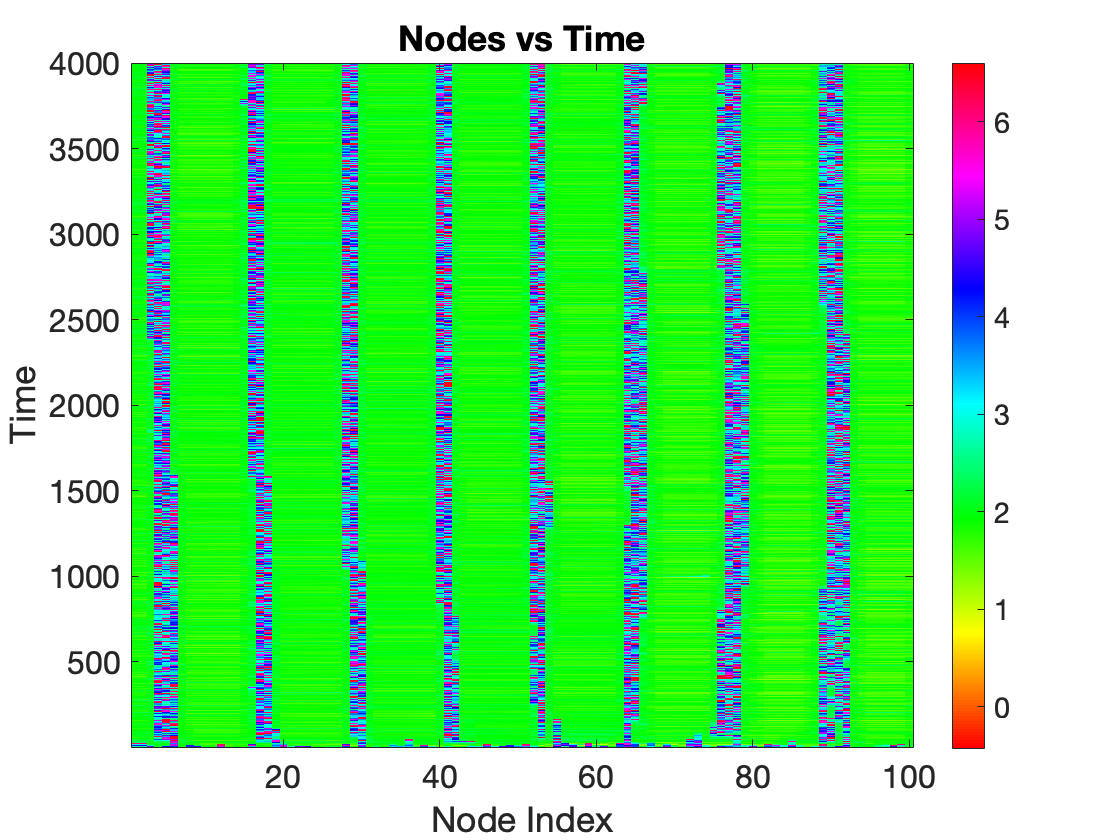}%
        }
\caption{Multi-chimera state in ring-star network with eight coherent groups. Parameters are considered as: $k_0=0.5, k_1=0.1, k_2=0.03, k=0.1, $ and $r=2.15.$}\label{multi_chimera}
\end{figure}

% \vspace{-1.3cm}
\subsection{Other interesting Spatiotemporal patterns}

\begin{figure}[ht!] 
    \centering
    \subfloat[]{%
        \includegraphics[width=0.37\textwidth]{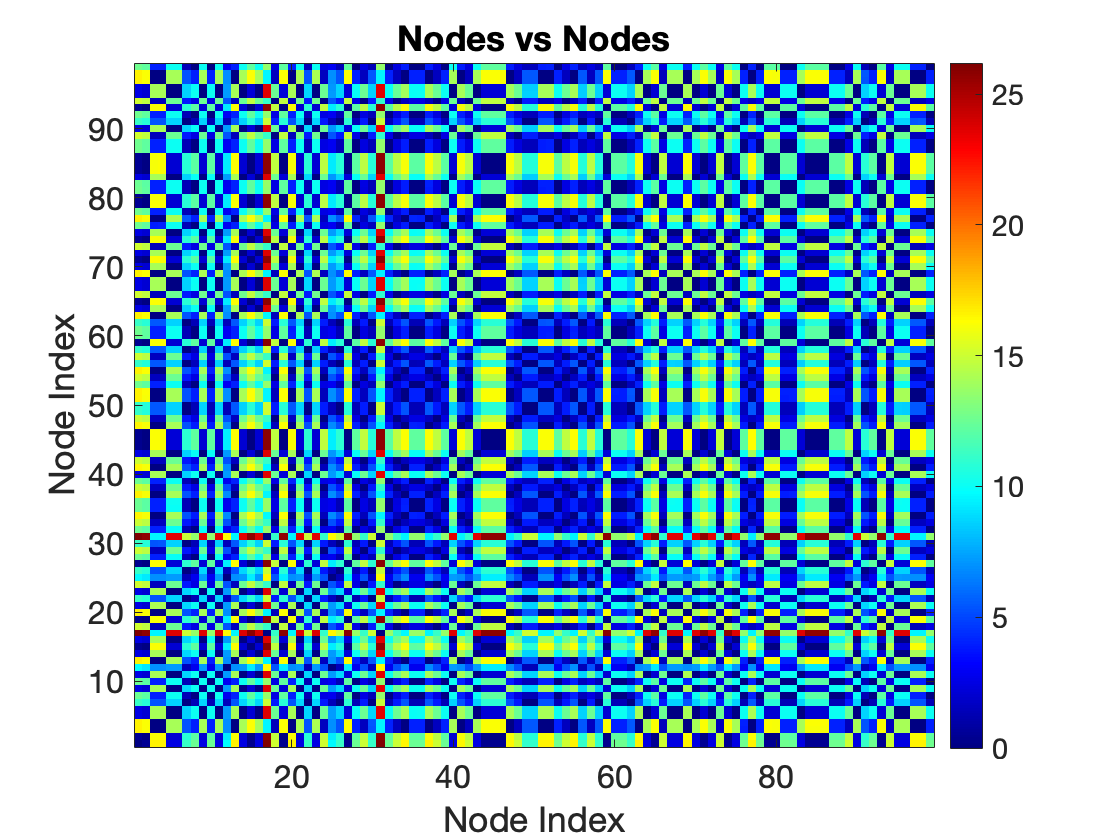}%
        }%
    %\hfill%
    \subfloat[Clustered state]{%
        \includegraphics[width=0.37\textwidth]{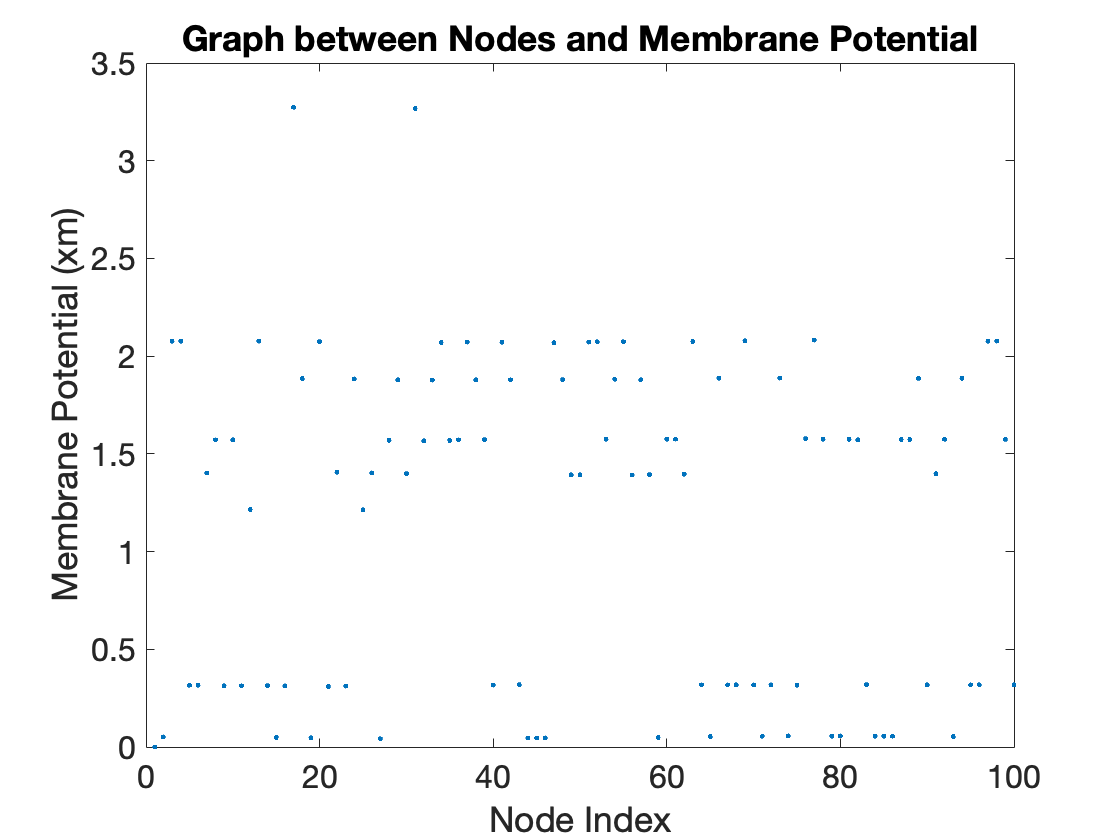}%
        }%
    %\hfill%
    \subfloat[]{%
        \includegraphics[width=0.37\textwidth]{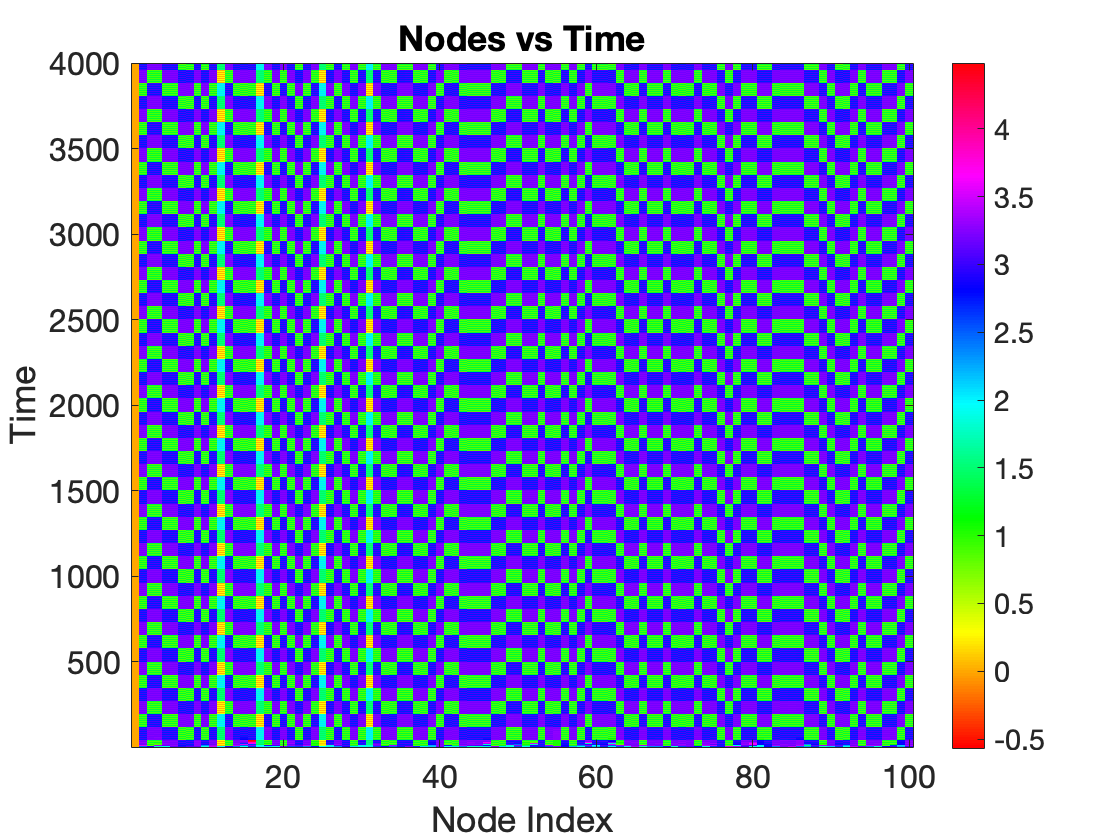}%
        }
        \vfill
        \subfloat[]{%
        \includegraphics[width=0.37\textwidth]{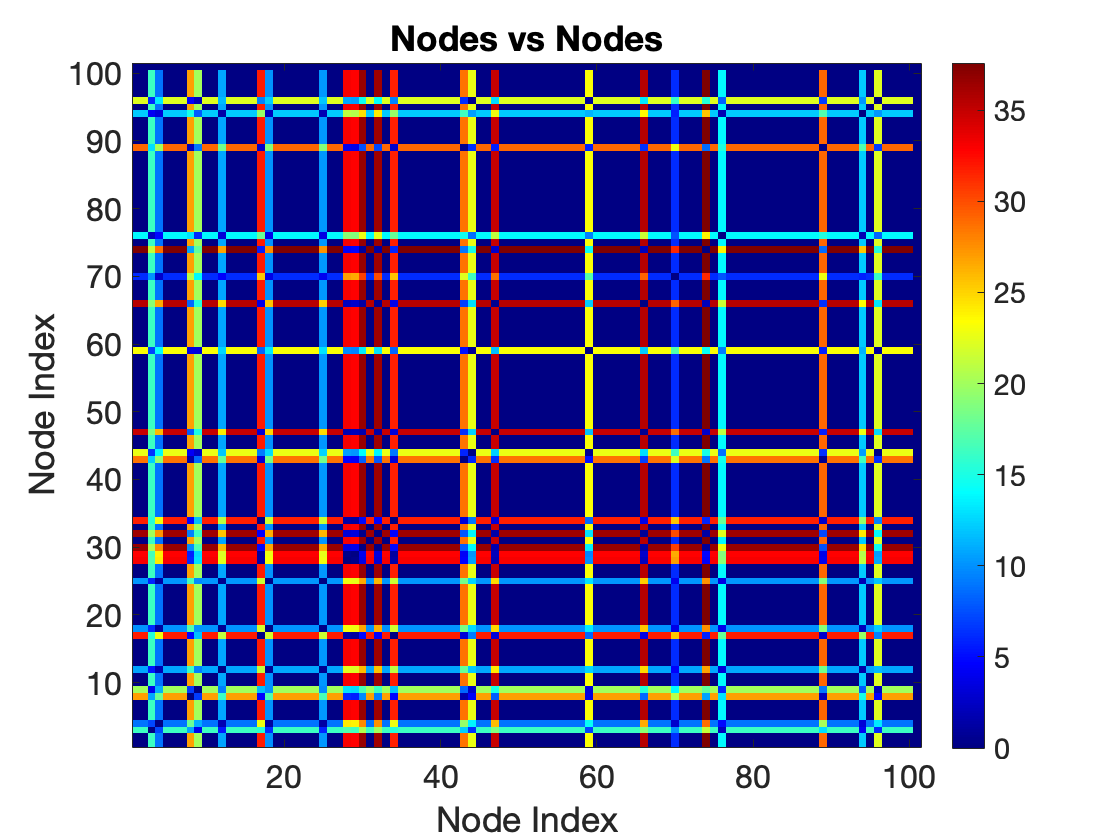}%
        }%
    %\hfill%
    \subfloat[The imperfect synchronized]{%
        \includegraphics[width=0.37\textwidth]{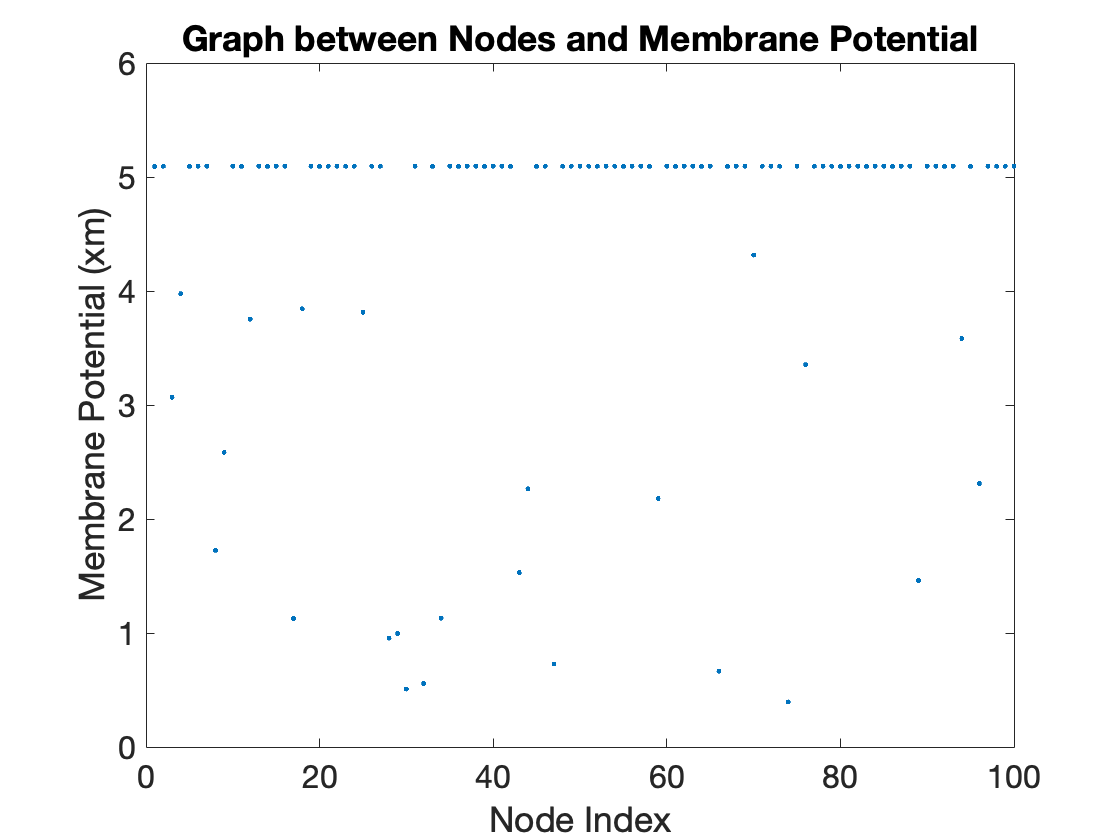}%
        }%
    %\hfill%
    \subfloat[]{%
        \includegraphics[width=0.37\textwidth]{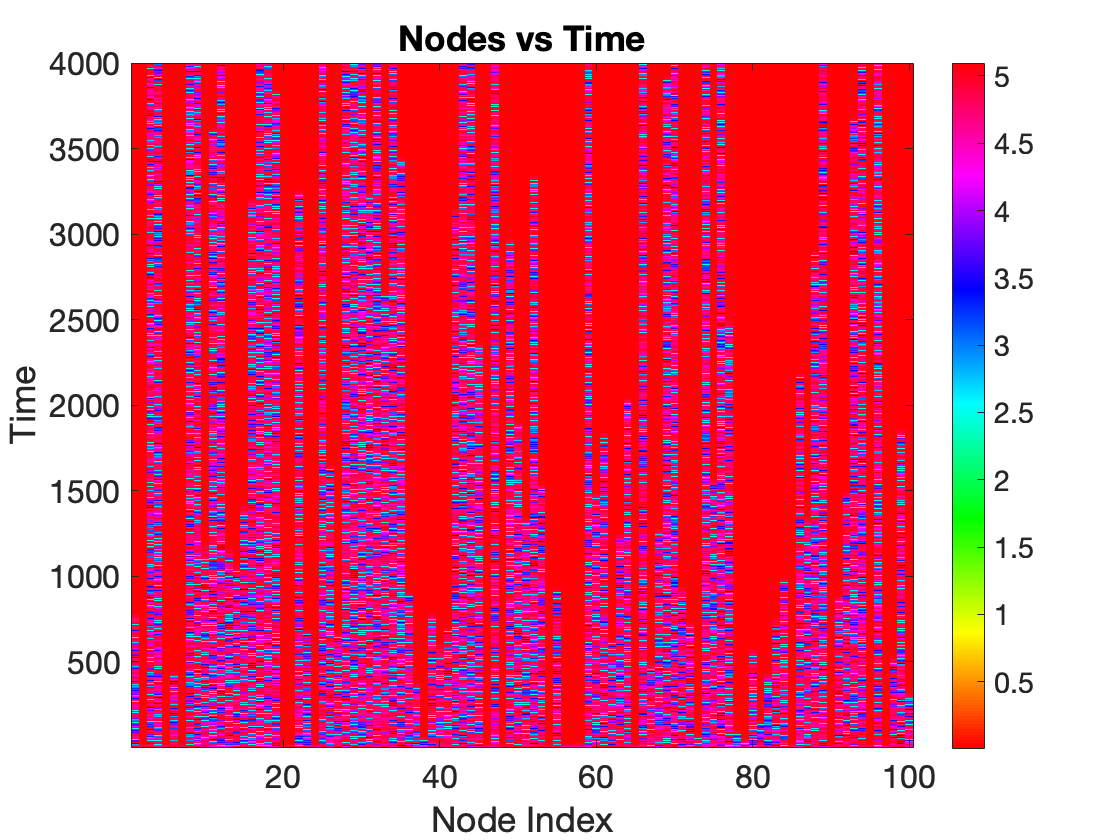}%
        }
\caption{(a),(b),(c) shows the six clustered states in the system under the ring network. (d),(e),(f) shows the imperfect synchronized state in the system under the star network.}\label{other_network}
\end{figure}

Additionally, we highlight some interesting patterns that are also observed while exploring ring and star network behavior with the variation of network coupling parameter, as shown in Fig. \ref{other_network}. First, Fig.\ref{other_network}(a),(b),(c) shows the six clustered states under the ring network for the parameters as $k_0=0.8, k_1=0.3, k_2=0.5, k=-0.7, r=1.3, $ and $ \sigma=0.01$.  Then, in Fig.\ref{other_network}(d) (e),(f) show the imperfect synchronized patterns in the star network for keeping the parameters as $k_0=0.8, k_1=0.3, k_2=0.5, k=-0.5, r=1.243, $ and $ \mu=0.0001$, as seen in a recent study \cite{vivekanandhan2022firing}. Such an intriguing dynamical state demonstrates the dynamical richness of the map $\mathcal{M}_{r,k}(x,\phi)$, highlighting the significance of our study in examining its network dynamical behavior.

%%%%%%%%%%%%-

%----------------------------------------------------------------
% \newpage
\section{Conclusion}

In this study, we introduced a novel two-dimensional discrete neuron model by embedding a cosine-based memristor into a reduced Chialvo map. The characteristic pinched hysteresis loops (PHLs) observed under various conditions confirm that the model is in accordance with the fundamental properties of a memristor. By incorporating this nonlinear memristive component, the model achieved a significant enhancement in dynamical richness while preserving a low-dimensional structure, making it computationally efficient yet behaviorally diverse. Through analytical and numerical analyses, we have uncovered a broad spectrum of dynamical phenomena, including the existence and uniqueness of fixed points, and explored diverse bifurcation structures such as Neimark-Sacker (NS), codimension-one, and codimension-two bifurcations. Additionally, this study confirms the presence of complex behaviors, such as multistability, spiking, and bursting. The presence of chaotic attractors and geometric complexity is further calculated by using the fractal dimension, which confirms the system's intricate nonlinear dynamics.

Beyond the single-neuron framework, we extended our investigation to the network configuration using the ring-star topological structure. This revealed a variety of rich spatiotemporal patterns, including imperfect synchronization, clustered states, traveling waves, and the presence of multichimera states, which demonstrated the model's capacity to replicate key features of a neuron's network. These findings highlight the model's ability to capture diverse neuron interactions, reflective of real neural circuits.

Compared to earlier studies such as \cite{muni2022dynamical,kumar2026study}, our model presents notable advancements. Whereas previous models were limited to identifying only two categories of fixed points, our framework provides a more comprehensive classification that encompasses stable, saddle, and unstable fixed points. Moreover, while prior works relied only on numerical bifurcation studies, the presented work rigorously establishes the conditions under which the map exhibits the NS-bifurcations. Importantly, the proposed model supports more complex network phenomena, including multi-chimera and wavy cluster states, surpassing the simpler chimera dynamics observed in earlier work. Compared to the three-dimensional model in \cite{muni2022dynamical}, this two-dimensional system retains all essential dynamics, including the forward and backward bifurcation patterns, while avoiding their collapse, and moreover, it also confirms the presence of multistability.   

In summary, the proposed memristor-coupled discrete neuron model offers a powerful yet compact framework that captures a wide range of nonlinear behaviors observed in higher-dimensional systems. These advancements establish a solid foundation for future investigations into theoretical neuroscience.  Furthermore, the model opens promising directions for practical applications, including neuromorphic hardware and brain-inspired computations.

% \bibliographystyle{unsrt}  

\bibliographystyle{unsrt}  
\bibliography{ref}

@article{izhikevich2000neural,
  title={Neural excitability, spiking and bursting},
  author={Izhikevich, Eugene M},
  journal={International journal of bifurcation and chaos},
  volume={10},
  number={06},
  pages={1171--1266},
  year={2000},
  publisher={World Scientific}
}

@article{hodgkin1952quantitative,
  title={A quantitative description of membrane current and its application to conduction and excitation in nerve},
  author={Hodgkin, Alan L and Huxley, Andrew F},
  journal={The Journal of physiology},
  volume={117},
  number={4},
  pages={500},
  year={1952}
}

@article{fitzhugh1955mathematical,
  title={Mathematical models of threshold phenomena in the nerve membrane},
  author={FitzHugh, Richard},
  journal={The bulletin of mathematical biophysics},
  volume={17},
  number={4},
  pages={257--278},
  year={1955},
  publisher={Springer}
}

@article{morris1981voltage,
  title={Voltage oscillations in the barnacle giant muscle fiber},
  author={Morris, Catherine and Lecar, Harold},
  journal={Biophysical journal},
  volume={35},
  number={1},
  pages={193--213},
  year={1981},
  publisher={Elsevier}
}

@article{hindmarsh1984model,
  title={A model of neuronal bursting using three coupled first order differential equations},
  author={Hindmarsh, James L and Rose, RM},
  journal={Proceedings of the Royal society of London. Series B. Biological sciences},
  volume={221},
  number={1222},
  pages={87--102},
  year={1984},
  publisher={The Royal Society London}
}

@article{izhikevich2007dynamical,
  author = {Izhikevich, E},
  description = {bib-komplett},
  journal = {MIT Press},
  month = Jul,
  pages = 111,
  title = {Dynamical Systems In Neuroscience},
  uri = {papers://7B65697B-E216-4648-8A41-C67830C0DC73/Paper/p75},
  year = 2007
}

@article{chialvo1995generic,
  title={Generic excitable dynamics on a two-dimensional map},
  author={Chialvo, Dante R},
  journal={Chaos, Solitons \& Fractals},
  volume={5},
  number={3-4},
  pages={461--479},
  year={1995},
  publisher={Elsevier}
}

@article{chua2003memristor,
  title={Memristor-the missing circuit element},
  author={Chua, Leon},
  journal={IEEE Transactions on circuit theory},
  volume={18},
  number={5},
  pages={507--519},
  year={1971},
  publisher={IEEE}
}

@article{cao2024discrete,
  title={A discrete {C}hialvo--{R}ulkov neuron network coupled with a novel memristor model: Design, {D}ynamical analysis, {DSP} implementation and its application},
  author={Cao, Hongli and Wang, Yu and Banerjee, Santo and Cao, Yinghong and Mou, Jun},
  journal={Chaos, Solitons \& Fractals},
  volume={179},
  pages={114466},
  year={2024},
  publisher={Elsevier}
}

@article{wang2025spatiotemporal,
  title={Spatiotemporal dynamics and synchronization in a memristive {C}hialvo neural network},
  author={Wang, Huihai and Chen, Hanqi and Sun, Kehui and Zhu, Wanting and Yao, Zhao},
  journal={Nonlinear Dynamics},
  volume={113},
  number={9},
  pages={10365--10377},
  year={2025},
  publisher={Springer}
}

@article{bao2021discrete,
  title={Discrete memristor hyperchaotic maps},
  author={Bao, Han and Hua, Zhongyun and Li, Houzhen and Chen, Mo and Bao, Bocheng},
  journal={IEEE Transactions on Circuits and Systems I: Regular Papers},
  volume={68},
  number={11},
  pages={4534--4544},
  year={2021},
  publisher={IEEE}
}

@article{peng2020discrete,
  title={A discrete memristor model and its application in {H}{\'e}non map},
  author={Peng, Yuexi and Sun, Kehui and He, Shaobo},
  journal={Chaos, Solitons \& Fractals},
  volume={137},
  pages={109873},
  year={2020},
  publisher={Elsevier}
}

@article{yu2022dynamic,
  title={Dynamic analysis and application in medical digital image watermarking of a new multi-scroll neural network with quartic nonlinear memristor},
  author={Yu, Fei and Chen, Huifeng and Kong, Xinxin and Yu, Qiulin and Cai, Shuo and Huang, Yuanyuan and Du, Sichun},
  journal={The European Physical Journal Plus},
  volume={137},
  number={4},
  pages={434},
  year={2022},
  publisher={Springer}
}

@article{chen2023new,
  title={A new mix chaotic circuit based on memristor--memcapacitor},
  author={Chen, Yixin and Mou, Jun and Jahanshahi, Hadi and Wang, Zhisen and Cao, Yinghong},
  journal={The European Physical Journal Plus},
  volume={138},
  number={1},
  pages={78},
  year={2023},
  publisher={Springer Berlin Heidelberg}
}

@article{xu2024dynamical,
  title={Dynamical effects of memristive electromagnetic induction on a {2D} {W}ilson neuron model},
  author={Xu, Quan and Wang, Kai and Shan, Yufan and Wu, Huagan and Chen, Mo and Wang, Ning},
  journal={Cognitive Neurodynamics},
  volume={18},
  number={2},
  pages={645--657},
  year={2024},
  publisher={Springer}
}

@article{wang2025complex,
  title={Complex synchronization in memristor-coupled {C}hialvo {N}eurons},
  author={Wang, Lili and Gao, Hao and Li, Chunbiao and Tang, Qianyuan and Chialvo, Dante},
  journal={The European Physical Journal Special Topics},
  pages={1--15},
  year={2025},
  publisher={Springer}
}

@article{xiu2020new,
  title={New chaotic memristive cellular neural network and its application in secure communication system},
  author={Xiu, Chunbo and Zhou, Ruxia and Liu, Yuxia},
  journal={Chaos, Solitons \& Fractals},
  volume={141},
  pages={110316},
  year={2020},
  publisher={Elsevier}
}

@article{yu2021dynamics,
  title={Dynamics analysis, hardware implementation and engineering applications of novel multi-style attractors in a neural network under electromagnetic radiation},
  author={Yu, Fei and Shen, Hui and Zhang, Zinan and Huang, Yuanyuan and Cai, Shuo and Du, Sichun},
  journal={Chaos, Solitons \& Fractals},
  volume={152},
  pages={111350},
  year={2021},
  publisher={Elsevier}
}

@article{liu2022memcapacitor,
  title={Memcapacitor-coupled {C}hebyshev hyperchaotic map},
  author={Liu, Xingce and Mou, Jun and Yan, Huizhen and Bi, Xiuguo},
  journal={International Journal of Bifurcation and Chaos},
  volume={32},
  number={12},
  pages={2250180},
  year={2022},
  publisher={World Scientific}
}

@article{muni2022dynamical,
  title={Dynamical effects of electromagnetic flux on chialvo neuron map: nodal and network behaviors},
  author={Muni, Sishu Shankar and Fatoyinbo, Hammed Olawale and Ghosh, Indranil},
  journal={International Journal of Bifurcation and Chaos},
  volume={32},
  number={09},
  pages={2230020},
  year={2022},
  publisher={World Scientific}
}

@article{kumar2026study,
  title={Study of reduced {C}hialvo map with electromagnetic flux: Dynamics and network behavior},
  author={Kumar, Ajay and Chandramouli, VVMS},
  journal={Applied Mathematics and Computation},
  volume={509},
  pages={129650},
  year={2026},
  publisher={Elsevier}
}

@article{wang2024multi,
  title={Multi-chimera states in a higher order network of {F}itz{H}ugh--{N}agumo oscillators},
  author={Wang, Zhen and Chen, Mingshu and Xi, Xiaojian and Tian, Huaigu and Yang, Rui},
  journal={The European Physical Journal Special Topics},
  volume={233},
  number={4},
  pages={779--786},
  year={2024},
  publisher={Springer}
}

@article{adhikari2019three,
  title={Three Fingerprints of Memristor},
  author={Shyam Prasad Adhikari and Maheshwar Prasad Sah and Hyongsuk Kim and Leon Ong Chua},
  journal={IEEE Transactions on Circuits and Systems I: Regular Papers},
  year={2013},
  volume={60},
  pages={3008-3021},
  url={https://api.semanticscholar.org/CorpusID:12665998}
}

@article{din2017global,
  title={Global stability and {N}eimark-{S}acker bifurcation of a host-parasitoid model},
  author={Din, Qamar},
  journal={International Journal of Systems Science},
  volume={48},
  number={6},
  pages={1194--1202},
  year={2017},
  publisher={Taylor \& Francis}
}

@article{chandramouli2025dynamics,
  title={Dynamics of {2D} delayed {H}omographic {R}icker map},
  author={Chandramouli, VVMS and others},
  journal={International Journal of Dynamics and Control},
  volume={13},
  number={2},
  pages={52},
  year={2025},
  publisher={Springer Nature BV}
}

@article{kocarev1993experimental,
author = {KOCAREV, Lj. and HALLE, K. S. and ECKERT, K. and CHUA, L. O.},
title = {EXPERIMENTAL OBSERVATION OF ANTIMONOTONICITY IN CHUA'S CIRCUIT},
journal = {International Journal of Bifurcation and Chaos},
volume = {03},
number = {04},
pages = {1051-1055},
year = {1993},
doi = {10.1142/S0218127493000878},
URL = { 
https://doi.org/10.1142/S0218127493000878
},
eprint = {     
https://doi.org/10.1142/S0218127493000878
}
}

@article{dawson1992antimonotonicity,
  title={Antimonotonicity: inevitable reversals of period-doubling cascades},
  author={Dawson, Silvina P and Grebogi, Celso and Yorke, James A and Kan, Ittai and Ko{\c{c}}ak, H{\"u}seyin},
  journal={Physics Letters A},
  volume={162},
  number={3},
  pages={249--254},
  year={1992},
  publisher={Elsevier}
}

@book{kuznet2019numerical,
     TITLE = {Numerical bifurcation analysis of maps},
     AUTHOR = {Kuznetsov, Yuri A. and Meijer, Hil G. E.},
    SERIES = {Cambridge Monographs on Applied and Computational Mathematics},
    VOLUME = {34},
     YEAR = {2019},
      NOTE = {From theory to software},
 PUBLISHER = {Cambridge University Press},
 address   = {Cambridge},
     PAGES = {xiv+407},
      ISBN = {978-1-108-49967-5},
   MRCLASS = {37Gxx (65P30)},
  MRNUMBER = {3930602},
       DOI = {10.1017/9781108585804},
       URL = {https://doi.org/10.1017/9781108585804}
}

@article{depannemaecker2026next,
  title={A next generation neural mass model with neuromodulation},
  author={Depannemaecker, Damien and Duprat, Chlo{\'e} and Casagrande, Gabriele and Saggio, Marisa and Athanasiadis, Anastasios Polykarpos and Angiolelli, Marianna and Sales Carbonell, Carola and Wang, Huifang and Petkoski, Spase and Sorrentino, Pierpaolo and others},
  journal={Journal of Computational Neuroscience},
  pages={1--21},
  year={2026},
  publisher={Springer}
}

@article{muni2022discrete,
  title={Discrete hybrid Izhikevich neuron model: Nodal and network behaviours considering electromagnetic flux coupling},
  author={Muni, Sishu Shankar and Rajagopal, Karthikeyan and Karthikeyan, Anitha and Arun, Sundaram},
  journal={Chaos, Solitons \& Fractals},
  volume={155},
  pages={111759},
  year={2022},
  publisher={Elsevier}
}

@article{radushev2026metric,
  title={Metric framework of coherent activity patterns identification in spiking neuronal networks},
  author={Radushev, Daniil and Dogonasheva, Olesia and Gutkin, Boris and Zakharov, Denis},
  journal={Chaos, Solitons \& Fractals},
  volume={203},
  pages={117645},
  year={2026},
  publisher={Elsevier}
}

@article{ashwin2025global,
  title={Global bifurcations organizing weak chimeras in three symmetrically coupled {K}uramoto oscillators with inertia},
  author={Ashwin, Peter and Bick, Christian},
  journal={Journal of Nonlinear Science},
  volume={35},
  number={2},
  pages={45},
  year={2025},
  publisher={Springer}
}

@article{vivekanandhan2022firing,
  title={Firing patterns of {I}zhikevich neuron model under electric field and its synchronization patterns},
  author={Vivekanandhan, Gayathri and Hamarash, Ibrahim Ismael and Ali Ali, Ahmed M and He, Shaobo and Sun, Kehui},
  journal={The European Physical Journal Special Topics},
  volume={231},
  number={22},
  pages={4017--4023},
  year={2022},
  publisher={Springer}
}

% George Kour and Raid Saabne.
% \newblock Real-time segmentation of on-line handwritten arabic script.
% \newblock In {\em Frontiers in Handwriting Recognition (ICFHR), 2014 14th
%   International Conference on}, pages 417--422. IEEE, 2014.

% George Kour and Raid Saabne.
% \newblock Fast classification of handwritten on-line arabic characters.
% \newblock In {\em Soft Computing and Pattern Recognition (SoCPaR), 2014 6th
%   International Conference of}, pages 312--318. IEEE, 2014.

% \bibitem{hadash2018estimate}
% Guy Hadash, Einat Kermany, Boaz Carmeli, Ofer Lavi, George Kour, and Alon
% \newblock Estimate and replace: A novel approach to integrating deep neural
%   networks with existing applications.
% \newblock {\em arXiv preprint arXiv:1804.09028}, 2018.

% \end{thebibliography}

\section{Appendix}
% \vspace{-0.3cm}
\begin{figure}[ht!] 
    \centering
    \subfloat[$r=2$]{%
        \includegraphics[width=0.5\textwidth,height=0.23\textwidth]{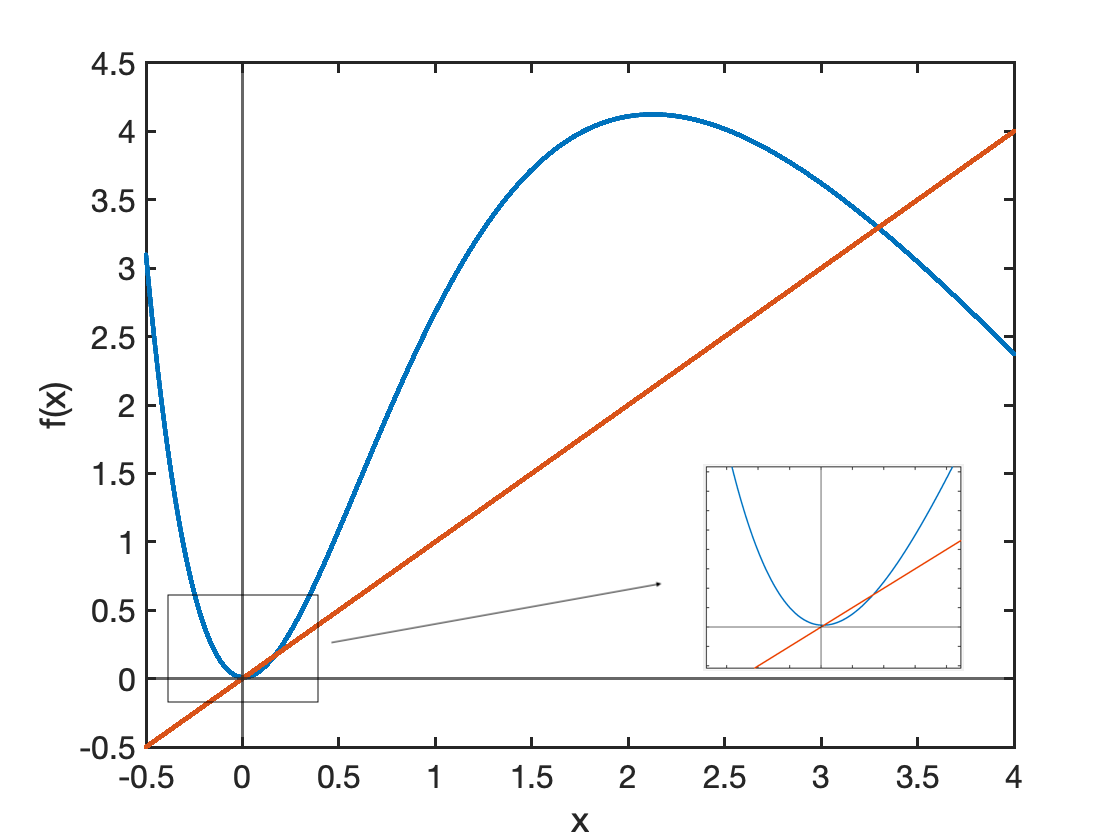}%
        }%
    % \hfill%
    \subfloat[$r=2.8$]{\includegraphics[width=0.5\textwidth,height=0.23\textwidth]{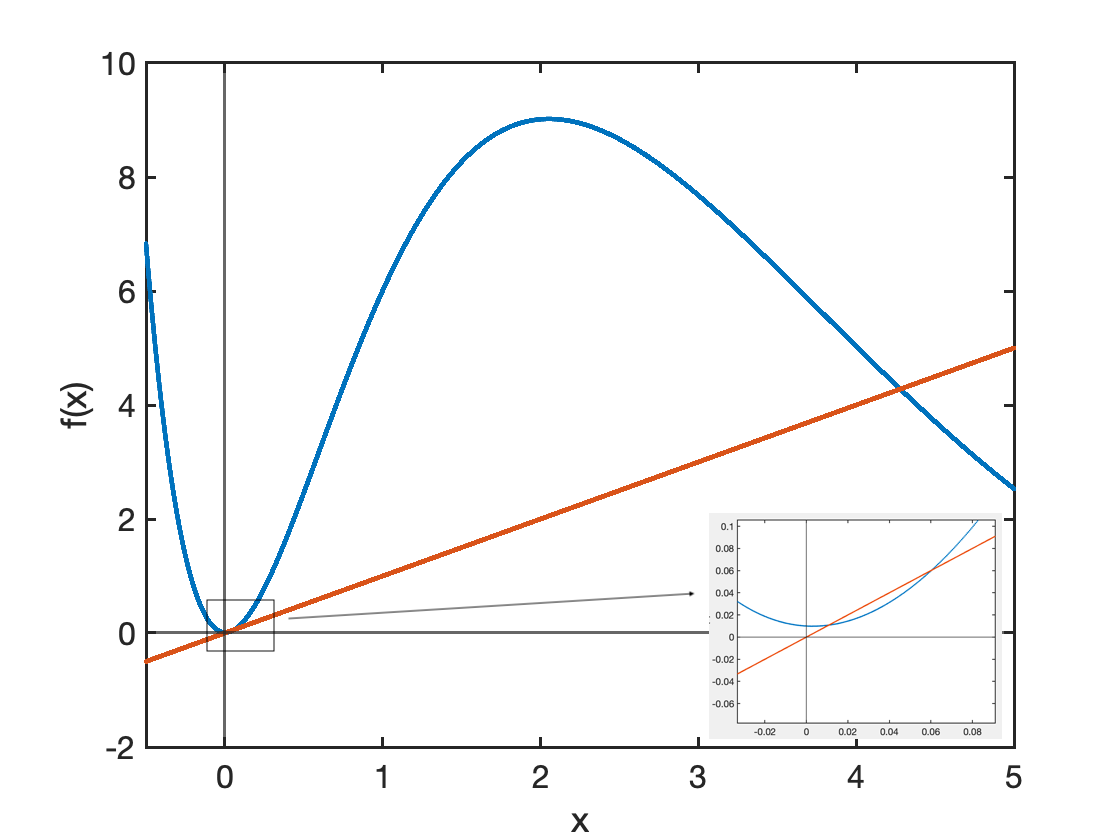}%
        }
        \vfill%
     \subfloat[$r=3$]{\includegraphics[width=0.5\textwidth,height=0.23\textwidth]{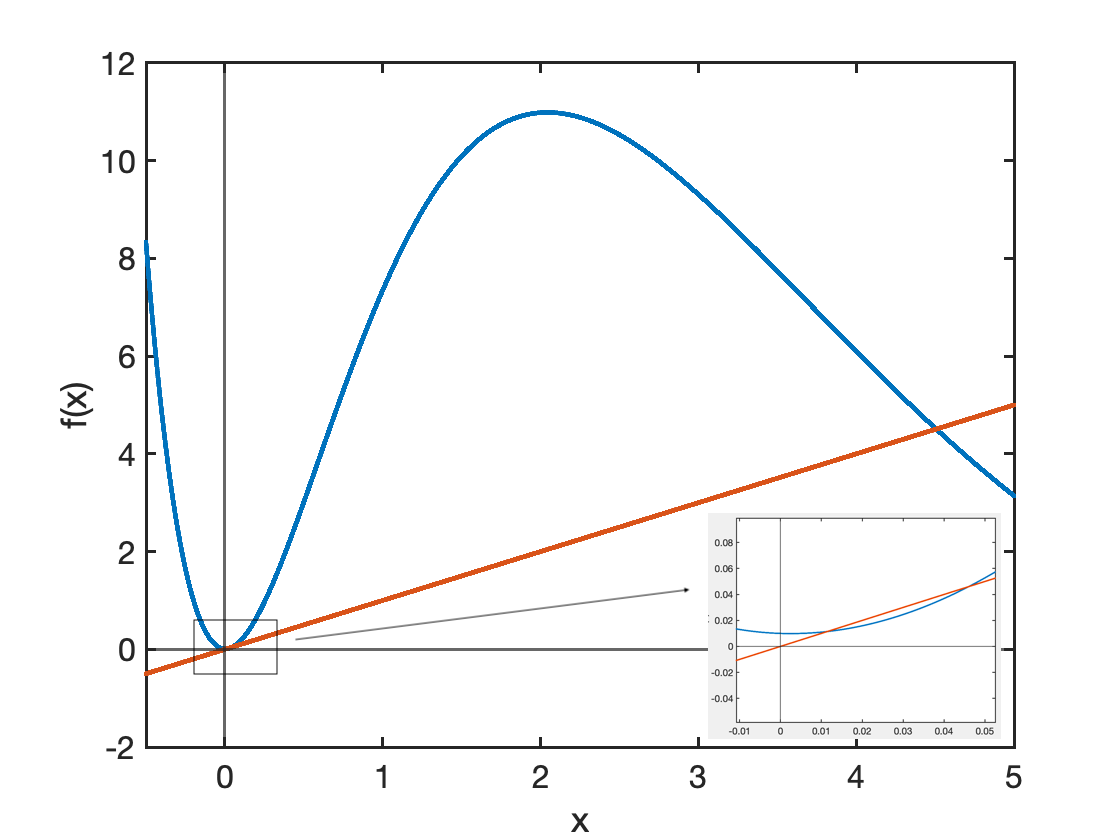}%
        }%
    % \hfill%
    \subfloat[$r=5$]{\includegraphics[width=0.5\textwidth,height=0.23\textwidth]{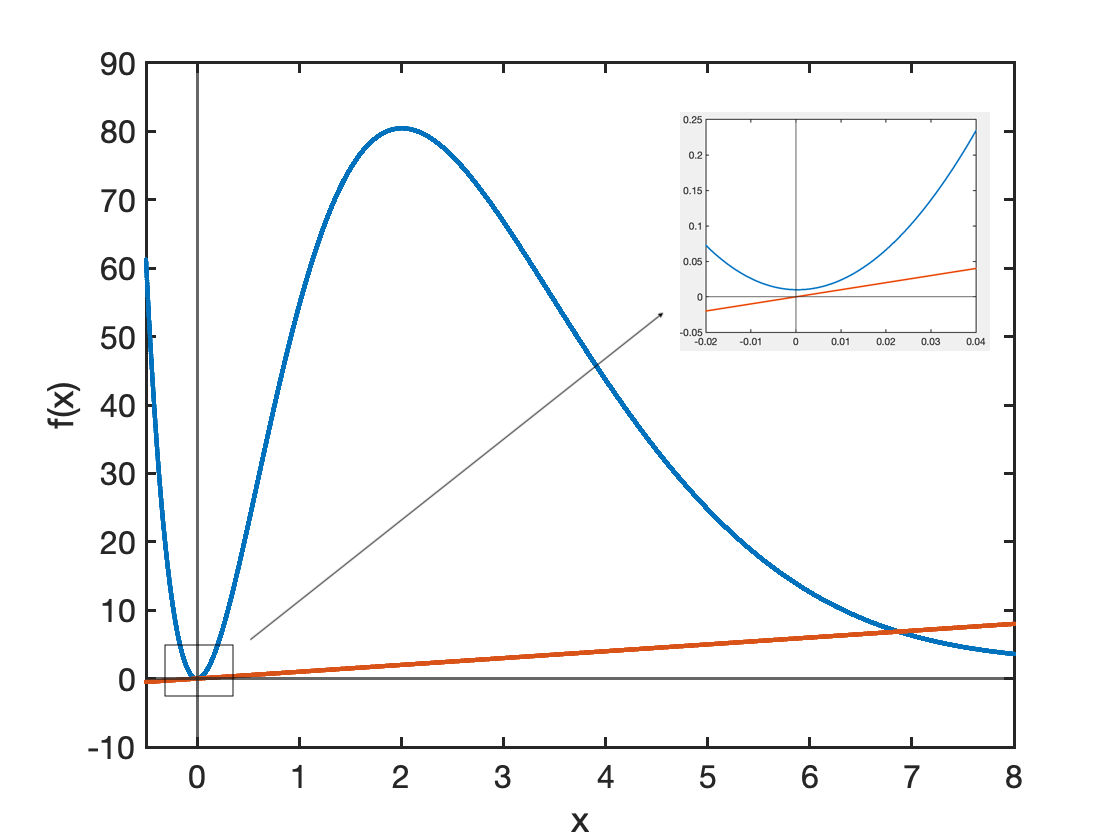}%
        }%
    \caption{ Fixed points of the map $\mathcal{M}_{r,k}(x,\phi)$ corresponding to specific parameter values $r$. Parameters are considered as: $ k_0=0.01, a_1=0.5, a_2=0.5, $ and $ k=-0.1$.}\label{Fig4}
\end{figure}

\end{document}